\documentclass[reqno,a4paper,11pt,final]{amsart}
\usepackage{amssymb,amsmath,amsthm}
\usepackage{ccaption}
\usepackage{enumerate}
\usepackage[margin=30mm]{geometry}
\usepackage{here}
\usepackage{multirow}
\usepackage[all]{xy}

\usepackage{mathtools}
\mathtoolsset{showonlyrefs=true}

\usepackage[svgnames]{xcolor}
\usepackage[colorlinks,citecolor=DarkGreen,linkcolor=FireBrick,linktocpage,unicode]{hyperref}

\usepackage[]{graphicx}

\usepackage[utf8]{inputenc}
\usepackage[T1]{fontenc}
\usepackage{lmodern}

\newtheorem{thm}{Theorem}[section]
\newtheorem{pro}[thm]{Proposition}
\newtheorem{lem}[thm]{Lemma}
\newtheorem{cor}[thm]{Corollary}

\theoremstyle{definition}
\newtheorem{dfn}[thm]{Definition}
\newtheorem{ex}[thm]{Example}
\newtheorem{rem}[thm]{Remark}
\newtheorem{nota}{Notation}

\newcommand{\INN}[2]{\langle #1, #2 \rangle}
\newcommand{\RANK}{\operatorname{rank}}
\newcommand{\MID}{\;\Bigg|\;}

\makeatletter
	
	\@addtoreset{equation}{section}
\makeatother

\title[Compact symmetric triads and symmetric triads with multiplicities]{Compact symmetric triads and Symmetric triads with multiplicities}
\author[K.~Baba]{Kurando Baba}
\author[O.~Ikawa]{Osamu Ikawa}


\address[K.~Baba]{Faculty of Science and Technology,
Tokyo University of Science,
2641 Yamazaki, Noda, Chiba, 278-8510, Japan. }
\email{kurando.baba@rs.tus.ac.jp}

\address[O.~Ikawa]{Faculty of Arts and Sciences, Kyoto Institute of Technology,
Matsugasaki, Sakyoku, Kyoto 606-8585, Japan. }
\email{ikawa@kit.ac.jp}

\thanks{The first author was partially supported by JSPS KAKENHI Grant Number JP25K07017.
The second author was partially supported by JSPS KAKENHI Grant Number JP22K03285.}

\date{}

\begin{document}

\maketitle

\begin{abstract}
In this paper,
we develop the theory of symmetric
triads with multiplicities.
First, we classify abstract symmetric triads
with multiplicities.
Second, we determine the
symmetric triads with multiplicities
corresponding to commutative compact symmetric triads.
As applications,
we give the classifications
for commutative compact symmetric triads,
which consist of two types depending on the choice of the equivalence relations.
\end{abstract}

\setcounter{tocdepth}{1}
\tableofcontents

\section{Introduction}

A \emph{compact symmetric triad}
is a triplet $(G,\theta_{1},\theta_{2})$
which consists of a compact connected semisimple Lie group $G$, and two involutions $\theta_{1}$ and $\theta_{2}$ on it.
We denote by $K_{i}$ ($i=1,2$)
the identity component
of the fixed-point subgroup of $\theta_{i}$
in $G$.
Then the homogeneous space
$G/K_{i}$ becomes a compact Riemannian symmetric space.
The study of compact symmetric triads
is motivated in the context
of the geometry of Hermann actions.
The isometric action of $K_{2}$
on $M_{1}=G/K_{1}$
is called a \textit{Hermann action}.
In the case when $\theta_{1}$ and $\theta_{2}$ are conjugate to
each other under an inner automorphism of $G$,
which we write $\theta_{1}\sim\theta_{2}$,
it is shown that
the corresponding Hermann action is
isomorphic to the isotropy action of $K_{1}$
on $G/K_{1}$.
In particular,
in the case when
$\theta_{1}=\theta_{2}$,
we have $K_{1}=K_{2}$,
so that the Hermann action
is nothing but the isotropy action.
The isotropy actions on compact Riemannian symmetric spaces have been extensively studied by many geometers. Hence, in the study of Hermann actions, it is essential to consider the case when $\theta_{1}\not\sim\theta_{2}$,
that is, $\theta_{1}$ and $\theta_{2}$ are not conjugate under inner automorphisms of $G$.
It is known that Hermann actions
have a geometrically good
property,
the so-called hyperpolarity
(\cite{HPTT}).
In general,
an isometric action of a compact connected Lie group on a connected complete Riemannian manifold $M$ is said to be \textit{hyperpolar} if there exists a connected complete flat submanifold $S\subset M$ such that
$S$ meets all orbits orthogonally.
Such a submanifold
is called a \textit{section}
for the hyperpolar action.
It is shown that the sections are totally geodesic submanifolds of $M$.
Kollross (\cite{Kollross}) classified hyperpolar actions on compact Riemannian symmetric spaces.
According to his classification,
most hyperpolar actions on compact Riemannian symmetric spaces
are exhausted by the Hermann actions.

The orbits of Hermann actions
give important examples
of homogeneous submanifolds
in Riemannian symmetric spaces.
In order to study Hermann actions,
the second author (\cite{Ikawa})
introduced the notion of symmetric triads
with multiplicities
as an extension of restricted root systems with multiplicities.
A compact symmetric triad $(G,\theta_{1},\theta_{2})$
is said to be \emph{commutative}, if $\theta_{1}\theta_{2}=\theta_{2}\theta_{1}$ holds.
Then
the symmetric triad with multiplicities,
which we write $(\tilde{\Sigma},\Sigma,W;m,n)$,
is constructed from a commutative compact symmetric triad
$(G,\theta_{1},\theta_{2})$.
Then we obtain a uniform method
to study the orbit space and properties of each orbit
for the corresponding Hermann action by means of $(\tilde{\Sigma},\Sigma,W;m,n)$.
In his paper \cite{Ikawa},
the dimension of each orbit was determined.
Furthermore, this method provides us descriptions
of the mean curvature vector fields
and the principal curvatures of orbits.
Then he derived the conditions
for the minimality,
the austerity in the sense of Harvey-Lawson
(\cite{HL}) and the total geodesicity,
respectively,
and classified totally geodesic orbits and austere orbits.
After him,
we find other studies
of orbits for Hermann actions
in Ohno (\cite{Ohno}),
Ohnita (\cite{Ohnita})
and Ohno-Sakai-Urakawa (\cite{OSU})
by means of symmetric triads with multiplicities.
Ikawa-Tanaka-Tasaki (\cite{ITT}) applied symmetric triads
to study the intersection of two real forms in a Hermitian symmetric space
of compact type.
Thus, the study of symmetric triads with multiplicities
has led to significant advances in the theory of symmetric spaces and related fields.

For further understanding of Hermann actions,
we will develop the theory
of symmetric triads with multiplicities.
Our concern is to determine $(\tilde{\Sigma},\Sigma,W;m,n)$ for commutative compact symmetric triads.
This determination yields the classifications for commutative compact symmetric triads,
which is the main subject of this paper.

In this paper,
we first give the classification of abstract
symmetric triads with multiplicities.
Let us consider
an equivalence relation
on the set of symmetric triads
with multiplicities as in Definition
\ref{dfn:e-relation},
which we write $\sim$.
This equivalence relation
comes from that
on the set of compact symmetric triads
due to Matsuki (\cite{Matsuki}).
We write
Matsuki's equivalence relation
as the same symbol $\sim$.
It is shown that,
for two isomorphic compact symmetric triads,
the corresponding Hermann actions
are essentially the same.
The classification
of abstract symmetric triads with multiplicities
is given by a case-by-case argument
based on the classification of abstract symmetric triads
\textit{without} multiplicities
due to the second author \cite[Theorem 2.19]{Ikawa}.
Our classification is summarized
in Theorem \ref{thm:class_symmI-III}.

Second,
we determine the symmetric triads
with multiplicities
corresponding to commutative compact symmetric triads.
Originally,
Matsuki (\cite{Matsuki})
classified
compact symmetric triads with respect to $\sim$.
It should be noted that 
the commutativity of $\theta_{1}$
and $\theta_{2}$ in $(G,\theta_{1},\theta_{2})$
is not preserved by
$\sim$,
and there does not necessarily exist a representative
with $\theta_{1}\theta_{2}=\theta_{2}\theta_{1}$
in its isomorphic class.
In the case when $G$ is simple,
it has been classified the isomorphism classes
of compact symmetric triads containing commutative representatives
(\cite{Conlon}, \cite{Matsuki}).
Recently,
the authors (\cite{BI}) gave an alternative method
to classify compact symmetric triads
in terms of the notion of double Satake diagrams.
Roughly speaking,
the double Satake diagram
of a compact symmetric triad
$(G,\theta_{1},\theta_{2})$
is expressed by the pair
of the Satake diagrams of $(G,\theta_{1})$
and $(G,\theta_{2})$.
Based on our classification,
we determined
the isomorphism classes
of compact symmetric triads
containing commutative representatives
(\cite[Remark 6.13]{BI}).
Then we can give a method
to determine the symmetric triad
with multiplicities corresponding to
commutative compact symmetric triads
by using double Satake diagrams.
Our determination is summarized
in Table \ref{table:symmtriext}.
From this table,
we can find that two commutative compact
symmetric triads are locally isomorphic to each other
if and only if their symmetric triads
with multiplicities
are isomorphic with respect to $\sim$
(Corollary \ref{cor:simsim_simplyconnected}).

Third,
we will apply the above results to classify
commutative compact symmetric triads
with respect to another equivalence
relation as in Definition \ref{dfn:CST_equiv},
which we write $\equiv$.
This equivalence relation
is finer than Matsuki's one $\sim$,
and preserves the commutativity of compact symmetric triads.
We will classify commutative compact symmetric triads
with respect to $\equiv$.
Then we note that their classification is carried out
not only in the case when $\theta_{1}\not\sim\theta_{2}$,
but also in the case when $\theta_{1}\sim\theta_{2}$.
Our motivation comes from
the context of reflective submanifolds
in compact Riemannian symmetric spaces.
Originally, Leung (\cite{Leung}) introduced the notion of reflective submanifolds in a Riemannian manifold,
and classified reflective
submanifolds in irreducible
compact Riemannian symmetric spaces.
For any commutative compact symmetric triad
$(G,\theta_{1},\theta_{2})$,
two homogeneous spaces
$G^{\theta_{1}\theta_{2}}/(K_{1}\cap K_{2})$ and $K_{2}/(K_{1}\cap K_{2})$,
which are mutually orthogonal in $G/K_{1}$, give
reflective submanifolds of $G/K_{1}$,
where $G^{\theta_{1}\theta_{2}}$ denotes the fixed-point subgroup of $\theta_{1}\theta_{2}$
in $G$.
Conversely,
all reflective submanifolds of $G/K_{1}$ are essentially given in such a way.
For two isomorphic commutative compact symmetric triads with respect to $\equiv$,
it is shown that the corresponding reflective submanifolds are congruent.
Here, we note that
the Lie group structures
of $G^{\theta_{1}\theta_{2}}$
and $K_{1}\cap K_{2}$
are preserved by $\equiv$
but not $\sim$.
Leung's classification method is based on
a correspondence between reflective submanifolds
in compact Riemannian symmetric spaces and pseudo-Riemannian symmetric pairs,
and the Berger's classification for the pseudo-Riemannian symmetric pairs (\cite{Berger}).
In contrast, in our classification,
the isomorphism classes of commutative compact symmetric triads
with respect to $\equiv$ are given by symmetric triads with multiplicities (Theorem \ref{thm:equivequiv_ok}).
In the above arguments, our classification does not require Berger's classification.
In addition, by using the generalized duality introduced by the authors and Sasaki (\cite{BIS}), which is a one-to-one correspondence between commutative compact symmetric triads and pseudo-Riemannian symmetric pairs, 
we can give an alternative proof of Berger's classification.
In principle, 
the classification of commutative compact symmetric triads
yields that of pseudo-Riemannian symmetric pairs
via the generalized duality.
Then, our results of this paper is also useful for explicitly deriving the
pseudo-Riemannian symmetric pairs corresponding to commutative compact symmetric triads
(see \cite{BIS2} for details).

The organization of this paper is as follows:
In Section \ref{sec:CST}, we review the Hermann action corresponding to compact symmetric triads,
and reflective submanifolds in compact Riemannian symmetric spaces.
In Section \ref{sec:symm_triad},
we explain abstract symmetric triads
with multiplicities and classify
them.
Section \ref{sec:cst_symm} is the main part of this paper.
In this section,
we first determine explicitly
the symmetric triads with multiplicities
corresponding to
commutative compact symmetric
triads (Table \ref{table:symmtriext}).
Second,
we give the classification of 
commutative compact symmetric triads
by means of symmetric triads
with multiplicities.
Section \ref{sec:sigma_action}
is devoted to study 
$\sigma$-actions
on compact connected Lie groups,
which give examples of Hermann actions.
In this section,
we develop the theories
of symmetric triads
with multiplicities
and of double Satake diagrams
for $\sigma$-actions.

\section{Compact symmetric triads}\label{sec:CST}

\subsection{Compact symmetric triads and Hermann actions}\label{sec:CST_Ha}

Let $(G,\theta_{1},\theta_{2})$ be a compact symmetric triad.
For each $i=1,2$,
we denote by $K_{i}$
the identity component of the fixed-point subgroup of $\theta_{i}$ in $G$.
Then $G/K_{i}$ is a compact Riemannian symmetric space with respect to
the Riemannian metric induced from a bi-invariant Riemannian metric on $G$.
The natural isometric action of $K_{2}$ on $G/K_{1}$ is called a \emph{Hermann action}.
In the case when $\theta_{1}=\theta_{2}$,
the corresponding Hermann action
is nothing but the isotropy action of $K_{1}$ on $G/K_{1}$.
We can show that the Hermann action is hyperpolar.
Here, an isometric action of a compact connected Lie group
on a Riemannian manifold
is called \emph{hyperpolar},
if there exists a flat, connected closed submanifold that
meets all orbits orthogonally.
Such a submanifold is called a \emph{section} of the action.
It is known that any section becomes a totally geodesic submanifold.

In what follows,
we give a section of the $K_{2}$-action on $G/K_{1}$.
Let $\mathfrak{g}$ be the Lie algebra of $G$
and $\exp:\mathfrak{g}\to G$ denote the exponential map of $G$.
The differential of $\theta_{i}$ at the identity element of $G$
gives an involution of $\mathfrak{g}$,
which we write the same symbol $\theta_{i}$ if there is no confusion.
We have the canonical decomposition $\mathfrak{g}=\mathfrak{k}_{i}\oplus\mathfrak{m}_{i}$
of $\mathfrak{g}$ for $\theta_{i}$.
We note that $\mathfrak{k}_{i}$ is the Lie algebra of $K_{i}$.
The composition $\theta_{1}\theta_{2}$
gives a (not necessarily involutive) automorphism of $\mathfrak{g}$.
The fixed-point subalgebra $\mathfrak{g}^{\theta_{1}\theta_{2}}$ of $\theta_{1}\theta_{2}$
in $\mathfrak{g}$ is expressed as follows:
\begin{equation}\label{eqn:gt1t2_expression}
\mathfrak{g}^{\theta_{1}\theta_{2}}
=\{X\in\mathfrak{g}\mid \theta_{1}\theta_{2}(X)=X\}
=\{X\in\mathfrak{g}\mid \theta_{1}(X)=\theta_{2}(X)\}
=\mathfrak{g}^{\theta_{2}\theta_{1}}.
\end{equation}
This subalgebra is the Lie algebra
of the fixed-point subgroup $G^{\theta_{1}\theta_{2}}$
of $\theta_{1}\theta_{2}$ in $G$.
Since $\mathfrak{g}^{\theta_{1}\theta_{2}}$
is $(\theta_{1},\theta_{2})$-invariant,
that is, $\mathfrak{g}^{\theta_{1}\theta_{2}}$
is both $\theta_{1}$ and $\theta_{2}$-invariant,
$\theta_{1}$ and $\theta_{2}$ give involutions on it.
Then $\theta_{1}=\theta_{2}$ holds on $\mathfrak{g}^{\theta_{1}\theta_{2}}$ by \eqref{eqn:gt1t2_expression}.
Its canonical decomposition is given by
\begin{equation}
\mathfrak{g}^{\theta_{1}\theta_{2}}=(\mathfrak{k}_{1}\cap\mathfrak{k}_{2})\oplus(\mathfrak{m}_{1}\cap\mathfrak{m}_{2}).
\end{equation}
Let $\mathfrak{a}$ be a maximal abelian subspace of $\mathfrak{m}_{1}\cap\mathfrak{m}_{2}$.
It is known that $A:=\exp(\mathfrak{a})$
is closed in $G^{\theta_{1}\theta_{2}}$.
Hence, $A$ becomes a toral subgroup.
The following theorem follows from Hermann \cite{Hermann}
(see also Matsuki \cite[Theorem 1]{Matsuki97}).

\begin{thm}\label{thm:Hermann}
Under the above settings, we have$:$
\begin{equation}
G=K_{1}AK_{2}=K_{2}AK_{1}.
\end{equation}
\end{thm}

Let $\pi_{1}:G\to G/K_{1}$ be the natural projection.
Then, $\pi_{1}(A)$
becomes a flat, totally geodesic submanifold of $G/K_{1}$.
It follows from Theorem \ref{thm:Hermann}
that each $K_{2}$-orbit intersects $\pi_{1}(A)$.
Furthermore, they are orthogonal at the intersection points.
Therefore, $\pi_{1}(A)$ is a section of the $K_{2}$-action.
In particular,
the cohomogeneity
(i.e., the maximal dimension of the $K_{2}$-orbits)
of this action
is equal to the dimension of $\mathfrak{a}$.
We call it the \emph{rank} of $(G,\theta_{1},\theta_{2})$,
which we write $\mathrm{rank}(G,\theta_{1},\theta_{2})$.

Let $\mathrm{Aut}(G)$
be the automorphism group of $G$
and $\mathrm{Int}(G)$ be the inner automorphism group of $G$.
Then $\mathrm{Int}(G)$ is a normal subgroup of $\mathrm{Aut}(G)$.
Matsuki introduced the following equivalence relation
on the set of compact symmetric triads.

\begin{dfn}[\cite{Matsuki}]\label{dfn:CST_sim}
Two compact symmetric triads
$(G,\theta_{1},\theta_{2}), (G,\theta_{1}',\theta_{2}')$
are \emph{isomorphic},
if there exist $\varphi\in\mathrm{Aut}(G)$
and $\tau\in\mathrm{Int}(G)$ satisfying the following relations:
\begin{equation}\label{eqn:Matsuki_equiv}
\theta_{1}'=\varphi\theta_{1}\varphi^{-1},\quad
\theta_{2}'=\tau\varphi\theta_{2}\varphi^{-1}\tau^{-1}.
\end{equation}
Then we write $(G,\theta_{1},\theta_{2})\sim(G,\theta_{1}',\theta_{2}')$.
\end{dfn}

Geometrically, $(G,\theta_{1},\theta_{2})\sim(G,\theta_{1}',\theta_{2}')$
means that their Hermann actions are isomorphic to each other.
Indeed,
if we let $\varphi\in\mathrm{Aut}(G)$
and $\tau\in\mathrm{Int}(G)$ in \eqref{eqn:Matsuki_equiv},
then the $K_{2}$-action on $G/K_{1}$
is isomorphic to $\tau^{-1}(K_{2}')$-action on $G/K_{1}'$
via the isometry $\Phi:G/K_{1}\to G/K_{1}'$ defined by
\begin{equation}
\Phi:G/K_{1}\to G/K_{1}';~
gK_{1}\mapsto \varphi(g)K_{1}',
\end{equation}
where $K_{i}'$
($i=1,2$)
denotes the identity component of
the fixed-point subgroup of $\theta_{i}'$ in $G$.
This yields that $\mathrm{rank}(G,\theta_{1},\theta_{2})=\mathrm{rank}(G,\theta_{1}',\theta_{2}')$ holds
if $(G,\theta_{1},\theta_{2})\sim(G,\theta_{1}',\theta_{2}')$.
We define the rank of the isomorphism class $[(G,\theta_{1},\theta_{2})]$
of $(G,\theta_{1},\theta_{2})$
as that of $(G,\theta_{1},\theta_{2})$.

We define the \emph{order} of $(G,\theta_{1},\theta_{2})$
as the order of the composition $\theta_{1}\theta_{2}$,
i.e., the smallest positive integer $k$
satisfying $(\theta_{1}\theta_{2})^{k}=1$.
If there is no such $k$,
then $(G,\theta_{1},\theta_{2})$
has infinite order.
We denote by $\mathrm{ord}(G,\theta_{1},\theta_{2})$
the order of $(G,\theta_{1},\theta_{2})$.
We note that
the value of $\mathrm{ord}(G,\theta_{1},\theta_{2})$
depends on the choice of a representative of
the isomorphism class $[(G,\theta_{1},\theta_{2})]$.
We define the order of $[(G,\theta_{1},\theta_{2})]$
by the minimum value of 
$\{\mathrm{ord}(G,\theta_{1}',\theta_{2}')\mid (G,\theta_{1}',\theta_{2}')\in [(G,\theta_{1},\theta_{2})]\}$,
which may be in $\mathbb{N}\cup\{\infty\}$.

Here, we review
on the classification for compact symmetric triads
in the case when $G$ is simple.
Originally, Matsuki (\cite{Matsuki}) gave the classification of the isomorphism classes
$[(G,\theta_{1},\theta_{2})]$ of compact symmetric triads.
After him, the authors (\cite{BI})
gave an alternative proof of his classification
in terms of the notion of double Satake diagrams.
Indeed, each $[(G,\theta_{1},\theta_{2})]$
is expressed by the pair of two Satake diagrams
for $(G,\theta_{1})$ and $(G,\theta_{2})$
in a natural manner.
We exhibited the list of
all the isomorphism classes
$[(G,\theta_{1},\theta_{2})]$ with their ranks and orders (see \cite[Table 4]{BI}).
In particular,
we find that
the order of $[(G,\theta_{1},\theta_{2})]$ is finite.
From the table, we obtain
the classification of
commutative compact symmetric triads
with respect to $\sim$
(cf.~\cite[Remark 6.13]{BI}).

The second author introduced
the notion of symmetric triads
with multiplicities in order to study the geometry of Hermann actions.
The purpose of this paper
is to develop the theory of symmetric triads with multiplicities.
As explained in Section \ref{sec:ccst_symm},
we construct
a symmetric triad with multiplicities
from any commutative compact symmetric triad
$(G,\theta_{1},\theta_{2})$ and a maximal abelian subspace $\mathfrak{a}$ of $\mathfrak{m}_{1}\cap\mathfrak{m}_{2}$.
By the construction,
the symmetric triad with multiplicities constructed
from $(G,\theta_{2},\theta_{1})$
coincides with that of $(G,\theta_{1},\theta_{2})$.
On the other hand,
any commutative compact symmetric triad
uniquely determines
the isomorphism class of symmetric triads
with multiplicities
with respect to 
the equivalence relation $\equiv$
in the sense of Definition \ref{dfn:stm_equiv},
which will be introduced in Section \ref{sec:symm_triad}.
Based on our classification,
we will determine the resulting symmetric triads
with multiplicities for commutative compact symmetric triads
concretely (see Table \ref{table:symmtriext} in Section \ref{sec:cst_symm}).
Furthermore,
for two commutative compact symmetric triads
$(G,\theta_{1},\theta_{2})$
and $(G,\theta_{1}',\theta_{2}')$,
we will prove that
$(G,\theta_{1},\theta_{2})$
is locally isomorphic to
$(G,\theta_{1}',\theta_{2}')$
or $(G,\theta_{2}',\theta_{1}')$
with respect to  $\sim$
if and only if their symmetric triads with multiplicities
are isomorphic in the sense of Definition \ref{dfn:e-relation},
which will be discussed in Section \ref{sec:cst_symm}.

\subsection{Compact symmetric triads and reflective submanifolds}\label{sec:CST_rsubmfd}

Let $\tilde{M}$ be a complete Riemannian manifold.
A submanifold $M$ of $\tilde{M}$
is said to be \emph{reflective},
if $M$ is complete with respect to the induced metric
and there exists
an involutive isometry $\rho$ of $\tilde{M}$
such that $M$ is a connected component of the fixed-point subset of $\rho$ in $\tilde{M}$
(\cite{Leung}).
It is known that any reflective submanifold
is totally geodesic (\cite[p.~61]{KN}).

We review the construction of
reflective submanifolds
of compact Riemannian symmetric spaces from commutative compact symmetric triads due to Leung \cite{Leung}.
Let $(G,\theta_{1},\theta_{2})$ be a commutative compact symmetric triad
and $o=eK_{1}$ denote the origin of the compact Riemannian symmetric space $G/K_{1}$.
Assume that $G$ is simply connected.
Then the fixed-point subgroup $G^{\theta}$ of
$\theta \in \{\theta_{1},\theta_{2},\theta_{1}\theta_{2}\}$ in $G$
is connected (cf.~\cite[Theorem 8.2 in Chapter VII]{Helgason}).
In particular, we have $K_{i}=G^{\theta_{i}}$ for $i=1,2$.
It follows from $\theta_{1}\theta_{2}=\theta_{2}\theta_{1}$
that
$K_{1}$ becomes $\theta_{2}$-invariant.
Let us define an involutive isometry of $G/K_{1}$
as follows:
\begin{equation}
\rho(gK_{1})=\theta_{2}(g)K_{1},\quad
g\in G.
\end{equation}
We denote by $(G/K_{1})^{\rho}$
the fixed-point subset of $\rho$ in $G/K_{1}$.
By the definition,
the origin $o$ is in $(G/K_{1})^{\rho}$.
Then the connected component of
$(G/K_{1})^{\rho}$ containing $o$
is a reflective submanifold of $G/K_{1}$.
It is shown that
the reflective submanifold
is expressed as $K_{2}/(K_{1}\cap K_{2})$.
Conversely, any
reflective submanifold
containing $o$
is constructed as above.
On the other hand,
$G^{\theta_{1}\theta_{2}}/(K_{1}\cap K_{2})$
gives another reflective submanifold of $G/K_{1}$,
which is called the complementary space of $K_{2}/(K_{1}\cap K_{2})$
(cf.~\cite[Corollary 1.2]{Leung}).
It is shown that
$G^{\theta_{1}\theta_{2}}/(K_{1}\cap K_{2})$
intersects $K_{2}/(K_{1}\cap K_{2})$ orthogonally.

Here, we define another kind of equivalence relation $\equiv$
on the set of compact symmetric triads as follows.

\begin{dfn}\label{dfn:CST_equiv}
Let $(G,\theta_{1},\theta_{2})$
and $(G,\theta_{1}',\theta_{2}')$
be two (not necessarily commutative) compact symmetric triads.
We write $(G,\theta_{1},\theta_{2})\equiv (G,\theta_{1}',\theta_{2}')$
if there exists an automorphism $\varphi$ of $G$
satisfying $\theta_{i}'=\varphi\theta_{i}\varphi^{-1}$
for $i=1,2$.
\end{dfn}

By the definition,
$(G,\theta_{1},\theta_{2})\equiv(G,\theta_{1}',\theta_{2}')$
yields $(G,\theta_{1},\theta_{2})\sim(G,\theta_{1}',\theta_{2}')$.
However, the converse does not hold in general.
Indeed, we find an example
of such commutative compact symmetric triads
in \cite[Example 6.5]{BI}.
On the other hand,
the equivalence relation $\equiv$
is compatible with the commutativity of compact symmetric triads,
namely,
if $(G,\theta_{1},\theta_{2})\equiv(G,\theta_{1}',\theta_{2}')$
and $(G,\theta_{1},\theta_{2})$ is commutative,
then so is $(G,\theta_{1}',\theta_{2}')$.
Furthermore,
the Lie group structures
of $G^{\theta_{1}\theta_{2}}$
and $K_{1}\cap K_{2}$
are independent of the choice of representatives in
the equivalence class of $(G,\theta_{1},\theta_{2})$
with respect to $\equiv$.
The following proposition gives our motivation
for introducing the equivalence relation $\equiv$.

\begin{pro}
If $(G,\theta_{1},\theta_{2})\equiv(G,\theta_{1}',\theta_{2}')$,
then
$K_{2}/(K_{1}\cap K_{2})$ and $G^{\theta_{1}\theta_{2}}/(K_{1}\cap K_{2})$
are isomorphic to 
$K_{2}'/(K_{1}'\cap K_{2}')$ and $G^{\theta_{1}'\theta_{2}'}/(K_{1}'\cap K_{2}')$,
respectively.
\end{pro}

The classification of reflective submanifolds
of compact Riemannian symmetric spaces
was given by Leung (\cite{Leung}).
Indeed, he found a correspondence between
reflective submanifolds
and pseudo-Riemannian symmetric pairs.
Then he determined
$K_{2}/(K_{1}\cap K_{2})$ and $G^{\theta_{1}\theta_{2}}/(K_{1}\cap K_{2})$
by means of the classification of pseudo-Riemannian symmetric pairs
due to Berger (\cite{Berger}).

The authors and Sasaki (\cite{BIS})
gave a one-to-one correspondence between
commutative compact symmetric triads
and pseudo-Riemannian symmetric pairs.
This correspondence is a generalization
of Cartan's duality,
which is the one-to-one correspondence
between compact symmetric pairs
and noncompact Riemannian symmetric pairs.
As shown in Section \ref{sec:cst_classify_equiv},
we will classify the local isomorphism class of
commutative compact symmetric triads $(G,\theta_{1},\theta_{2})$
with respect to $\equiv$ in terms of symmetric triads
with multiplicities.
Then
each equivalence class of $(G,\theta_{1},\theta_{2})$
is specified by the equivalence class of the symmetric triads with multiplicities (see Theorem \ref{thm:equivequiv_ok}).

As applications, we can give a method to determine the Lie group structures of $K_{1}\cap K_{2}$
and $G^{\theta_{1}\theta_{2}}$ from a given commutative compact symmetric triad
$(G,\theta_{1},\theta_{2})$.
Hence we can obtain an alternative proof of Leung's classification theorem
for reflective submanifolds.
On the other hand, 
we can also give an alternative proof
of Berger's classification for pseudo-Riemannian symmetric pairs
by means of the generalized duality
and our classification of
commutative compact symmetric triads
(see \cite{BIS2} for the detail).

\section{The classification for symmetric triads with multiplicities}\label{sec:symm_triad}

In this section,
we first recall the notions of root systems and
abstract symmetric triads with multiplicities.
In particular,
we review the classification of symmetric triads \emph{without} multiplicities
due to the second author (\cite{Ikawa}).
In this classification, symmetric triads are divided into three types (I)--(III).
Second, by using his classification, we will give the classification of
symmetric triads \emph{with} multiplicities.
Finally, we will introduce the notion of symmetric triads of type (IV)
with multiplicities and give its classification.
The above two classifications will be applied
to classify the equivalence classes
of commutative compact symmetric triads
with respect to $\equiv$
in the next section.

\subsection{Preliminaries}

\subsubsection{Root systems}
We begin with recalling the definition of a root system.
Let $\mathfrak{t}$ be a finite dimensional real vector space
with an inner product $\INN{\cdot}{\cdot}$.
We write $\|\alpha\|=\INN{\alpha}{\alpha}^{1/2}$ for $\alpha\in\mathfrak{t}$.
For $\alpha\in\mathfrak{t}-\{0\}$ we define a linear isometry $w_{\alpha}\in O(\mathfrak{t})$ by
\begin{equation}\label{eqn:dfn_reflection}
w_\alpha (H)=H-2\frac{\langle\alpha ,H\rangle}{\|\alpha\|^2}\alpha.
\end{equation}
Then $w_\alpha$ satisfies $w_{\alpha}(\alpha)=-\alpha$,
$w_{\alpha}(H)=H$
($H\in(\mathbb{R}\alpha)^{\perp}$)
and
$w_{\alpha}^2=1$.
Here, $(\mathbb{R}\alpha)^{\perp}$
denotes the orthogonal complement of $\mathbb{R}\alpha$ in $\mathfrak{t}$.

\begin{dfn}
A finite subset $\Delta\subset\mathfrak{t}-\{0\}$ is called a \textit{root system}
of $\mathfrak{t}$, if it satisfies the following two conditions:
\begin{enumerate}
\item $\mathfrak{t}=\mathrm{span}_{\mathbb{R}}(\Delta)$.
\item If $\alpha$ and $\beta$ are in $\Delta$,
then $w_{\alpha}(\beta)=\beta-2\dfrac{\INN{\alpha}{\beta}}{\|\alpha\|^{2}}\alpha$ is in $\Delta$, and
$2\dfrac{\langle\alpha ,\beta\rangle}{\|\alpha\|^2}$ is in $\mathbb{Z}$.
\end{enumerate}
In addition, a root system $\Delta$ is said to be \textit{reduced},
if it satisfies the following condition:
\begin{enumerate}
\setcounter{enumi}{2}
\item If $\alpha$ and $\beta$ are in $\Delta$ with $\beta = m\alpha$,
then $m=\pm 1$ holds.
\end{enumerate}
\end{dfn}
A root system $\Delta$
of $\mathfrak{t}$
is said to be reducible
if
there exist two nonempty subsets $\Delta_{1}$
and $\Delta_{2}$ of $\Delta$ satisfying the following conditions:
\[
\Delta=\Delta_{1}\cup\Delta_{2},\quad
\Delta_{1}\cap\Delta_{2}=\emptyset,\quad
\INN{\Delta_{1}}{\Delta_{2}}=\{0\}.
\]
Otherwise it is said to be \textit{irreducible}.
Any root system is uniquely decomposed into irreducible ones,
namely, there exist unique irreducible root systems $\Delta_{1},\dotsc,\Delta_{l}$
up to permutation of the indices
such that
$\Delta=\Delta_{1}\cup \dotsb \cup \Delta_{l}$
and that $\INN{\Delta_{i}}{\Delta_{j}}=\{0\}$ for $1\leq i\neq j\leq l$.
This decomposition of $\Delta$ is called the irreducible decomposition of $\Delta$.

Let $\Delta$ and $\Delta'$ be
root systems of $\mathfrak{t}$ and $\mathfrak{t}'$, respectively.
A linear isomorphism $f:\mathfrak{t}\to\mathfrak{t}'$
satisfying $f(\Delta)=\Delta'$
and preserving the integers $2\INN{\beta}{\alpha}/\|\alpha\|^{2}$
is called an \textit{isomorphism} of root systems between $\Delta$ and $\Delta'$.
Two root systems $\Delta$ and $\Delta'$ are \textit{isomorphic},
which we write $\Delta\simeq\Delta'$,
if there exists such $f$ (cf.~\cite[p.~150]{Knapp}).
We find that
$\simeq$ gives an equivalence relation on the set of root systems.

Let $\mathrm{Aut}(\Delta)$ denote the group of all automorphisms of $\Delta$.
It is clear that $\mathrm{Aut}(\Delta)$ is a finite group.
The \textit{Weyl group} $W(\Delta)$ of a root system $\Delta$ is defined by a subgroup of $O(\mathfrak{t})$
generated by $\{w_\alpha \,|\, \alpha\in\Delta\}$.
Then $W(\Delta)$ is a normal subgroup of $\mathrm{Aut}(\Delta)$.
In particular, $W(\Delta)$ is a finite group.
It is known that $W(\Delta)$ acts simply transitively on the set of fundamental systems of $\Delta$.

We use the notations of irreducible root systems $\Delta$ and the set $\Delta^{+}$ of positive roots
with respect to a specified ordering as follows (cf.~\cite{Bo}):

\begin{nota}\label{nota:root_system}
There are five infinite families $A_{r\geq 1}$, $B_{r\geq 1}$, $C_{r\geq 1}$, $D_{r\geq 2}$ (reduced),
and $BC_{r\geq 1}$ (non-reduced), which are called the \textit{classical type},
and five exceptional cases $E_{r}$ ($r=6,7,8$), $F_{4}$, $G_{2}$ (reduced),
which are called the \textit{exceptional type}.
\begin{itemize}
\setlength{\leftskip}{-5mm} %%% leftskip
\item Classical type:
\setlength{\parskip}{.2cm} %
\setlength{\itemsep}{.2cm} %
\begin{itemize}
\setlength{\leftskip}{-10mm} %%% leftskip
\setlength{\parskip}{.1cm} %
\setlength{\itemsep}{.1cm} %
\item[] $A_r^+=\{e_i-e_j\mid 1\leq i<j\leq r+1\}$,
\item[] $B_r^+=\{e_i\mid 1\leq i\leq r\}\cup \{e_i\pm e_j\mid 1\leq i<j\leq r\}$,
\item[] $C_r^+=\{2e_i\mid 1\leq i\leq r\}\cup\{e_i\pm e_j\mid 1\leq i<j\leq r\}$,
\item[] $D_r^+=\{e_i\pm e_j\mid 1\leq i<j\leq r\}$,
\item[] $BC_r^+=\{e_i,2e_i\mid 1\leq i\leq r\}\cup\{e_i\pm e_j\mid 1\leq i<j\leq r\}$.
\end{itemize}
\item Exceptional type:
\begin{itemize}
\setlength{\leftskip}{-10mm} %%% leftskip
\setlength{\parskip}{.1cm} %
\setlength{\itemsep}{.1cm} %
\item[] $E_{6}^{+}=\{\pm e_{i}+e_{j}\mid 1 \leq i< j\leq 5\}$\\
\phantom{hogehogehogehogehoge\,}$\cup\left\{\dfrac{1}{2}(e_{8}-e_{7}-e_{6}+\displaystyle\sum_{i=1}^{5}(-1)^{\nu(i)}e_{i})\MID \displaystyle\sum_{i=1}^{5}\nu(i):\text{even}\right\}$,
\item[] $E_{7}^{+}=\{\pm e_{i}+e_{j}\mid 1 \leq i< j\leq 6\}\cup\{-e_{7}+e_{8}\}$\\
\phantom{hogehogehogehogehogehog\,}$\cup\left\{\dfrac{1}{2}(-e_{7}+e_{8}+\displaystyle\sum_{i=1}^{6}(-1)^{\nu(i)}e_{i})\MID \displaystyle\sum_{i=1}^{6}\nu(i):\text{odd}\right\}$,
\item[] $E_{8}^{+}=\{\pm e_{i}+e_{j}\mid 1 \leq i< j\leq 8\}\cup\left\{\dfrac{1}{2}(e_{8}+\displaystyle\sum_{i=1}^{7}(-1)^{\nu(i)}e_{i})\MID \displaystyle\sum_{i=1}^{7}\nu(i):\text{even}\right\},$
\item[] $F_{4}^{+}=\{e_{i}\mid 1\leq i \leq 4\}\cup\{e_{i}\pm e_{j}\mid 1 \leq i < j \leq 4\}\cup\left\{\dfrac{1}{2}(e_{1}\pm e_{2}\pm e_{3}\pm e_{4})\right\}$,
\item[] $G_{2}^{+}=\{e_{1}-e_{2},-e_{1}+e_{3},-e_{2}+e_{3},-2e_{1}+e_{2}+e_{3},e_{1}-2e_{2}+e_{3},-e_{1}-e_{2}+2e_{3}\}$.
\end{itemize}
\end{itemize}
For each above root system $\Delta$,
the subscript denotes the \textit{rank} of $\Delta$,
that is, the dimension of $\mathrm{span}_{\mathbb{R}}(\Delta)$.
For formal reasons, we put $D_{1}=\emptyset$.
Here, we have the following isomorphisms:
\begin{equation}\label{eqn:root_iso}
\begin{array}{ccccc}
D_{3}\simeq A_{3},&
D_{2}\simeq A_{1}\cup A_{1},&
B_{1}\simeq C_{1}\simeq A_{1},&
B_{2}\simeq C_{2}.
\end{array}
\end{equation}
For the sets $\Delta^{+}$ of positive roots above, the sets of simple roots $\Pi=\Pi(\Delta^{+})$
and its highest roots $\tilde{\delta}$
are given as follows:
\begin{itemize}
\setlength{\leftskip}{-5mm} %%% leftskip
\item Classical type:
\setlength{\parskip}{.2cm} %
\setlength{\itemsep}{.2cm} %
\begin{itemize}
\setlength{\leftskip}{-10mm} %%% leftskip
\setlength{\parskip}{.1cm} %
\setlength{\itemsep}{.1cm} %
\item[] $A_{r}^{+}:\begin{cases} % A_{r}
\Pi=\{\alpha_1=e_1-e_2,\dotsc ,\alpha_r=e_r-e_{r+1}\},\\
\tilde{\delta}=e_{1}-e_{r+1}={\alpha_{1}+\dotsb+\alpha_{r}},
\end{cases}$
\item[] $B_{r}^{+}:\begin{cases} % B_{r}
\Pi=\Pi(B_{r}^{+})=\{\alpha_1=e_1-e_2,\dotsc,\alpha_{r-1}=e_{r-1}-e_r,\alpha_r=e_r\},\\
\tilde{\delta}=e_{1}+e_{2}=\alpha_{1}+2\alpha_{2}+\dotsb+2\alpha_{r},
\end{cases}$
\item[] $C_{r}^{+}:\begin{cases} % C_{r}
\Pi=\{\alpha_1=e_1-e_2,\dotsc ,\alpha_{r-1}=e_{r-1}-e_r,\alpha_r=2e_r\},\\
\tilde{\delta}=2e_{1}=2\alpha_{1}+2\alpha_{2}+\dotsb+2\alpha_{r-1}+\alpha_{r},
\end{cases}$
\item[] $D_{r}^{+}:\begin{cases} % D_{r}
\Pi=\{\alpha_1=e_1-e_2,\dotsc ,\alpha_{r-1}=e_{r-1}-e_r,\alpha_r=e_{r-1}+e_r\},\\
\tilde{\delta}=e_{1}+e_{2}=\alpha_{1}+2\alpha_{2}+\dotsb+2\alpha_{r-2}+\alpha_{r-1}+\alpha_{r},
\end{cases}$
\item[] $BC_{r}^{+}:\begin{cases} % BC_{r}
\Pi=\Pi(B_{r}^{+}),\\
\tilde{\delta}=2e_{1}=2\alpha_{1}+\dotsb+2\alpha_{r}.
\end{cases}$
\end{itemize}
%%%
\item Exceptional type:
\begin{itemize}
\setlength{\leftskip}{-18mm} %%% leftskip
\setlength{\parskip}{.1cm} %
\setlength{\itemsep}{.1cm} %
\item[] $E_{6}^{+}:\begin{cases} % E_{6}
\Pi\!=\!\Pi(E_{6}^{+})\!=\!\left\{\alpha_{1}\!=\!\dfrac{1}{2}(e_{1}+e_{8})-\dfrac{1}{2}(e_{2}+e_{3}+e_{4}+e_{5}+e_{6}+e_{7})\right\}\\
\phantom{hogeh}\cup\{\alpha_{2}=e_{1}+e_{2},\alpha_{3}=e_{2}-e_{1},\alpha_{4}=e_{3}-e_{2},\alpha_{5}=e_{4}-e_{3},\alpha_{6}=e_{5}-e_{4}\},\\
\tilde{\delta}\!=\!\dfrac{1}{2}(e_{1}+e_{2}+e_{3}+e_{4}+e_{5}-e_{6}-e_{7}+e_{8})
\!=\!\alpha_{1}\!+2\alpha_{2}\!+2\alpha_{3}\!+3\alpha_{4}\!+2\alpha_{5}\!+\alpha_{6},
\end{cases}$
\item[] $E_{7}^{+}:\begin{cases} % E_{7}
\Pi=\Pi(E_{7}^{+})=\Pi(E_{6}^{+})\cup\{\alpha_{7}=e_{6}-e_{5}\},\\
\tilde{\delta}=e_{8}-e_{7}=2\alpha_{1}+2\alpha_{2}+3\alpha_{3}+4\alpha_{4}+3\alpha_{5}+2\alpha_{6}+\alpha_{7},
\end{cases}$
\item[] $E_{8}^{+}:\begin{cases} % E_{8}
\Pi=\Pi(E_{7}^{+})\cup\{\alpha_{8}=e_{7}-e_{6}\},\\
\tilde{\delta}=e_{7}+e_{8}=2\alpha_{1}+3\alpha_{2}+4\alpha_{3}+6\alpha_{4}+5\alpha_{5}+4\alpha_{6}+3\alpha_{7}+2\alpha_{8},
\end{cases}$
\item[] $F_{4}^{+}:\begin{cases} % F_{4}
\Pi=\{\alpha_{1}=e_{2}-e_{3},\alpha_{2}=e_{3}-e_{4},\alpha_{3}=e_{4},\alpha_{4}=\dfrac{1}{2}(e_{1}-e_{2}-e_{3}-e_{4})\},\\
\tilde{\delta}=e_{1}+e_{2}=2\alpha_{1}+3\alpha_{2}+4\alpha_{3}+2\alpha_{4},
\end{cases}$
\item[] $G_{2}^{+}:\begin{cases} % G_{2}
\Pi=\{\alpha_{1}=e_{1}-e_{2},\alpha_{2}=-2e_{1}+e_{2}+e_{3}\},\\
\tilde{\delta}=-e_{1}-e_{2}+2e_{3}=3\alpha_{1}+2\alpha_{2}.
\end{cases}$
\end{itemize}
\end{itemize}
\end{nota}

\subsubsection{Symmetric triads with multiplicities}\label{subsec:ST}
In this subsection,
we begin with recalling the definition of a root system with multiplicity.

\begin{dfn}\label{dfn:rootsig}
Let $\Sigma$ be a root system of a real vector space $\mathfrak{a}$
with an inner product $\INN{\cdot}{\cdot}$.
Put $\mathbb{R}_{>0}=\{x\in\mathbb{R}\mid x>0\}$.
Consider a mapping $m:\Sigma\rightarrow \mathbb{R}_{>0};\alpha\mapsto m_\alpha$ which satisfies
\begin{equation}\label{eqn:multi}
m_{w_{\alpha}(\beta)}=m_{\beta}\quad\mbox{for}\quad \alpha,\beta\in\Sigma.
\end{equation}
We call $m_\alpha$ the \textit{multiplicity} of $\alpha$.
If the multiplicities are given, we call $(\Sigma; m)$ the \textit{root system with multiplicities}.
\end{dfn}

If $\Sigma$ is irreducible, then it is known that $W(\Sigma)$ acts transitively on each
subset $\{\beta\in\Sigma\mid \|\beta\|=\|\alpha\|\}$ for $\alpha\in\Sigma$ (cf.~\cite[11.~(b), p.~204]{Knapp}).
Thus the condition (\ref{eqn:multi}) means
\begin{equation}\label{eqn:multi_length}
m_\alpha=m_\beta \quad\text{ if }\quad \|\alpha\|=\|\beta\|.
\end{equation}

We define an equivalence relation $\simeq$
on the set of root systems with multiplicities as follows:

\begin{dfn}\label{dfn:rootmequiv}
Let $(\Sigma; m)$ and $(\Sigma'; m')$
be root systems with multiplicities of $\mathfrak{a}$ and $\mathfrak{a}'$, respectively.
Then $(\Sigma; m)$ and $(\Sigma';m')$
are \textit{isomorphic} if there exists an
isomorphism $f:\mathfrak{a}\to\mathfrak{a}'$ of the root systems
$\Sigma$ and $\Sigma'$
satisfying $m_{\alpha}=m'_{f(\alpha)}$ $(\alpha\in\Sigma)$.
If $(\Sigma; m)$ and $(\Sigma';m')$
are isomorphic,
then we write $(\Sigma; m)\simeq(\Sigma'; m')$.
\end{dfn}

We review the definition of a symmetric triad.

\begin{dfn}[{\cite[Definition 2.2]{Ikawa}}]\label{dfn:symmtriad}
A triple $(\tilde{\Sigma},\Sigma ,W)$ is a \textit{symmetric triad} of $\mathfrak{a}$, if it satisfies the following six conditions:
\begin{enumerate}[(1)]
\item $\tilde{\Sigma}$ is an irreducible root system of $\mathfrak{a}$.
\item $\Sigma$ is a root system of $\mathrm{span}_{\mathbb{R}}(\Sigma)$.
\item $W$ is a nonempty subset of $\mathfrak{a}$, which is invariant under the multiplication by $-1$, and $\tilde{\Sigma}=\Sigma\cup W$.
\item $\Sigma\cap W$ is a nonempty subset. If we put $l=\max\{\|\alpha\|\mid \alpha\in \Sigma\cap W\}$, then
$\Sigma\cap W=\{\alpha\in\tilde{\Sigma}\mid \|\alpha\|\leq l\}$.
\item For $\alpha\in W,\lambda\in\Sigma-W$, the integer
$2\dfrac{\langle\alpha ,\lambda\rangle}{\|\alpha\|^2}$ is odd if and only if 
$w_\alpha\lambda \in W-\Sigma$.
\item For $\alpha\in W,\lambda\in W-\Sigma$, the integer $2\dfrac{\langle\alpha ,\lambda\rangle}{\|\alpha\|^2}$
is odd if and only if $w_\alpha\lambda \in \Sigma-W$.
\end{enumerate}
\end{dfn}

When $(\tilde{\Sigma},\Sigma ,W)$ is a symmetric triad of $\mathfrak{a}$, then
$\mathrm{span}_{\mathbb{R}}(\Sigma)=\mathfrak{a}$ (see \cite[Remark~1.13]{Ikawa2}).
It is known that $W$ is invariant under the action of the Weyl group $W(\Sigma)$ of $\Sigma$
(\cite[Proposition 2.7]{Ikawa}).
This fact will be used in (2) of Definition~\ref{dfn:multi}.

We define a lattice $\Gamma$ of $\mathfrak{a}$
for $\tilde{\Sigma}$ by
\begin{equation}\label{eqn:Gamma}
\Gamma=\left\{X\in\mathfrak{a}\MID
\langle\lambda ,X\rangle \in\frac{\pi}{2}\mathbb{Z}
\quad (\lambda\in\tilde{\Sigma})\right\}.
\end{equation}

\begin{dfn}[{\cite[Definition~2.13]{Ikawa}}]\label{dfn:multi}
Let $(\tilde{\Sigma},\Sigma ,W)$ be a symmetric triad of $\mathfrak{a}$.
Put $\mathbb{R}_{\geq 0}=\{x\in\mathbb{R}\mid x\geq 0\}$.
Consider two mappings $m,n:\tilde{\Sigma}\rightarrow \mathbb{R}_{\geq 0}$
which satisfy the following four conditions:
\begin{enumerate}[(1)]
\item $m(\lambda )=m(-\lambda ),\; n(\alpha )=n(-\alpha )$ for $\lambda,\alpha\in\tilde{\Sigma}$ and
$$m(\lambda )>0\Leftrightarrow \lambda \in \Sigma ,\quad
n(\alpha )>0\Leftrightarrow \alpha\in W.
$$
\item When $\lambda\in \Sigma,\alpha\in W,s\in W(\Sigma )$ then
$m(\lambda )=m(s\lambda ), n(\alpha )=n(s\alpha )$.
\item When $\sigma\in W(\tilde{\Sigma})$, the Weyl group of
$\tilde{\Sigma}$, and $\lambda\in\tilde{\Sigma}$ then
$n(\lambda )+m(\lambda )=n(\sigma\lambda )+m(\sigma\lambda )$.
\item Let $\lambda\in\Sigma\cap W$ and $\alpha\in W$.

If $\dfrac{2\langle\alpha ,\lambda\rangle}{\|\alpha\|^2}$ is even then
$m(\lambda )=m(w_{\alpha}\lambda)$.

If $\dfrac{2\langle\alpha ,\lambda\rangle}{\|\alpha\|^2}$ is odd then
$m(\lambda )=n(w_{\alpha}\lambda)$.
\end{enumerate}
We call $m(\lambda )$ and $n(\alpha )$ the \textit{multiplicities} of
$\lambda$ and $\alpha$, respectively.
If multiplicities are given, we call $(\tilde{\Sigma},\Sigma ,W;m,n)$
the \textit{symmetric triad with multiplicities}.
\end{dfn}

The second author defined an equivalence relation on the set of symmetric triads in \cite{Ikawa}.
We extend it to an equivalence relation on the set of symmetric triads \textit{with} multiplicities as follows:

\begin{dfn}\label{dfn:e-relation}
Let $(\tilde{\Sigma},\Sigma ,W;m,n)$ and $(\tilde{\Sigma}',\Sigma ',W';m',n')$ be
symmetric triads with multiplicities of $\mathfrak a$ and ${\mathfrak a}'$, respectively.
Then two symmetric triads $(\tilde{\Sigma},\Sigma ,W)$ and $(\tilde{\Sigma}',\Sigma ',W')$ are
\textit{isomorphic}, if there exist an isomorphism $f:\tilde{\Sigma}\to\tilde{\Sigma}'$ of root systems
and $Y\in \Gamma$ such that
\begin{equation}\label{eqn:equiv-relation}
\begin{cases}
\Sigma '-W' \!=\! \{f(\alpha )\,|\,\alpha\in\Sigma\!-\!W,\langle \alpha ,2Y\rangle \in 2\pi \mathbb{Z}\}
\!\cup\!\{f(\alpha )\,|\,\alpha\in W\!-\!\Sigma ,\langle \alpha ,2Y\rangle \in\pi +2\pi\mathbb{Z}\},\\
W'-\Sigma '\!=\!\{f(\alpha )\,|\,\alpha\in W\!-\!\Sigma ,\langle \alpha ,2Y\rangle \in 2\pi \mathbb{Z}\}
\!\cup\!\{f(\alpha )\,|\,\alpha\in \Sigma\!-\!W,\langle \alpha ,2Y\rangle \in\pi +2\pi\mathbb{Z}\}.
\end{cases}
 \end{equation}
In addition,
if $f$ and $Y$ satisfy the following condition, we say that
$(\tilde{\Sigma},\Sigma ,W;m,n)$ and $(\tilde{\Sigma}',\Sigma ',W';m',n')$ are \textit{isomorphic}:
For $\alpha\in\tilde{\Sigma}$,
\begin{equation}\label{eqn:equiv-relation-mn}
\begin{cases}
\begin{array}{lll}
m(\alpha)=m'(f(\alpha)),& n(\alpha)=n'(f(\alpha))& \mbox{if}\quad \langle\alpha ,2Y\rangle\in 2\pi\mathbb{Z},\\
m(\alpha)=n'(f(\alpha)),& n(\alpha)=m'(f(\alpha))& \mbox{if}\quad \langle\alpha ,2Y\rangle\in \pi+2\pi\mathbb{Z}.
\end{array}
\end{cases}
\end{equation}
\end{dfn}
\noindent
If $(\tilde{\Sigma},\Sigma ,W;m,n)$ and $(\tilde{\Sigma}',\Sigma ',W';m',n')$ are
isomorphic, we write
\[
(\tilde{\Sigma},\Sigma ,W;m,n)\sim (\tilde{\Sigma}',\Sigma ',W';m',n').
\]
The relation $\sim$ is an equivalence relation.

\subsubsection{Examples of symmetric triads with multiplicities}

Based on \cite[Theorem~2.19]{Ikawa}
we give
symmetric triads $(\tilde{\Sigma},\Sigma ,W;m,n)$ of $\mathfrak a$ with multiplicities
as follows:
\begin{enumerate}
\item[(I)] In the case where $\Sigma \supsetneq W$:
\begin{itemize}
\item Type (I-$B_r$): $\tilde{\Sigma}=\Sigma=B_r\supset W=\{\pm e_i\mid 1\leq i\leq r\}$.
\[
\begin{array}{ccc}
0<m(\pm e_i)=\mbox{const},&
0<m(\pm e_i\pm e_j)=\mbox{const}\; (i\not= j),&
0<n(\pm e_i)=\mbox{const}.
\end{array}
\]
\item Type (I-$C_r$): $\tilde{\Sigma}=\Sigma=C_r\supset W=D_r$.
When $r\geq 3$, then
\[
\begin{array}{cc}
0<m(\pm e_i\pm e_j)=n(\pm e_i\pm e_j)=\mbox{const}\; (i\not= j),&
0<m(\pm 2e_i)=\mbox{const}.
\end{array}
\]
When $r=2$, then
\[
\begin{array}{ccc}
0<m(\pm e_1\pm e_2)=\mbox{const},&
0<m(\pm 2e_i)=\mbox{const},&
0<n(\pm e_1\pm e_2)=\mbox{const}.
\end{array}
\]
\item Type (I-$BC_r$-$A_1^r$): $\tilde{\Sigma}=\Sigma=BC_r\supset W=A_1^r=\{\pm e_i\mid 1\leq i\leq r\}$.
\begin{equation*}
\begin{array}{ll}
0<m(\pm e_i)=\mbox{const}, & 0<m(\pm e_i\pm e_j)=\mbox{const}\; (i\not= j),\\
0<m(\pm 2e_i)=\mbox{const}, & 0<n(\pm e_i)=\mbox{const}.
\end{array}
\end{equation*}
\item Type (I-$BC_r$-$B_r$): $\tilde{\Sigma}=\Sigma=BC_r\supset W=B_r$.
When $r\geq 3$,
\[
\begin{array}{l}
0<m(\pm e_i)=n(\pm e_i)=\mbox{const},\quad 0<m(\pm 2e_i)=\mbox{const},\\
0<m(\pm e_i\pm e_j)=n(\pm e_i\pm e_j)=\mbox{const}\; (i\not=j).\\
\end{array}
\]
When $r=2$, then
\begin{equation*}
\begin{array}{ll}
0<m(\pm e_i)=n(\pm e_i)=\mbox{const},&0<m(\pm 2e_i)=\mbox{const},\\
0<m(\pm e_{1}\pm e_{2})=\mbox{const},&0<n(\pm e_{1}\pm e_{2})=\mbox{const}.
\end{array}
\end{equation*}
\item Type (I-$F_4$): $\tilde{\Sigma}=\Sigma=F_4\supset W=\{$short roots in $F_4\}\simeq D_4$.
\begin{equation*}
\begin{array}{cc}
0<m(\alpha\in W)=n(\alpha\in W)=\mbox{const},&
0<m(\lambda\in\Sigma-W)=\mbox{const}.
\end{array}
\end{equation*}
\end{itemize}
\item[(II)] In the case where $W \supsetneq \Sigma$:
\begin{itemize}
\item Type (II-$BC_r$): $\tilde{\Sigma}=W=BC_r\supset \Sigma=B_r$.
\[
\begin{array}{l}
0<n(\pm e_i)=m(\pm e_i)=\mbox{const},\quad 0<n(\pm 2e_i)=\mbox{const},\\
0<n(\pm e_i\pm e_j)=m(\pm e_i\pm e_j)=\mbox{const}\; (i\not= j).
\end{array}
\]
\end{itemize}
\item[(I')] In the case where $\Sigma \not=W$ except for (I) and (II):
\begin{itemize}
\item Type (I'-$C_r$): $\tilde{\Sigma}=W=C_r\supset \Sigma=D_r$.
When $r\geq 3$, then
\[
\begin{array}{cc}
0<m(\pm e_i\pm e_j)=n(\pm e_i\pm e_j)=\mbox{const},&
0<n(\pm 2e_i)=\mbox{const}.
\end{array}
\]
When $r=2$, then
\[
\begin{array}{l}
n(\pm 2e_1)=n(\pm 2e_2)=\mbox{const},\quad m(\pm (e_1+e_2))=n(\pm (e_1-e_2))=\mbox{const},\\
n(\pm (e_1+e_2))=m(\pm (e_1-e_2))=\mbox{const}.
\end{array}
\]
It is known $(\mbox{I-}C_r)\sim (\mbox{I'-}C_r)$ as symmetric triads.
\item Type (I'-$F_4$): $\tilde{\Sigma}=F_4$ and
\begin{align*}
\Sigma&=\{\mbox{short roots of }F_4\}\cup \{\pm e_1\pm e_2,\pm e_3\pm e_4\}\simeq C_4,\\
W&=\{\mbox{short roots of }F_4\}\cup \{\pm e_1\pm e_3,\pm e_1\pm e_4,\pm e_2\pm e_3,\pm e_2\pm e_4\}.
\end{align*}
\[
\begin{array}{cc}
0<m(\mbox{short})=n(\mbox{short})=\mbox{const},&0<m(\mbox{long}\in\Sigma)=n(\mbox{long}\in W)=\mbox{const}.
\end{array}
\]
It is known $(\mbox{I-}F_4)\sim (\mbox{I'-}F_4)$ as symmetric triads.
\item Type (I'-$B_r)_s$:\; $(r\geq 3,1\leq s\leq r-1)$: $\tilde{\Sigma}=B_r$ and
\[
\begin{array}{cc}
\Sigma=B_s\cup B_{r-s},&W=(B_r-\Sigma )\cup \{\pm e_i\}.
\end{array}
\]
\begin{align*}
& 0<m(\pm e_1)=\cdots =m(\pm e_s)=n(\pm e_{s+1})=\cdots =n(\pm e_{r})=\mbox{const},\\
& 0<n(\pm e_1)=\cdots =n(\pm e_s)=m(\pm e_{s+1})=\cdots =m(\pm e_r)=\mbox{const},\\
& 0<m(\pm e_i\pm e_j)=m(\pm e_{s+k}\pm e_{s+l})=n(\pm e_p\pm e_{s+q})=\mbox{const}\\
& \qquad (1\leq i<j\leq s,1\leq k<l\leq r-s,1\leq p\leq s,1\leq q\leq r-s).
\end{align*}
It is known $(\mbox{I-}B_r)\sim (\mbox{I'-}B_r)$ as symmetric triads.
\item Type $($I'-$BC_r$-$A_1^r)_s$\; $(1\leq s\leq r-1)$: $\tilde{\Sigma}=BC_r$ and
\[
\begin{array}{cc}
\Sigma=BC_s\cup BC_{r-s},&W=(BC_r-\Sigma )\cup \{\pm e_i\}.
\end{array}
\]
\begin{align*}
& 0<m(\pm 2e_i)=\mbox{const}\quad (1\leq i\leq r),\\
& 0<m(\pm e_1)=\cdots =m(\pm e_s)=n(\pm e_{s+1})=\cdots =n(\pm e_{r})=\mbox{const},\\
& 0<n(\pm e_1)=\cdots =n(\pm e_s)=m(\pm e_{s+1})=\cdots =m(\pm e_r)=\mbox{const},\\
& 0<m(\pm e_i\pm e_j)=m(\pm e_{s+k}\pm e_{s+l})=n(\pm e_p\pm e_{s+q})=\mbox{const}\\
& \qquad (1\leq i<j\leq s,1\leq k<l\leq r-s,1\leq p\leq s,1\leq q\leq r-s).
\end{align*}
It is known $($I-$BC_r$-$A_1^r)$ $\sim$ $($I'-$BC_r$-$A_1^r)_s$
as symmetric triads.
\end{itemize}
\item[(III)] In the case where $\tilde{\Sigma}=\Sigma =W$:
\begin{itemize}
\item Type (III-$A_r$): $\tilde{\Sigma}=\Sigma=W=A_r$.
\[
\begin{array}{cc}
0<m(\lambda)=n(\lambda)=\mbox{const}&(\lambda\in\tilde{\Sigma}).
\end{array}
\]
\item Type (III-$B_r$): $\tilde{\Sigma}=\Sigma=W=B_r$.
When $r\geq 3$ then
\[
\begin{array}{cc}
0<m(\pm e_i)=n(\pm e_i)=\mbox{const},&0<m(\pm e_i\pm e_j)=n(\pm e_i\pm e_j)=\mbox{const}\; (i\not=j).
\end{array}
\]
When $r=2$ then
\[
\begin{array}{lll}
0<m(\pm e_i)=n(\pm e_i)=\mbox{const},&0<m(\pm e_1\pm e_2)=\mbox{const},&
0<n(\pm e_1\pm e_2)=\mbox{const}.
\end{array}
\]
When $r=1$ then $m(\pm e_1)=n(\pm e_1)$.
\item Type (III-$C_r$): $\tilde{\Sigma}=\Sigma=W=C_r$.
\[
\begin{array}{l}
0<m(\pm e_i\pm e_j)=n(\pm e_i\pm e_j)=\mbox{const}\; (i\not= j),\\
0<m(\pm 2e_i)=\mbox{const},\quad 0<n(\pm 2e_i)=\mbox{const}.
\end{array}
\]
\item Type (III-$BC_r$): $\tilde{\Sigma}=\Sigma=W=BC_r$.
When $r\geq 3$, then
\[
\begin{array}{ll}
0<m(\pm e_i)=n(\pm e_i)=\mbox{const},&0<m(\pm e_i\pm e_j)=n(\pm e_i\pm e_j)=\mbox{const},\\
0<m(\pm 2e_i)=\mbox{const},&0<n(\pm 2e_i)=\mbox{const}.
\end{array}
\]
When $r=2$, then
\[
\begin{array}{lll}
0<m(\pm e_i)=n(\pm e_i)=\mbox{const},&0<m(\pm e_1\pm e_2)=\mbox{const},&\\
0<n(\pm e_1\pm e_2)=\mbox{const},& 0<m(\pm 2e_i)=\mbox{const},&0<n(\pm 2e_i)=\mbox{const}.
\end{array}
\]
When $r=1$, then
\[
\begin{array}{lll}
0<m(\pm e_{1})=n(\pm e_{1}), & 0<m(\pm 2e_{1}), & 0<n(\pm 2e_{1}).
\end{array}
\]
\item Type (III-$D_r$): $\tilde{\Sigma}=\Sigma=W=D_r$.
\begin{equation*}
\begin{array}{cc}
0 < m(\lambda)=n(\lambda)=\mbox{const}&(\lambda\in\tilde{\Sigma}).
\end{array}
\end{equation*}
\item Type (III-$E_r$) ($r=6,7,8$): $\tilde{\Sigma}=\Sigma=W=E_{r}$.
\begin{equation*}
\begin{array}{cc}
0 < m(\lambda)=n(\lambda)=\mbox{const}&(\lambda\in\tilde{\Sigma}).
\end{array}
\end{equation*}
\item Type (III-$F_4$): $\tilde{\Sigma}=\Sigma=W=F_4$.
\[
\begin{array}{cc}
0<m(\mbox{short})=n(\mbox{short})=\mbox{const},&
0<m(\mbox{long})=n(\mbox{long})=\mbox{const}.
\end{array}
\]
\item Type (III-$G_2$): $\tilde{\Sigma}=\Sigma=W=G_2$.
\[
\begin{array}{cc}
0<m(\mbox{short})=n(\mbox{short})=\mbox{const},&0<m(\mbox{long})=n(\mbox{long})=\mbox{const}.
\end{array}
\]
\end{itemize}
\end{enumerate}

\subsection{The classification of symmetric triads with multiplicities}\label{sec:symm_classification}

In what follows, we classify the symmetric triads with multiplicities under the equivalence relation $\sim$ based on the above descriptions.

\begin{ex}\rm\label{ex:one-ele}
When
$(\tilde{\Sigma},\Sigma,W)=$ (I-$BC_r$-$B_r$)\; $(r\geq 3)$, (II-$BC_r$), (III-$A_r$), (III-$B_r$)\; $(r\geq 3 \mbox{ or } r=1)$, (III-$BC_r$)\; $(r\geq 3 \mbox{ or } r=1)$, (III-$D_r$),
(III-$E_{r}$) ($r=6,7,8$),
(III-$F_4$), (III-$G_2$), then the isomorphism class of
the corresponding symmetric triad
$(\tilde{\Sigma},\Sigma,W;m,n)$ with multiplicities consists of only one element, that is,
$(\tilde{\Sigma},\Sigma,W;m,n)$ itself.
\end{ex}

\begin{ex}\rm \label{ex:I-B_r}
There is an equivalence relation:
$$
(\mbox{I-}B_r;m,n)\sim (\mbox{I-}B_r;m',n')\sim ((\mbox{I'-}B_r)_{s};m'',n'')\sim ((\mbox{I'-}B_r)_{s};m''',n''').
$$
Here the relation between $m,n$ and $m',n'$ is given by
$$
m'(\pm e_i)=n(\pm e_i),\quad n'(\pm e_i)=m(\pm e_i),\quad m'(\pm e_i\pm e_j)=m(\pm e_i\pm e_j).
$$
The relation between $m,n$ and $m'',n''$ is given by
\begin{align*}
& n''(\pm e_i)=m''(\pm e_{s+j})=m(\pm e_k),\\
& m''(\pm e_i)=n''(\pm e_{s+j})=n(\pm e_k)\quad (1\leq i\leq s,s+1\leq j\leq r-s,1\leq k\leq r),\\
& m''(\pm e_i\pm e_j)=m''(\pm e_{s+k}\pm e_{s+l})=n''(\pm e_p\pm e_{s+q})=m(\pm e_a\pm e_b),\\
& (1\leq i<j\leq s,1\leq k<l\leq r-s,1\leq p\leq s,1\leq q\leq r-s,1\leq a<b\leq r).
\end{align*}
The relation between $m,n$ and $m''',n'''$ is given by
\begin{align*}
& n'''(\pm e_i)=m'''(\pm e_{s+j})=n(\pm e_k),\\
& m'''(\pm e_i)=n'''(\pm e_{s+j})=m(\pm e_k)\quad (1\leq i\leq s,s+1\leq j\leq r-s,1\leq k\leq r),\\
& m'''(\pm e_i\pm e_j)=m'''(\pm e_{s+k}\pm e_{s+l})=n'''(\pm e_p\pm e_{s+q})=m(\pm e_a\pm e_b),\\
& (1\leq i<j\leq s,1\leq k<l\leq r-s,1\leq p\leq s,1\leq q\leq r-s,1\leq a<b\leq r).
\end{align*}
\end{ex}

\begin{ex}\label{ex:IBCA}\rm
There is an equivalence relation:
\[
(\mbox{I-}BC_r\mbox{-}A_1^r;m,n)\!\sim\!(\mbox{I-}BC_r\mbox{-}A_1^r;m',n')\!\sim\!((\mbox{I'-}BC_r\mbox{-}A_1^r)_{s};m'',n'')\!\sim\!((\mbox{I'-}BC_r\mbox{-}A_1^r)_{s};m''',n''').
\]
Here the relation between $m,n$ and $m',n'$ is given by
$$
m'(\pm e_i)=n(\pm e_i),\, n'(\pm e_i)=m(\pm e_i),\, m'(\pm e_i\pm e_j)=m(\pm e_i\pm e_j), m'(\pm 2e_i)=m(\pm 2e_i).
$$
The relation between $m,n$ and $m'',n''$ is given by
\begin{align*}
& m''(\pm 2e_i)=m(\pm 2e_i),\\
& n''(\pm e_i)=m''(\pm e_{s+j})=m(\pm e_k),\\
& m''(\pm e_i)=n''(\pm e_{s+j})=n(\pm e_k)\quad (1\leq i\leq s,s+1\leq j\leq r-s,1\leq k\leq r),\\
& m''(\pm e_i\pm e_j)=m''(\pm e_{s+k}\pm e_{s+l})=n''(\pm e_p\pm e_{s+q})=m(\pm e_a\pm e_b)\\
& (1\leq i<j\leq s,1\leq k<l\leq r-s,1\leq p\leq s,1\leq q\leq r-s,1\leq a<b\leq r).
\end{align*}
The relation between $m,n$ and $m''',n'''$ is given by
\begin{align*}
& m'''(\pm 2e_i)=m(\pm 2e_i),\\
& n'''(\pm e_i)=m'''(\pm e_{s+j})=n(\pm e_k),\\
& m'''(\pm e_i)=n'''(\pm e_{s+j})=m(\pm e_k)\quad (1\leq i\leq s,s+1\leq j\leq r-s,1\leq k\leq r),\\
& m'''(\pm e_i\pm e_j)=m'''(\pm e_{s+k}\pm e_{s+l})=n'''(\pm e_p\pm e_{s+q})=m(\pm e_a\pm e_b)\\
& (1\leq i<j\leq s,1\leq k<l\leq r-s,1\leq p\leq s,1\leq q\leq r-s,1\leq a<b\leq r).
\end{align*}
\end{ex}

\begin{ex}\label{ex.i-bcr-br}\rm $(\mbox{I-}BC_2\mbox{-}B_2)$: $\tilde{\Sigma}=\Sigma=BC_2,W=B_2$.
\begin{equation*}
\begin{array}{ll}
m(\pm e_i)=n(\pm e_i)=\mbox{const},&m(\pm 2e_i)=\mbox{const},\\
m(\pm e_1\pm e_2)=\mbox{const}=:c_1,&n(\pm e_1\pm e_2)=\mbox{const}=:c_2.
\end{array}
\end{equation*}
When $c_1=c_2$, then the isomorphism class of $(\tilde{\Sigma},\Sigma,W;m,n)$ consists of only one element.
When $c_1\not= c_2$, then the isomorphism class of $(\tilde{\Sigma},\Sigma,W;m,n)$
consists of two elements, that is, $(\tilde{\Sigma},\Sigma ,W;m,n)$ itself and
$(\tilde{\Sigma},\Sigma ,W;m',n')$.
Here
\begin{equation*}
\begin{array}{ll}
m'(\pm e_i)=n'(\pm e_i)=m(\pm e_i)=n(\pm e_i),&m'(\pm 2e_i)=m(\pm 2e_i),\\
m'(\pm e_1\pm e_2)=n(\pm e_1\pm e_2),&n'(\pm e_1\pm e_2)=m(\pm e_1\pm e_2).
\end{array}
\end{equation*}
\end{ex}

\begin{ex}\rm
$(\mbox{I-}C_r)$: $\tilde{\Sigma}=\Sigma=C_r, W=D_r$.
When $r\geq 3$, the isomorphism class of $(\tilde{\Sigma},\Sigma,W;m,n)$ consists of two elements, that is,
$(\tilde{\Sigma},\Sigma,W;m,n)$ itself and $(\mbox{I'-}C_r,m',n')$.
Here
\begin{align*}
& m'(\pm e_i\pm e_j)=n'(\pm e_i\pm e_j)=m(\pm e_i\pm e_j)=n(\pm e_i\pm e_j),\\
& n'(\pm 2e_i)=m(\pm 2e_i).
\end{align*}
When $r=2$ there is an equivalence relation
$$
(\mbox{I-}C_2;m,n)\sim (\mbox{I-}C_2;m',n')\sim (\mbox{I'-}C_2;m'',n'')\sim (\mbox{I'-}C_2;m''',n''').
$$
Here the relation between $m,n$ and $m',n'$ is given by
\begin{align*}
& m'(\pm (e_1\pm e_2))=n(\pm (e_1\pm e_2)),\quad n'(\pm (e_1\pm e_2))=m(\pm (e_1\pm e_2)),\\
& m'(\pm 2e_1)=m'(\pm 2e_2)=m(\pm 2e_1)=m(\pm 2e_2).
\end{align*}
The relation between $m,n$ and $m'',n''$ is given by
\begin{align*}
& m''(\pm (e_1-e_2))=n''(\pm (e_1+e_2))=m(\pm (e_1\pm e_2)),\\
& m''(\pm (e_1+e_2))=n''(\pm (e_1-e_2))=n(\pm (e_1\pm e_2)),\\
& n''(\pm 2e_1)=n''(\pm 2e_2)=m(\pm 2e_1)=m(\pm 2e_2).
\end{align*}
The relation between $m,n$ and $m''',n'''$ is given by
\begin{align*}
& m'''(\pm (e_1-e_2))=n'''(\pm (e_1+e_2))=n(\pm (e_1\pm e_2)),\\
& m'''(\pm (e_1+e_2))=n'''(\pm (e_1-e_2))=m(\pm (e_1\pm e_2)),\\
& n'''(\pm 2e_1)=n'''(\pm 2e_2)=m(\pm 2e_1)=m(\pm 2e_2).
\end{align*}
\end{ex}

\begin{ex}\label{ex:if4}\rm
$(\mbox{I-}F_4;m,n)\sim (\mbox{I'-}F_4;m',n')$, where
\begin{align*}
& m'(\mbox{short})=n'(\mbox{short})=m(\mbox{short})=n(\mbox{short}),\\
& m'(\mbox{long}\in\Sigma)=n'(\mbox{long}\in W)=m(\mbox{long}).
\end{align*}
\end{ex}

\begin{ex} \rm When $(\tilde{\Sigma},\Sigma,W)=(\mbox{\rm{III}-}B_2)$, then
\begin{align*}
&m(\pm e_{i})=n(\pm e_{i})=\mbox{const},\\
&m(\pm e_{1}\pm e_{2})=\mbox{const}=:c_{1},\quad n(\pm e_{1}\pm e_{2})=\mbox{const}=:c_{2}.
\end{align*}
\noindent
If $c_1=c_2$ then the isomorphism class of $(\tilde{\Sigma},\Sigma,W;m,n)$ consists of only one element.
If $c_1\not= c_2$ the isomorphism class of $(\tilde{\Sigma},\Sigma,W;m,n)$ consists of $(\tilde{\Sigma},\Sigma,W;m,n)$ itself and
$(B_2,B_2,B_2;m',n')$, where
\begin{align*}
& m'(\pm e_i)=n'(\pm e_i)=m(\pm e_i)=n(\pm e_i), \\
& m'(\pm e_1\pm e_2)=n(\pm e_1\pm e_2),\; n'(\pm e_1\pm e_2)=m(\pm e_1\pm e_2).
\end{align*}
\end{ex}

\begin{ex}\rm \label{ex:III-C_r} When $(\tilde{\Sigma},\Sigma,W)=(\mbox{III-}C_r)\; (r\geq 2)$, $\tilde{\Sigma}=\Sigma =W=C_r$ then
\begin{align*}
&m(\pm e_i\pm e_j)=n(\pm e_i\pm e_j),\\%\quad
&m(\pm 2e_i)=\mbox{const}=:c_{1},\quad n(\pm 2e_i)=\mbox{const}=:c_{2}.
\end{align*}
When $c_{1}=c_{2}$ then the isomorphism class of $(\tilde{\Sigma},\Sigma,W;m,n)$ consists of only one element.
When $c_{1}\neq c_{2}$ then the isomorphism class of $(\tilde{\Sigma},\Sigma,W;m,n)$ consists of $(\tilde{\Sigma},\Sigma,W;m,n)$ itself and
$(C_r,C_r,C_r;m',n')$, where
\begin{align*}
&m'(\pm e_i\pm e_j)=n'(\pm e_i\pm e_j)=m(\pm e_i\pm e_j)=n(\pm e_i\pm e_j),\\
&m'(\pm 2e_i)=n(\pm 2e_i),\quad n'(\pm 2e_i)=m(\pm 2e_i).
\end{align*}
\end{ex}

\begin{ex}\rm \label{ex:III-BC_2}
$(\mbox{III-}BC_2): $
\begin{align*}
& m(\pm e_1\pm e_2)=\mbox{const}=:c_1,\quad n(\pm e_1\pm e_2)=\mbox{const}=:c_2,\\
& m(\pm 2e_i)=\mbox{const},\quad
n(\pm 2e_i)=\mbox{const}, \quad m(\pm e_i)=n(\pm e_i)=\mbox{const}.
\end{align*}
Then the isomorphism class of $(\tilde{\Sigma},\Sigma, W;m,n)$ consists of two elements,
i.e.,
$(\tilde{\Sigma},\Sigma, W;m,n)$ itself and
$(BC_2,BC_2,BC_2;m',n')$, where
\begin{align*}
& m'(\pm e_1\pm e_2)=c_2,\quad n'(\pm e_1\pm e_2)=c_1,\\
& m'(\pm 2e_i)=m(\pm 2e_i),\quad n'(\pm 2e_i)=n(\pm 2e_i),\\
& m'(\pm e_i)=n'(\pm e_i)=m(\pm e_i)=n(\pm e_i).
\end{align*}
\end{ex}

\begin{thm}\label{thm:class_symmI-III}
Examples \ref{ex:one-ele}--\ref{ex:III-BC_2}
exhibit the classification of the isomorphism classes
for the symmetric triads with multiplicities.
\end{thm}

It is a routine work to prove this theorem,
so that we omit the proof.

Here, let us introduce an equivalence relation $\equiv$ for symmetric triads with multiplicities as follows.
Intuitively, 
$(\tilde{\Sigma},\Sigma,W;m,n)\equiv(\tilde{\Sigma}',\Sigma',W';m',n')$
means that these are the same as symmetric triads with multiplicities.

\begin{dfn}\label{dfn:stm_equiv}
Let $(\tilde{\Sigma},\Sigma,W;m,n)$ and
$(\tilde{\Sigma}',\Sigma',W';m',n')$ be symmetric triads with multiplicities of $\mathfrak{a}$ and $\mathfrak{a}'$,
respectively.
We write $(\tilde{\Sigma},\Sigma,W)\equiv (\tilde{\Sigma}',\Sigma',W')$
if there exists an isomorphism $f:\mathfrak{a}\to \mathfrak{a}'$ of root systems
such that $f(\Sigma)=\Sigma'$, $f(W)=W'$.
In addition, if $f$ satisfies
$m(\alpha)=m'(f(\alpha))$, $n(\alpha)=n'(f(\alpha))$ for all $\alpha\in\tilde{\Sigma}$,
then we also write
$(\tilde{\Sigma},\Sigma,W;m,n)\equiv (\tilde{\Sigma}',\Sigma',W';m',n')$.
We call such a mapping $f$ an isomorphism of
$(\tilde{\Sigma},\Sigma,W;m,n)$ and $(\tilde{\Sigma}',\Sigma',W';m',n')$
with respect to $\equiv$.
\end{dfn}

Then we give special isomorphisms for symmetric triads
with multiplicities with respect to $\equiv$.
From (\ref{eqn:root_iso}) we have the following special isomorphisms for symmetric triads:
\[
\begin{array}{l}
(\mbox{I-}B_{1})\equiv(\mbox{I-}C_{1})\equiv(\mbox{III-}A_{1})\equiv(\mbox{III-}B_{1})\equiv(\mbox{III-}C_{1}),\quad
(\mbox{I-}BC_{1}\mbox{-}A_{1}^{1})\equiv(\mbox{I-}BC_{1}\mbox{-}B_{1}),\\
(\mbox{I-}B_{2})\equiv(\mbox{I-}C_{2}),\quad
(\mbox{I'-}B_{2})_{1}\equiv(\mbox{I'-}C_{2}),\quad
(\mbox{III-}A_{3})\equiv (\mbox{III-}D_{3}).
\end{array}
\]
As a consequence we have special isomorphisms for symmetric triads with multiplicities.

\begin{ex}\label{ex:stmSI}
We find the following relations:
\begin{enumerate}[(1)]
\item $(\mbox{I-}B_{2};m,n)\equiv(\mbox{I-}C_{2};m',n')$ with
\[
m(\text{short})=m'(\text{short}),\quad n(\text{short})=n'(\text{short}),
\quad
m(\text{long})=m'(\text{long}).
\]
\item $((\mbox{I'-}B_{2})_{1};m,n)\equiv(\mbox{I'-}C_{2};m',n')$
with 
\[
m(\text{short})=m'(\text{short}),\quad n(\text{short})=n'(\text{short}),
\quad
n(\text{long})=n'(\text{long}).
\]
\item $(\mbox{I'-}BC_{2}\mbox{-}A_{1}^{2};m,n)\equiv
(\mbox{I'-}BC_{2}\mbox{-}A_{1}^{2};m',n')$ with
\[
\begin{array}{l}
m(e_{1})=m'(e_{2}),\quad
n(e_{1})=n'(e_{2}),\quad
m(e_{2})=m'(e_{1}),\quad
n(e_{2})=n'(e_{1}),\\
m(\text{middle})=m'(\text{middle}),\quad
n(\text{middle})=n'(\text{middle}),\\
m(\text{longest})=m'(\text{longest}),\quad
n(\text{longest})=n'(\text{longest}).
\end{array}
\]
\end{enumerate}
\end{ex}

\subsection{Symmetric triads of type (IV) with multiplicities}\label{subsec:ST-IV}

We define a symmetric triad of type (IV).
The motivation of Definition \ref{def:symiv}
comes from the study of compact symmetric triads
$(G,\theta_{1},\theta_{2})$ with $\theta_{1}\sim\theta_{2}$
(see Section \ref{sec:typeIV_CST}).

\begin{dfn}\label{def:symiv}
$(\tilde{\Sigma},\Sigma,W;m,n)$ is a \textit{symmetric triad of type (IV) with multiplicities} of $\mathfrak{a}$ if it satisfies the
following three conditions:
\begin{enumerate}[(1)]
\item $\tilde{\Sigma}$ is an irreducible root system of $\mathfrak{a}$.
\item We define a lattice $\Gamma$ of $\mathfrak{a}$ by
$\Gamma=\{X\in\mathfrak{a}\mid \langle\lambda,X\rangle \in(\pi/2)\mathbb{Z}\; (\lambda\in\tilde{\Sigma})\}$.
Then there exists $Y\in\Gamma$ such that
\begin{align*}
\Sigma=\Sigma_{Y}&:=\{\lambda\in\tilde{\Sigma}\mid \langle\lambda,2Y\rangle \in 2\pi\mathbb{Z}\}=\{\lambda\in\tilde{\Sigma}\mid\langle\lambda,Y\rangle\in\pi\mathbb{Z}\},\\
W=W_{Y}&:=\{\lambda\in\tilde{\Sigma}\mid\langle\lambda,2Y\rangle\in\pi+2\pi\mathbb{Z}\}=\tilde{\Sigma}-\Sigma.
\end{align*}
\item $m$ and $n$ are mappings from $\tilde{\Sigma}$ to $\mathbb{R}_{\geq0}$
such that there
exists a multiplicity $\tilde{m}:\tilde{\Sigma}\to\mathbb{R}_{>0}$
on $\tilde{\Sigma}$ in the sense of Definition \ref{dfn:rootsig}
satisfying the followings:
$$
m(\lambda)=\tilde{m}_\lambda,
\quad n(\lambda)=0
\quad  m(\alpha)=0,\quad n(\alpha)=\tilde{m}_\alpha,
$$
for $\lambda\in\Sigma$ and $\alpha\in W$.
\end{enumerate}
We call $(\tilde{\Sigma},\Sigma,W)$ the \textit{symmetric triad of type (IV)} if we forget multiplicities $m$ and $n$.
We also
call $(\tilde{\Sigma};\tilde{m})$ the \textit{base} of $(\tilde{\Sigma},\Sigma,W;m,n)$.
\end{dfn}

When $(\tilde{\Sigma},\Sigma,W)$
is a symmetric triad of type (IV),
then $\Sigma=\Sigma_{Y}$ is a root system of $\mathrm{span}_{\mathbb{R}}(\Sigma)(\subset \mathfrak{a})$.
We remark that a symmetric triad $(\tilde{\Sigma},\Sigma,W)$ of type (IV) is not a symmetric triad in the sense of Definition~\ref{dfn:symmtriad} since
$\Sigma\cap W=\emptyset$.
In Definition \ref{def:symiv}
if we put $Y=0$,
then the obtained symmetric triad of type (IV) with multiplicities
has the form $(\tilde{\Sigma}, \tilde{\Sigma}, \emptyset; \tilde{m}, 0)$,
which we call the \textit{trivial} symmetric triad of type (IV) with multiplicities.

We define an equivalence relation on the set of symmetric triads of type (IV) with multiplicities as follows.

\begin{dfn}\label{def:symivequiv}
Let
$(\tilde{\Sigma},\Sigma,W;m,n)$ and $(\tilde{\Sigma}',\Sigma',W';m',n')$
be symmetric triads of type (IV) with multiplicities of $\mathfrak{a}$ and $\mathfrak{a}'$, respectively.
We denote by $(\tilde{\Sigma};\tilde{m})$
and $(\tilde{\Sigma}';\tilde{m}')$
the bases of $(\tilde{\Sigma},\Sigma,W;m,n)$
and $(\tilde{\Sigma}',\Sigma',W';m',n')$, respectively.
Two symmetric triads $(\tilde{\Sigma},\Sigma,W;m,n)$
and $(\tilde{\Sigma}',\Sigma',W';m',n')$
are \textit{isomorphic},
if $(\tilde{\Sigma};\tilde{m})\simeq(\tilde{\Sigma}';\tilde{m}')$ in the sense of Definition \ref{dfn:rootmequiv}.
If $(\tilde{\Sigma},\Sigma,W;m,n)$ and $(\tilde{\Sigma}',\Sigma',W';m',n')$ are isomorphic,
then we write $(\tilde{\Sigma},\Sigma,W;m,n) \sim (\tilde{\Sigma}',\Sigma',W';m',n')$.
\end{dfn}

We note that the classification of the isomorphism classes of symmetric triads with multiplicities
of type (IV)
reduces to that of the isomorphism classes of irreducible root systems with multiplicities.
The latter is derived from the classification of the irreducible root systems (cf.~Notation \ref{nota:root_system})
and the condition (\ref{eqn:multi_length}).

For symmetric triads of type (IV) with multiplicities,
we also introduce an equivalence relation $\equiv$ as in Definition \ref{dfn:stm_equiv}.
It is verified that 
$(\tilde{\Sigma},\Sigma,W;m,n)\equiv
(\tilde{\Sigma}',\Sigma',W';m',n')$
implies $(\tilde{\Sigma},\Sigma,W;m,n)\sim(\tilde{\Sigma}',\Sigma',W';m',n')$.
The classification of the $\equiv$-equivalence classes of symmetric triads of type (IV)
\textit{with} multiplicities are easily derived from
that of symmetric triads of type (IV)
\textit{without} multiplicities.

In the sequel, we focus our attention to classify
the set of all equivalence classes with respect to $\equiv$
for symmetric triads of type (IV) with multiplicities.
Then it is sufficient to
classify all possible $\Sigma=\Sigma_{Y}$ for $Y \in \Gamma$.
Let $(\tilde{\Sigma},\Sigma,W)$
be a symmetric triad of type (IV) of $\mathfrak{a}$.
Then there exists $Y\in\Gamma$ such that
$\Sigma=\Sigma_{Y}$.
Denote by $\tilde{W}(\tilde{\Sigma})$ and
$W(\tilde{\Sigma})$
the affine Weyl group and the Weyl group of $\tilde{\Sigma}$,
respectively.
We note that $\tilde{W}(\tilde{\Sigma})$
is a subgroup of $O(\mathfrak{a})\ltimes \mathfrak{a}$ generated by
$\{(s_{\alpha}, (2n\pi/\|\alpha\|^{2})\alpha)\mid n\in\mathbb{Z},\alpha\in\tilde{\Sigma}\}$,
where $s_{\alpha}$ is the linear isometry defined in (\ref{eqn:dfn_reflection}).
The action $(s_{\alpha},(2n\pi/\|\alpha\|^{2})\alpha)$
on $\mathfrak{a}$ is a reflection
with respect to the hyperplane $\{H\in\mathfrak{a} \mid \INN{\alpha}{H}=
n\pi\}$.
\begin{lem}
$\tilde{W}(\tilde{\Sigma})
=W(\tilde{\Sigma})\ltimes
\sum_{\alpha\in\tilde{\Sigma}}2\pi\mathbb{Z}(\alpha/\|\alpha\|^{2})$.
\end{lem}

\begin{proof}
The Weyl group
$W(\tilde{\Sigma})$
acts on the set
$\sum_{\alpha\in\tilde{\Sigma}}2\pi\mathbb{Z}(\alpha/\|\alpha\|^{2})$
invariantly.
We find that
$W(\tilde{\Sigma})\ltimes\sum_{\alpha\in\tilde{\Sigma}}2\pi\mathbb{Z}(\alpha/\|\alpha\|^{2})$ is a subgroup of $O(\mathfrak{a})\ltimes\mathfrak{a}$.
From
$W(\tilde{\Sigma})=\{(w,0)\mid w\in W(\tilde{\Sigma})\}$,
we have
\begin{equation}\label{eqn:sumwtilde1}
W(\tilde{\Sigma})
\subset \tilde{W}(\tilde{\Sigma})
\subset W(\tilde{\Sigma})\ltimes
\sum_{\alpha\in\tilde{\Sigma}}2\pi\mathbb{Z}\dfrac{\alpha}{\|\alpha\|^{2}}.
\end{equation}
For any $H\in\mathfrak{a}$,
we obtain
\begin{align*}
(s_{\alpha},0)\cdot(s_{\alpha}^{-1},-2n\pi(\alpha/\|\alpha\|^{2}))(H)
&=(s_{\alpha},0)(s_{\alpha}^{-1}H-2n\pi(\alpha/\|\alpha\|^{2}))\\
&=H+2n\pi(\alpha/\|\alpha\|^{2})\\
&=(1,2n\pi(\alpha/\|\alpha\|^{2}))(H).
\end{align*}
This yields
$(1,2n\pi(\alpha/\|\alpha\|^{2}))=
(s_{\alpha},0)\cdot(s_{\alpha}^{-1},-2n\pi(\alpha/\|\alpha\|^{2}))
\in\tilde{W}(\tilde{\Sigma})$.
Furthermore, for $\alpha,\beta\in\tilde{\Sigma}$,
we get
\[
(1,2n\pi(\alpha/\|\alpha\|^{2})+2m\pi(\beta/\|\beta\|^{2}))=
(1,2n\pi(\alpha/\|\alpha\|^{2}))\cdot
(1,2m\pi(\beta/\|\beta\|^{2}))
\in\tilde{W}(\tilde{\Sigma}).
\]
Hence the following relation holds:
\begin{equation}\label{eqn:sumwtilde2}
(1,\sum_{\alpha\in\tilde{\Sigma}}2n_{\alpha}\pi\dfrac{\alpha}{\|\alpha\|^{2}})\in \tilde{W}(\tilde{\Sigma})
\quad
(n_{\alpha}\in\mathbb{Z}).
\end{equation}
From
\eqref{eqn:sumwtilde1} and \eqref{eqn:sumwtilde2}
we conclude that
\[
\tilde{W}(\tilde{\Sigma})
=W(\tilde{\Sigma})\ltimes
\sum_{\alpha\in\tilde{\Sigma}}2\pi\mathbb{Z}\dfrac{\alpha}{\|\alpha\|^{2}}.
\]
\end{proof}

Then it is shown that
$\Gamma$ is invariant under the action of $\tilde{W}(\tilde{\Sigma})$.
A connected component
of $\{
H\in\mathfrak{a}
\mid
\INN{\alpha}{H}\not\in\pi\mathbb{Z}
\,(\alpha\in\tilde{\Sigma})\}$ is called a \textit{cell} for $\tilde{\Sigma}$.
It is known that
$\tilde{W}(\tilde{\Sigma})$ permutes the cells.
Furthermore, if we select and fix any cell $Q$,
then $\mathfrak{a}$ is decomposed into
\begin{equation*}
\mathfrak{a} = \bigcup_{w\in\tilde{W}(\tilde{\Sigma})}w\,\overline{Q},
\end{equation*}
where $\overline{Q}$ denotes the closure of $Q$ in $\mathfrak{a}$.
It follows from this decomposition
that for $Y\in \Gamma$,
there exist $X\in\overline{Q}$ and $w=(s,v)\in\tilde{W}(\tilde{\Sigma})$
such that $Y=w(X)$.
Then we obtain that $X$ is in $\Gamma$ and that
\begin{equation*}
\INN{\lambda}{2Y}=\INN{\lambda}{2w(X)}=\INN{s^{-1}(\lambda)}{2X}+2\INN{\lambda}{v},\quad
2\INN{\lambda}{v}\in2\pi\mathbb{Z}
\end{equation*}
for all $\lambda\in\tilde{\Sigma}$.
Since $s\in W(\tilde{\Sigma})$ induces a permutation of $\tilde{\Sigma}$,
the equation above implies that
\begin{equation*}
(\tilde{\Sigma},\Sigma,W)
=(\tilde{\Sigma},\Sigma_{Y},\tilde{\Sigma}-\Sigma_{Y})
=(s\tilde{\Sigma},s\Sigma_{X},s(\tilde{\Sigma}-\Sigma_{X}))
\equiv (\tilde{\Sigma},\Sigma_{X},\tilde{\Sigma}-\Sigma_{X}).
\end{equation*}
Therefore,
it is sufficient to classify all possible $\Sigma=\Sigma_{X}$
for $X \in \Gamma\cap \overline{Q}$.
Here,
we take a useful cell $Q_{0}$ defined as follows.
Let $\tilde{\Pi}$ be a fundamental system of $\tilde{\Sigma}$
and $\tilde{\delta}$ be the highest root of $\tilde{\Sigma}$ with respect to $\tilde{\Pi}$.
Set $\tilde{\Pi}^{*}=\tilde{\Pi}\cup\{\tilde{\delta}\}$ and
$Q_{0}=\{X \in \mathfrak{a} \mid 0 < \INN{\lambda}{X} < \pi\,(\lambda\in\tilde{\Pi}^{*})\}$,
which is a cell for $\tilde{\Sigma}$ whose closure contains the origin $0$.
From the above argument,
we shall classify $\Sigma_{Y}$
for $Y \in \Gamma\cap\overline{Q_{0}}
=\{X\in\mathfrak{a}\mid\INN{\lambda}{X}\in\{0,\pi/2,\pi\} (\lambda\in\tilde{\Pi}^{*})\}$.
We define a vector $\alpha^{i}\in\mathfrak{a}$
($1\leq i \leq r:=\mathrm{rank}\,\tilde{\Sigma}$) by
$\INN{\alpha^{i}}{\alpha_{j}}=\delta_{ij}$
for $\tilde{\Pi}=\{\alpha_{1},\ldots,\alpha_{r}\}$.
Then
\begin{equation*}
\Gamma\cap\overline{Q_{0}}
=\left\{Y=\dfrac{\pi}{2}\sum_{i=1}^{r}n_{i}\alpha^{i}\biggm| n_{i}\in\{0,1,2\},(\langle \tilde{\delta},Y\rangle=)\dfrac{\pi}{2}\sum_{i=1}^{r}m_{i}n_{i}\in\left\{0,\dfrac{\pi}{2},\pi\right\}\right\},
\end{equation*}
where the positive integer $m_{i}$ is given by 
$\tilde{\delta}=\sum_{i=1}^{r}m_{i}\alpha_{i}$.
Since each $m_{i}$ is a positive integer,
we have the following result.

\begin{lem}\label{lem:Y}
$Y\in\Gamma\cap\overline{Q_{0}}$ has an expression mentioned below:
\begin{enumerate}[(1)]
\item if $\INN{\tilde{\delta}}{Y}= \pi/2$, then
$Y=(\pi/2)\alpha^{i}$ for some $i \in \{1,\ldots,r\}$ with $m_{i}=1$,
\item if $\INN{\tilde{\delta}}{Y}= \pi$, then
\begin{equation*}
Y=\begin{cases}
\dfrac{\pi}{2}\alpha^{i} & \text{for some } i\in\{1,\ldots,r\} \text{ with } m_{i}=2\rule{0pt}{4ex},\,(2\text{-}\mathrm{i})\\
\dfrac{\pi}{2}\alpha^{i}+\dfrac{\pi}{2}\alpha^{j} & \text{for some } i\neq j\in\{1,\ldots,r\} \text{ with } m_{i}=m_{j}=1\rule{0pt}{4ex},\,(2\text{-}\mathrm{ii})\\
\pi\alpha^{i} & \text{for some } i\in\{1,\ldots,r\} \text{ with } m_{i}=1,\,(2\text{-}\mathrm{iii})
\end{cases}
\end{equation*}
\item if $\INN{\tilde{\delta}}{Y}=0$, then $Y=0$.
\end{enumerate}
\end{lem}

In the case when (2-iii) or (3),
it is clear that
$\Sigma_{0}=\Sigma_{\pi\alpha^{i}}=\tilde{\Sigma}$ for any $i$
with $m_{i}=1$,
which gives a trivial symmetric triad $(\tilde{\Sigma},\tilde{\Sigma},\emptyset;\tilde{m},0)$ of type (IV) with multiplicities.
The following result is useful to construct
an isomorphism
with respect to $\equiv$
between two symmetric triads
of type (IV) with multiplicities,
and to determine the type of $\Sigma=\Sigma_{Y}$
for $Y \in \Gamma\cap\overline{Q_{0}}$ as in Lemma \ref{lem:Y}.

\begin{thm}\label{thm:SigmaY_str}
Let $Y$ be in $\Gamma\cap\overline{Q_{0}}$.
\begin{enumerate}[(1)]
\item When $Y=(\pi/2)\alpha^{i_{0}}\,(m_{i_{0}}=1)$,
then $\tilde{\Pi}-\{\alpha_{i_{0}}\}$ is a fundamental system of $\Sigma_{Y}$.
In particular,
$\mathrm{rank}\,\Sigma_{Y}=\mathrm{rank}\,\tilde{\Sigma}-1$ holds.
\item
\begin{enumerate}[(i)]
\item When $Y=(\pi/2)\alpha^{i_{0}}\,(m_{i_{0}}=2)$,
then $(\tilde{\Pi}-\{\alpha_{i_{0}}\})\cup\{-\tilde{\delta}\}$
is a fundamental system of $\Sigma_{Y}$.
In particular,
$\mathrm{rank}\,\Sigma_{Y}=\mathrm{rank}\,\tilde{\Sigma}$ holds.
\item When $Y=(\pi/2)(\alpha^{i_{1}}+\alpha^{i_{2}})\,(m_{i_{1}}=m_{i_{2}}=1,i_{1}\neq i_{2})$,
then $(\tilde{\Pi}-\{\alpha_{i_{1}},\alpha_{i_{2}}\})\cup\{-\tilde{\delta}\}$ is a fundamental system of $\Sigma_{Y}$.
In particular,
$\mathrm{rank}\,\Sigma_{Y}=\mathrm{rank}\,\tilde{\Sigma}-1$ holds.
\end{enumerate}
\end{enumerate}
\end{thm}

\begin{proof}
(1) It follows from $m_{i_{0}}=1$
that for any $\alpha =\sum_{i}c_{i}\alpha_{i}\in\tilde{\Sigma}$,
the condition that
$\alpha$ is in $\Sigma_{Y}$ is equivalent to $c_{i_{0}}=0$.
Therefore, $\Sigma_{Y}=\{\sum_{i\neq i_{0}}c_{i}\alpha_{i}\in\tilde{\Sigma}\}$ holds.

(2-i) It is clear that
$(\tilde{\Pi}-\{\alpha_{i_{0}}\})\cup\{-\tilde{\delta}\}$ is
linearly independent over $\mathbb{R}$.
It follows from $m_{i_{0}}=2$
that for any $\alpha =\sum_{i}c_{i}\alpha_{i}\in\tilde{\Sigma}$,
the condition that $\alpha$ be in $\Sigma_{Y}$ is equivalent to $c_{i_{0}}\in\{0,\pm 2\}$.
Then for $\lambda = \sum_{i}c_{i}\alpha_{i}\in \Sigma_{Y}$,
$\lambda$ has an expression $\lambda=\sum_{i\neq i_{0}}c_{i}\alpha_{i}$ or
\begin{equation*}
\lambda = \sum_{i\neq i_{0}}c_{i}\alpha_{i}\pm 2\alpha_{i_{0}}
= \mp \left(\sum_{i\neq i_{0}}(m_{i}\mp c_{i})\alpha_{i}  + (-\tilde{\delta})\right).
\end{equation*}
This expression implies that
$(\tilde{\Pi}-\{\alpha_{i_{0}}\})\cup\{-\tilde{\delta}\}$ is a fundamental system of $\Sigma_{Y}$.

In a similar argument as in (2-i) we can prove the statement (2-ii).
\end{proof}

In the following,
we describe the classification of the non-trivial symmetric triads of type (IV)
and these equivalence classes with respect to $\equiv$.
For this purpose,
we shall follow the notations of irreducible root systems,
their fundamental systems and their highest roots as in
Notation \ref{nota:root_system}.

\begin{ex}\label{ex:a} %%% A_r
Assume that the base of $(\tilde{\Sigma},\Sigma,W)$
is $\tilde{\Sigma}=A_{r}$.
In this case,
we have
$(\tilde{\Sigma},\Sigma_{(\pi/2)\alpha^{i}},W_{(\pi/2)\alpha^{i}})
\equiv
(\tilde{\Sigma},\Sigma_{(\pi/2)\alpha^{r+1-i}},W_{(\pi/2)\alpha^{r+1-i}})$
for $1\leq i\leq r$.
It is also verified that
$(\tilde{\Sigma},\Sigma_{(\pi/2)(\alpha^{j}+\alpha^{k})}, W_{(\pi/2)(\alpha^{j}+\alpha^{k})})
\equiv (\tilde{\Sigma},\Sigma_{(\pi/2)\alpha^{k-j}}, W_{(\pi/2)\alpha^{k-j}})$ for 
$1\leq j < k \leq r$.
Therefore the $\equiv$-equivalence classes of
the non-trivial symmetric triads $(\tilde{\Sigma},\Sigma,W)$
consist of
$(\tilde{\Sigma},\Sigma_{Y},W_{Y})$ with $Y=(\pi/2)\alpha^{l}$ for $1\leq l \leq [(r+1)/2]$,
where $[(r+1)/2]$ denotes the greatest integer less than or equal to $(r+1)/2$.
Moreover, it follows from Theorem \ref{thm:SigmaY_str}, (1)
that $\Sigma_{(\pi/2)\alpha^{l}}$ has the following expression:
\begin{equation}\label{eqn:sigma_pi/ai}
\Sigma_{(\pi/2)\alpha^{l}}
=\{e_{i}-e_{j}\mid
1\leq i < j \leq l\}\cup
\{e_{i}-e_{j}\mid l+1\leq i < j <r+1\},
\end{equation}
where we write $\tilde{\Sigma}=\{e_{i}-e_{j}\mid 1\leq i< j\leq r+1\}$ (cf.~Notation \ref{nota:root_system}).
In particular, we have $
\Sigma_{(\pi/2)\alpha^{l}} \simeq
A_{l-1}\cup A_{r-l}$ for $1\leq l \leq [(r+1)/2]$.
\end{ex}

As shown in Example \ref{ex:a},
for a symmetric triad $(\tilde{\Sigma},\Sigma_{Y},W_{Y})$ with $Y\in\Gamma\cap\overline{Q_{0}}$,
our description of $\Sigma_{Y}$ as a subset of $\tilde{\Sigma}$
like (\ref{eqn:sigma_pi/ai})
is easily verified by using Theorem \ref{thm:SigmaY_str}.
Thus, in the following examples,
we omit the description of $\Sigma_{Y}$ as a subset of $\tilde{\Sigma}$
and show only its type as a root system.

\begin{ex}\label{ex:b} %%% B_r
Assume that the base of $(\tilde{\Sigma},\Sigma,W)$
is $\tilde{\Sigma}=B_{r}$.
Then the $\equiv$-equivalence classes of
the non-trivial symmetric triads $(\tilde{\Sigma},\Sigma,W)$
consist of
$(\tilde{\Sigma},\Sigma_{Y},W_{Y})$ with $Y=(\pi/2)\alpha^{l}$ for $1\leq l \leq r$.
Moreover, $\Sigma_{(\pi/2)\alpha^{l}}$ is isomorphic to $B_{r-1}\,(l=1)$ or $D_{l}(\mbox{long})\cup B_{r-l}\,(2\leq l \leq r)$.
Here, $D_{l}(\mbox{long})$ means that
it is a root system of type $D_{l}$ which consists of long roots in $B_{r}$.
\end{ex}

\begin{ex}\label{ex:c} %%% C_r
Assume that the base of $(\tilde{\Sigma},\Sigma,W)$
is $\tilde{\Sigma}=C_{r}$.
Then the $\equiv$-equivalence classes of
the non-trivial symmetric triads $(\tilde{\Sigma},\Sigma,W)$
consist of
$(\tilde{\Sigma},\Sigma_{Y},W_{Y})$ with $Y=(\pi/2)\alpha^{l}$ for $1\leq l \leq r$.
Moreover,
$\Sigma_{(\pi/2)\alpha^{l}}$ is isomorphic to $C_{l}\cup C_{r-l}\,(1\leq l \leq r-1)$
or $A_{r-1}(\mbox{short})\,(l=r)$.
\end{ex}

\begin{ex}\label{ex:d} %%% D_r
Assume that the base of $(\tilde{\Sigma},\Sigma,W)$
is $\tilde{\Sigma}=D_{r}$.
Then we have
\begin{align*}
(\tilde{\Sigma},\Sigma_{(\pi/2)\alpha^{r-1}},W_{(\pi/2)\alpha^{r-1}})
&\equiv (\tilde{\Sigma},\Sigma_{(\pi/2)\alpha^{r}},W_{(\pi/2)\alpha^{r}})\\
&\equiv (\tilde{\Sigma},\Sigma_{(\pi/2)(\alpha^{1}+\alpha^{r-1})},W_{(\pi/2)(\alpha^{1}+\alpha^{r-1})})\\
&\equiv (\tilde{\Sigma},\Sigma_{(\pi/2)(\alpha^{1}+\alpha^{r})},W_{(\pi/2)(\alpha^{1}+\alpha^{r})}),
\end{align*}
and
\[
(\tilde{\Sigma},\Sigma_{(\pi/2)\alpha^{1}},W_{(\pi/2)\alpha^{1}})
\equiv (\tilde{\Sigma},\Sigma_{(\pi/2)(\alpha^{r-1}+\alpha^{r})},W_{(\pi/2)(\alpha^{r-1}+\alpha^{r})}).
\]
Therefore the $\equiv$-equivalence classes of
the non-trivial symmetric triads $(\tilde{\Sigma},\Sigma,W)$
consist of
$(\tilde{\Sigma},\Sigma_{Y},W_{Y})$ with $Y=(\pi/2)\alpha^{l}$ for $1\leq l \leq r-1$.
Moreover,
$\Sigma_{(\pi/2)\alpha^{l}}$ is isomorphic to $D_{r-1}$ ($l=1$),
$D_{l}\cup D_{r-l}\,(2\leq l \leq r-2)$ or $A_{r-1}\,(l=r-1)$.
\end{ex}

\begin{ex}\label{ex:bc} %%% BC_r
Assume that the base of $(\tilde{\Sigma},\Sigma,W)$
is $\tilde{\Sigma}=BC_{r}$.
Then the $\equiv$-equivalence classes of
the non-trivial symmetric triads $(\tilde{\Sigma},\Sigma,W)$
consist of
$(\tilde{\Sigma},\Sigma_{Y},W_{Y})$ with $Y=(\pi/2)\alpha^{l}$ for $1\leq l \leq r$.
Moreover,
$\Sigma_{(\pi/2)\alpha^{l}}$
is isomorphic to $C_{l}(\mbox{middle, long})\cup BC_{r-l}\,(1\leq l \leq r)$.
\end{ex}

\begin{ex}\label{ex:e6} %%% E_6
Assume that the base of $(\tilde{\Sigma},\Sigma,W)$
is $\tilde{\Sigma}=E_{6}$.
Then the $\equiv$-equivalence classes of
the non-trivial symmetric triads $(\tilde{\Sigma},\Sigma,W)$
consist of
$(\tilde{\Sigma},\Sigma_{Y},W_{Y})$ with $Y=(\pi/2)\alpha^{l}$ for $l=1,2$.
Moreover,
$\Sigma_{(\pi/2)\alpha^{1}}, \Sigma_{(\pi/2)\alpha^{2}}$ 
are isomorphic to $D_{5}, A_{1}\cup A_{5}$, respectively.

\begin{proof}
It is sufficient to show the followings:
\begin{align*}
&(\tilde{\Sigma},\Sigma_{(\pi/2)\alpha^{1}},W_{(\pi/2)\alpha^{1}})
\equiv(\tilde{\Sigma},\Sigma_{(\pi/2)\alpha^{6}},W_{(\pi/2)\alpha^{6}})
\equiv(\tilde{\Sigma},\Sigma_{(\pi/2)(\alpha^{1}+\alpha^{6})},W_{(\pi/2)(\alpha^{1}+\alpha^{6})}),\\
&(\tilde{\Sigma},\Sigma_{(\pi/2)\alpha^{2}},W_{(\pi/2)\alpha^{2}})
\equiv (\tilde{\Sigma},\Sigma_{(\pi/2)\alpha^{3}},W_{(\pi/2)\alpha^{3}})
\equiv (\tilde{\Sigma},\Sigma_{(\pi/2)\alpha^{5}},W_{(\pi/2)\alpha^{5}}).
\end{align*}
We define two linear isometries $f$ and $\tilde{f}$ on $\tilde{\Sigma}$ as follows:
\begin{align*}
&f:\tilde{\Sigma}\to\tilde{\Sigma};
\alpha_{1}\mapsto\alpha_{6},
\alpha_{2}\mapsto\alpha_{2},
\alpha_{3}\mapsto\alpha_{5},
\alpha_{4}\mapsto\alpha_{4},
\alpha_{5}\mapsto\alpha_{3},
\alpha_{6}\mapsto\alpha_{1},\\
&\tilde{f}:\tilde{\Sigma}\to\tilde{\Sigma};
-\tilde{\delta}\mapsto\alpha_{6},
\alpha_{2}\mapsto\alpha_{5},
\alpha_{3}\mapsto\alpha_{3},
\alpha_{4}\mapsto\alpha_{4},
\alpha_{5}\mapsto\alpha_{2},
\alpha_{6}\mapsto-\tilde{\delta}.
\end{align*}
Then,
\[
\begin{cases}
(\tilde{\Sigma},\Sigma_{(\pi/2)\alpha^{1}},W_{(\pi/2)\alpha^{1}})
\equiv
(\tilde{\Sigma},\Sigma_{(\pi/2)\alpha^{6}},W_{(\pi/2)\alpha^{6}}),\\
(\tilde{\Sigma},\Sigma_{(\pi/2)\alpha^{3}},W_{(\pi/2)\alpha^{3}})
\equiv(\tilde{\Sigma},\Sigma_{(\pi/2)\alpha^{5}},W_{(\pi/2)\alpha^{5}})
\end{cases}
\]
via $f$.
We also obtain
\[
\begin{cases}
(\tilde{\Sigma},\Sigma_{(\pi/2)\alpha^{1}},W_{(\pi/2)\alpha^{1}})
\equiv
(\tilde{\Sigma},\Sigma_{(\pi/2)(\alpha^{1}+\alpha^{6})},W_{(\pi/2)(\alpha^{1}+\alpha^{6})}),\\
(\tilde{\Sigma},\Sigma_{(\pi/2)\alpha^{2}},W_{(\pi/2)\alpha^{2}})
\equiv
(\tilde{\Sigma},\Sigma_{(\pi/2)\alpha^{5}},W_{(\pi/2)\alpha^{5}})
\end{cases}
\]
via $\tilde{f}$.
The types of $\Sigma_{(\pi/2)\alpha^{1}}$ and $\Sigma_{(\pi/2)\alpha^{2}}$
are easily derived from Theorem \ref{thm:SigmaY_str}.
Hence we get the assertion.
\end{proof}
\end{ex}

\begin{ex}\label{ex:e7} %%% E_7
Assume that the base of $(\tilde{\Sigma},\Sigma,W)$
is $\tilde{\Sigma}=E_{7}$.
In this case,
we have $(\tilde{\Sigma},\Sigma_{(\pi/2)\alpha^{1}},W_{(\pi/2)\alpha^{1}})
\equiv (\tilde{\Sigma},\Sigma_{(\pi/2)\alpha^{6}},W_{(\pi/2)\alpha^{6}})$.
Therefore the $\equiv$-equivalence classes of
the non-trivial symmetric triads $(\tilde{\Sigma},\Sigma,W)$
consist of
$(\tilde{\Sigma},\Sigma_{Y},W_{Y})$ with $Y=(\pi/2)\alpha^{l}$ for $l=1,2,7$.
Moreover,
$\Sigma_{(\pi/2)\alpha^{1}},\Sigma_{(\pi/2)\alpha^{2}},\Sigma_{(\pi/2)\alpha^{7}}$
are isomorphic to $A_{1}\cup D_{6}, A_{7},E_{6}$,
respectively.
\end{ex}

\begin{ex}\label{ex:e8} %%% E_8
Assume that the base of $(\tilde{\Sigma},\Sigma,W)$
is $\tilde{\Sigma}=E_{8}$.
Then the $\equiv$-equivalence classes of
the non-trivial symmetric triads $(\tilde{\Sigma},\Sigma,W)$
consist of
$(\tilde{\Sigma},\Sigma_{Y},W_{Y})$ with $Y=(\pi/2)\alpha^{l}$ for $l=1,8$.
Moreover,
$\Sigma_{(\pi/2)\alpha^{1}},\Sigma_{(\pi/2)\alpha^{8}}$
are isomorphic to
$D_{8}$ and $A_{1}\cup E_{7}$, respectively.
\end{ex}

\begin{ex}\label{ex:f4} %%% F_4
Assume that the base of $(\tilde{\Sigma},\Sigma,W)$
is $\tilde{\Sigma}=F_{4}$.
Then the $\equiv$-equivalence classes of
the non-trivial symmetric triads $(\tilde{\Sigma},\Sigma,W)$
consist of
$(\tilde{\Sigma},\Sigma_{Y},W_{Y})$ with $Y=(\pi/2)\alpha^{l}$ for $l=1,4$.
Moreover,
$\Sigma_{(\pi/2)\alpha^{1}},\Sigma_{(\pi/2)\alpha^{4}}$
are isomorphic to
$A_{1}(\mbox{long})\cup C_{3}$ and $B_{4}$, respectively.
\end{ex}

\begin{ex}\label{ex:g2} %%% G_2
Assume that the base of $(\tilde{\Sigma},\Sigma,W)$
is $\tilde{\Sigma}=G_{2}$.
Then the $\equiv$-equivalence classes of
the non-trivial symmetric triads $(\tilde{\Sigma},\Sigma,W)$
consist of
$(\tilde{\Sigma},\Sigma_{(\pi/2)\alpha^{2}},W_{(\pi/2)\alpha^{2}})$.
Moreover, $\Sigma_{(\pi/2)\alpha^{2}}\simeq A_{1}(\mbox{long})\cup A_{1}(\mbox{short})$ holds.
\end{ex}

From the above argument, we conclude:

\begin{thm}\label{thm:classSTmIV}
Table \ref{table:ivsymm} mentioned below
gives the classification of 
the equivalence classes of non-trivial symmetric triads of type (IV) with respect to $\equiv$.
\end{thm}

We have the following result by means of the classification of symmetric triads of type (IV) as in Table \ref{table:ivsymm}.

\begin{cor}
Let $(\tilde{\Sigma},\Sigma,W)$ and $(\tilde{\Sigma}',\Sigma',W')$
be symmetric triads of type (IV).
Then $(\tilde{\Sigma},\Sigma,W) \equiv (\tilde{\Sigma}',\Sigma',W')$
if and only if the types of $\tilde{\Sigma}$ and $\Sigma$
coincide with those of $\tilde{\Sigma}'$ and $\Sigma'$, respectively.
\end{cor}

\begin{rem}
(1) Symmetric triads of types (I)--(III) satisfy $\operatorname{rank}\tilde{\Sigma}=\operatorname{rank}\Sigma$.
On the other hand, from the classification of symmetric triads of type (IV)
we have $\RANK\tilde{\Sigma}-\RANK\Sigma\in\{0,1\}$.
(2) We give special isomorphisms for symmetric triads of type (IV)
with respect to $\equiv$:
\[
\begin{array}{cc}
(B_{2},B_{1},W)\equiv(C_{2},A_{1},W),&
(B_{2},D_{2},W)\equiv(C_{2},C_{1}\cup C_{1},W),\\
(A_{3},A_{2},W)\equiv(D_{3},A_{2},W),&
(A_{3},A_{1}\cup A_{1},W)\equiv(D_{3},D_{2},W).
\end{array}
\]
\end{rem}

\begin{table}[!!ht]
\caption{The classification of non-trivial symmetric triads of type (IV)}\label{table:ivsymm}
\centering
\renewcommand{\arraystretch}{1.8}
\begin{tabular}{cccc}
\hline
\hline
Type\footnote{} & $\tilde{\Sigma}$ & $\Sigma$ & Remark\\
\hline
\hline
(IV'-$A_{r}$)$_{l}$ & $A_{r}$ & $A_{l-1}\cup A_{r-l}$ & $1\leq l \leq [(r+1)/2]$\\
\hline
(IV-$B_{r}$)$_{l}$ & \multirow{2}{*}{$B_{r}$} & $D_{l}(\mbox{\text{long}})\cup B_{r-l}$ & $2\leq l \leq r$\\
%\cline{1-1}\cline{3-4}
(IV'-$B_{r}$) & & $B_{r-1}$ & \\
\hline
(IV-$C_{r}$)$_{l}$ & \multirow{2}{*}{$C_{r}$} & $C_{l}\cup C_{r-l}$ & $1 \leq l \leq r-1$\\
%\cline{1-1}\cline{3-4}
(IV'-$C_{r}$) & & $A_{r-1}(\mbox{\text{short}})$ & \\
\hline
(IV-$D_{r}$)$_{l}$ & \multirow{3}{*}{$D_{r}$} & $D_{r-l}\cup D_{l}$ & $2\leq l \leq r-2$ \\
%\cline{1-1}\cline{3-4}
(IV'-$D_{r}$-$D$) & & $D_{r-1}$ &\\
%\cline{1-1}\cline{3-4}
(IV'-$D_{r}$-$A$) & & $A_{r-1}$ &\\
\hline
(IV-$BC_{r}$)$_{l}$ & $BC_{r}$ & $C_{l}(\mbox{\text{middle, long}})\cup BC_{r-l}$ & $1\leq l \leq r$\\
\hline
(IV-$E_{6}$) & \multirow{2}{*}{$E_{6}$} & $A_{1}\cup A_{5}$ & \\
%\cline{1-1}\cline{3-4}
(IV'-$E_{6}$) &  & $D_{5}$ & \\
\hline
(IV-$E_{7}$-$AD$) & \multirow{3}{*}{$E_{7}$} & $A_{1}\cup D_{6}$ & \\
%\cline{1-1}\cline{3-4}
(IV-$E_{7}$-$A$) &  & $A_{7}$ & \\
%\cline{1-1}\cline{3-4}
(IV'-$E_{7}$) &  & $E_{6}$ & \\
\hline
(IV-$E_{8}$-$D$) & \multirow{2}{*}{$E_{8}$} & $D_{8}$ & \\
%\cline{1-1}\cline{3-4}
(IV-$E_{8}$-$AE$) & & $A_{1}\cup E_{7}$ & \\
\hline
(IV-$F_{4}$-$AC$) & \multirow{2}{*}{$F_{4}$} & $A_{1}(\mbox{\text{long}})\cup C_{3}$ & \\
%\cline{1-1}\cline{3-4}
(IV-$F_{4}$-$B$) & & $B_{4}$ & \\
\hline
(IV-$G_{2}$) & $G_{2}$ & $A_{1}(\mbox{\text{short}})\cup A_{1}(\mbox{\text{long}})$ &\\
\hline
\hline
\end{tabular}
\renewcommand{\arraystretch}{1.0}
\end{table}
\footnotetext{
\textit{Notation:}
We use the notations (IV-$X_{r}$-$Y_{r}$)
and (IV'-$X_{r}$-$Y_{r-1}$) for symmetric triads of type (IV) with
$\tilde{\Sigma}=X_{r}$ and $\Sigma=Y_{k}$ ($k=r, r-1$).
We use the the symbol ' in IV'
if the corresponding symmetric triad of type (IV)
satisfies $\mathrm{rank}\,\tilde{\Sigma}=\mathrm{rank}\,\Sigma-1$.
Otherwise, we omit it.
We shall omit to write $Y_{k}$ if it is uniquely determined from $\tilde{\Sigma}$.}

\section{Commutative compact symmetric triads and symmetric triads with multiplicities}\label{sec:cst_symm}

The purpose of this section is to study
symmetric triads with multiplicities
constructed from commutative compact symmetric triads $(G,\theta_{1},\theta_{2})$
with $\theta_{1}\not\sim\theta_{2}$.
We first recall this construction,
which was given by the second author \cite{Ikawa}.
Then, we will show that
any two isomorphic commutative compact symmetric triads
with respect to $\sim$ as in Definition \ref{dfn:CST_sim}
correspond to
the same symmetric triad with multiplicities
up to $\sim$
as in Definition \ref{dfn:e-relation}
(see Proposition \ref{pro:cstsim_symmsim}).
Second,
our concern is to determine
the corresponding symmetric triad 
$(\tilde{\Sigma},\Sigma,W;m,n)$
with multiplicities.
The authors \cite{BI} classified compact symmetric triads
with respect to $\sim$ in terms of the notion of double Satake diagrams.
Each isomorphism class
of a commutative compact symmetric triad
is characterized by the double Satake diagrams.
Then, we will give a method to determine 
$(\tilde{\Sigma},\Sigma,W;m,n)$
by using double Satake diagrams
and determine it based on this method
(see Proposition \ref{pro:symm_determ_base}
and Theorem \ref{thm:determ_symmetrictriads}).
As applications,
we will show the converse of Proposition \ref{pro:cstsim_symmsim}
at the Lie algebra level
(see Corollary \ref{cor:simsim_simplyconnected}
for the precise statement).

Third, we will give a similar result
for the commutative compact symmetric triads
$(G,\theta_{1},\theta_{2})$
with $\theta_{1}\sim\theta_{2}$.

Finally,
we will give the classification for
commutative compact symmetric triads with respect
to $\equiv$
in terms of symmetric triads with multiplicities whether $\theta_{1}\sim\theta_{2}$ holds or not.

\subsection{Symmetric triads with multiplicities for commutative compact symmetric triads}\label{sec:ccst_symm}

We first recall
the construction
of the symmetric triad
with multiplicities of
a commutative compact symmetric
triad due to \cite{Ikawa}.
Let $G$ be a compact connected semisimple Lie group with Lie algebra $\mathfrak{g}$.
Let $(G,\theta_{1},\theta_{2})$ be a
commutative compact symmetric triad.
Fix an invariant inner product on $\mathfrak{g}$,
which we write $\INN{\cdot}{\cdot}$.
We write the differential
of $\theta_{i}$
($i=1,2$)
at the identity element of $G$
as the same symbol $\theta_{i}$,
if there is no confusion.
By the commutativity of $\theta_{1}$ and $\theta_{2}$,
we get the simultaneous eigenspace decomposition
of $\mathfrak{g}$ for $(\theta_{1},\theta_{2})$,
which we write
\begin{equation}
\mathfrak{g}
=(\mathfrak{k}_{1}\cap\mathfrak{k}_{2})\oplus
(\mathfrak{m}_{1}\cap\mathfrak{m}_{2})\oplus
(\mathfrak{k}_{1}\cap\mathfrak{m}_{2})\oplus
(\mathfrak{m}_{1}\cap\mathfrak{k}_{2}),
\end{equation}
where $\mathfrak{k}_{i}$ and $\mathfrak{m}_{i}$
denote the $(+1)$-eigenspace and the $(-1)$-eigenspace
of $\theta_{i}$, respectively.
Let $\mathfrak{a}$
be a maximal abelian subspace of $\mathfrak{m}_{1}\cap\mathfrak{m}_{2}$.
We denote by $\mathfrak{g}^{\mathbb{C}}$
the complexification of $\mathfrak{g}$.
For each $\alpha\in\mathfrak{a}$,
we define the complex subspace
$\mathfrak{g}(\mathfrak{a},\alpha)$ of $\mathfrak{g}^{\mathbb{C}}$
as follows:
\begin{equation}
\mathfrak{g}(\mathfrak{a},\alpha)=\{
X\in\mathfrak{g}^{\mathbb{C}}
\mid [H,X]
=\sqrt{-1}\INN{\alpha}{H}X,\,
H\in\mathfrak{a}
\}.
\end{equation}
We put
\begin{equation}
\tilde{\Sigma}
=\{\alpha\in\mathfrak{a}-\{0\}\mid
\mathfrak{g}(\mathfrak{a},\alpha)\neq\{0\}\}.
\end{equation}
Then we have the following
decomposition of $\mathfrak{g}^{\mathbb{C}}$:
\begin{equation}
\mathfrak{g}^{\mathbb{C}}
=\mathfrak{g}(\mathfrak{a},0)
\oplus\sum_{\alpha\in\tilde{\Sigma}}\mathfrak{g}(\mathfrak{a},\alpha).
\end{equation}
For each $\alpha\in\mathfrak{a}$,
$\epsilon\in\{1, -1\}$,
we define the complex subspace $\mathfrak{g}(\mathfrak{a},\alpha,\epsilon)$ of $\mathfrak{g}^{\mathbb{C}}$ by
\begin{equation}
\mathfrak{g}(\mathfrak{a},\alpha,\epsilon)
=\{X\in\mathfrak{g}(\mathfrak{a},\alpha)
\mid
\theta_{1}\theta_{2}(X)=\epsilon X\}.
\end{equation}
Since $\mathfrak{g}(\mathfrak{a},\alpha)$
is $\theta_{1}\theta_{2}$-invariant,
$\mathfrak{g}(\mathfrak{a},\alpha)$
is decomposed into
\begin{equation}
\mathfrak{g}(\mathfrak{a},\alpha)
=\mathfrak{g}(\mathfrak{a},\alpha,1)
\oplus\mathfrak{g}(\mathfrak{a},\alpha,-1).
\end{equation}
We define the subsets
$\Sigma$ and $W$ of $\tilde{\Sigma}$ as follows:
\begin{equation}
\Sigma=\{\alpha \in \tilde{\Sigma} \mid
\mathfrak{g}(\mathfrak{a},\alpha,1)\neq
\{0\}\},
\quad
W=\{
\alpha\in\tilde{\Sigma}
\mid \mathfrak{g}(\mathfrak{a},\alpha,-1)\neq
\{0\}
\}.
\end{equation}
For each $\alpha\in\tilde{\Sigma}$,
we write the dimensions
of $\mathfrak{g}(\mathfrak{a},\alpha,1)$
and
$\mathfrak{g}(\mathfrak{a},\alpha,-1)$
as $m(\alpha)$
and $n(\alpha)$, respectively.
Under the above settings,
it follows from
\cite[Theorem 4.33, (1)]{Ikawa}
that,
if $G$ is simple
and
$\theta_{1}\not\sim\theta_{2}$
holds, then
$(\tilde{\Sigma},\Sigma,W;m,n)$
satisfies the axiom
of symmetric triads with multiplicities stated in Definitions
\ref{dfn:symmtriad} and
\ref{dfn:multi}.
Furthermore,
it can be verified that
$(\tilde{\Sigma},\Sigma,W;m,n)$
is independent for the choice
of $\mathfrak{a}$ up to isomorphism
for $\equiv$ as in Definition \ref{dfn:stm_equiv}.
We call
$(\tilde{\Sigma},\Sigma,W;m,n)$
the \textit{symmetric triad
with multiplicities}
of $\mathfrak{a}$
corresponding to $(G,\theta_{1},\theta_{2})$.
By this construction,
$(G,\theta_{1},\theta_{2})$
and $(G,\theta_{2},\theta_{1})$
give the same one.

Next,
we show that
the isomorphism class
of a commutative compact symmetric triad
with respect to $\sim$
determines
that of the corresponding
symmetric triad
with multiplicities.
Namely, we prove the following
proposition.

\begin{pro}\label{pro:cstsim_symmsim}
Let 
$(G,\theta_{1},\theta_{2})$,
$(G,\theta_{1}',\theta_{2}')$
be two commutative compact symmetric triads with 
$\theta_{1}\not\sim\theta_{2}$,
$\theta_{1}'\not\sim\theta_{2}'$.
We denote by
$(\tilde{\Sigma},\Sigma,W;m,n)$
and $(\tilde{\Sigma}',\Sigma',W';m',n')$ the symmetric triad with multiplicities
corresponding to $(G,\theta_{1},\theta_{2})$ and
$(G,\theta_{1}',\theta_{2}')$, respectively.
If $(G,\theta_{1},\theta_{2})
\sim (G,\theta_{1}',\theta_{2}')$,
then
$(\tilde{\Sigma},\Sigma,W;m,n)
\sim(\tilde{\Sigma}',\Sigma',W';m',n')
$ holds in the sense of Definition
\ref{dfn:e-relation}.
\end{pro}

We note that
this proposition gives a refinement
of \cite[Theorem 4.33, (2)]{Ikawa}.
In order to
give our proof of Proposition \ref{pro:cstsim_symmsim},
we prepare the following lemma.

\begin{lem}\label{lem:cstequiv_symmequiv}
In the same settings
as in Proposition
\ref{pro:cstsim_symmsim},
if $(G,\theta_{1},\theta_{2})\equiv 
(G,\theta_{1}',\theta_{2}')$
then
we have $(\tilde{\Sigma},\Sigma,W;m,n)\equiv
(\tilde{\Sigma}',\Sigma',W';m',n')$.
\end{lem}

\begin{proof}
By the assumption,
there exists $\varphi\in\mathrm{Aut}(G)$
satisfying $\theta_{i}'=\varphi\theta_{i}\varphi^{-1}$
($i=1,2$).
We set $\mathfrak{a}''=d\varphi(\mathfrak{a})$, which is a maximal abelian subspace of $\mathfrak{m}_{1}'\cap\mathfrak{m}_{2}'$.
We denote by
$(\tilde{\Sigma}'',\Sigma'',W'';m'',n'')$ the symmetric triad of $\mathfrak{a}''$ with multiplicities
constructed
from $(G,\theta_{1}',\theta_{2}')$.
Then $d\varphi|_{\mathfrak{a}}:\mathfrak{a}\to\mathfrak{a}''$
gives an isomorphism of symmetric triads with multiplicities
between
$(\tilde{\Sigma},\Sigma,W;m,n)$ and
$(\tilde{\Sigma}'',\Sigma'',W'';m'',n'')$
with respect to $\equiv$.
Thus, we have
$(\tilde{\Sigma},\Sigma,W;m,n)\equiv
(\tilde{\Sigma}'',\Sigma'',W'';m'',n'')
\equiv
(\tilde{\Sigma}',\Sigma',W';m',n')$.
\end{proof}

We are ready to prove
Proposition
\ref{pro:cstsim_symmsim}.

\begin{proof}[Proof of Proposition \ref{pro:cstsim_symmsim}]
First we recall
the isomorphism
between $(\tilde{\Sigma},\Sigma,W)$
and $(\tilde{\Sigma}',\Sigma',W')$
as in the proof
of \cite[Theorem 4.33, (2)]{Ikawa}.
By definition,
$(G,\theta_{1}',\theta_{2}')\sim(G,\theta_{1},\theta_{2})$
implies that
there exists $g\in G$
satisfying $(G,\theta_{1}',\theta_{2}')\equiv
(G,\theta_{1},\tau_{g}\theta_{2}\tau_{g}^{-1})$.
It follows from
Theorem \ref{thm:Hermann}
that
there exist
$k_{i}\in K_{i}$
($i=1,2$)
and $Y\in\mathfrak{a}$
satisfying $g=k_{1}\exp(Y)k_{2}$,
from which $\tau_{g}=\tau_{k_{1}}\tau_{\exp(Y)}\tau_{k_{2}}$ holds.
By the commutativity $\tau_{k_{i}}\theta_{i}=\theta_{i}\tau_{k_{i}}$
($i=1,2$),
we have 
\begin{equation}
(G,\theta_{1},\tau_{g}\theta_{2}\tau_{g}^{-1})
\equiv
(G,\tau_{k_{1}}^{-1}\theta_{1}\tau_{k_{1}},
\tau_{\exp(Y)}\tau_{k_{2}}\theta_{2}\tau_{k_{2}}^{-1}\tau_{\exp(Y)}^{-1})
=(G,\theta_{1},\tau_{\exp(Y)}\theta_{2}\tau_{\exp(Y)}^{-1}).
\end{equation}
Hence,
without loss of generalities,
we assume that $\theta_{1}'=\theta_{1}$
and $\theta_{2}'=\tau_{\exp(Y)}\theta_{2}\tau_{\exp(Y)}^{-1}$
and that the corresponding symmetric triad
$(\tilde{\Sigma}',\Sigma',W';m',n')$ with multiplicities
is constructed from the maximal abelian subspace
$\mathfrak{a}'=e^{\mathrm{ad}(Y)}\mathfrak{a}=\mathfrak{a}$
of $\mathfrak{m}_{1}'\cap\mathfrak{m}_{2}'=\mathfrak{m}_{1}\cap
e^{\mathrm{ad}(Y)}\mathfrak{m}_{2}$
by Lemma \ref{lem:cstequiv_symmequiv}.
The commutativity of $\theta_{1}$
and $\tau_{\exp(Y)}\theta_{2}\tau_{\exp(Y)}^{-1}$
implies that $\exp(4Y)$ is an element of the center of $G$.
Since we have $\mathrm{Ad}(\exp(4Y))=1$,
the element
$Y$ is in the lattice $\Gamma$
defined in \eqref{eqn:Gamma}.
Therefore,
$f=\mathrm{id}_{\mathfrak{a}}$
and $Y\in\Gamma$
give the isomorphism between
$(\tilde{\Sigma},\Sigma,W)$
and $(\tilde{\Sigma}',\Sigma',W')$
with respect to $\sim$.

Second, we show that
the above isomorphism also gives that between
$(\tilde{\Sigma},\Sigma,W;m,n)$
and $(\tilde{\Sigma}',\Sigma',W';m',n')$.
From the above argument, we have $\tilde{\Sigma}'=\tilde{\Sigma}$.
For each $\alpha\in\tilde{\Sigma}$,
we obtain $\mathfrak{g}(\mathfrak{a}',\alpha,\epsilon)
=\mathfrak{g}(\mathfrak{a},\alpha,e^{\sqrt{-1}\INN{\alpha}{2Y}}\epsilon)$,
from which, if $\INN{\alpha}{2Y} \in \pi+2\pi\mathbb{Z}$
holds, then we get
\begin{align}
m'(\alpha)
&=\dim_{\mathbb{C}}\mathfrak{g}(\mathfrak{a}',\alpha,1)=\dim_{\mathbb{C}}\mathfrak{g}(\mathfrak{a},\alpha,-1)=n(\alpha),\\
n'(\alpha)
&=\dim_{\mathbb{C}}\mathfrak{g}(\mathfrak{a}',\alpha,-1)=\dim_{\mathbb{C}}\mathfrak{g}(\mathfrak{a},\alpha,1)=m(\alpha).
\end{align}
A similar calculation yields
$m'(\alpha)=m(\alpha)$
and $n'(\alpha)=n(\alpha)$
for $\alpha\in\tilde{\Sigma}$ with $\INN{\alpha}{2Y}\in 2\pi\mathbb{Z}$.
Thus, we have completed the proof.
\end{proof}

We note that there exist
two commutative compact symmetric triads
$(G,\theta_{1},\theta_{2})\sim(G,\theta_{1}',\theta_{2}')$
satisfying $(\tilde{\Sigma},\Sigma,W;m,n)\not\equiv
(\tilde{\Sigma}',\Sigma',W';m',n')$.
Such commutative compact symmetric triads
will be found in Example \ref{ex:E6Sp4SU6SU2_equiv}.

The following lemma is fundamental in our argument.

\begin{lem}\label{lem:CSTSYMM_SYMM'_CST'}
Let $(G,\theta_{1},\theta_{2})$
be a commutative compact symmetric triad with $\theta_{1}\not\sim\theta_{2}$.
We denote by
$(\tilde{\Sigma},\Sigma,W;m,n)$ the corresponding symmetric triad
with multiplicities.
Then,
for any symmetric triad
$(\tilde{\Sigma}',\Sigma',W';m',n')\sim(\tilde{\Sigma},\Sigma,W;m,n)$
with multiplicities,
there exists a commutative compact symmetric triad
$(G,\theta_{1}',\theta_{2}')$
such that
$(G,\theta_{1}',\theta_{2}')$
is isomorphic to
$(G,\theta_{1},\theta_{2})$ with respect to $\sim$,
and that the symmetric triad with multiplicities
of $(G,\theta_{1}',\theta_{2}')$
is $(\tilde{\Sigma}',\Sigma',W';m',n')$.
\end{lem}

\begin{proof}
From $(\tilde{\Sigma}',\Sigma',W';m',n')\sim(\tilde{\Sigma},\Sigma,W;m,n)$,
there exists an element $Y$ of the lattice $\Gamma$ satisfying
\eqref{eqn:equiv-relation} and
\eqref{eqn:equiv-relation-mn}.
Then,
$(G,\theta_{1}',\theta_{2}')=
(G,\theta_{1},\tau_{\exp(Y)}\theta_{2}\tau_{\exp(Y)}^{-1})$
gives a commutative compact symmetric triad
as in the statement.
\end{proof}

\subsection{Descriptions of symmetric triads with multiplicities by double Satake diagrams}\label{sec:symm_2satake}

We will determine the isomorphism class
of the symmetric triad with multiplicities
corresponding to a commutative compact symmetric triad.
For this,
we make use of the notion of double Satake diagrams
for compact symmetric triads.
The authors classified the isomorphism
class of commutative compact symmetric triads
with respect to $\sim$
in terms of the notion of double Satake diagrams
(cf.~\cite[Remark 6.13]{BI}).
Then any commutative compact symmetric
triad is realized as a double Satake diagram.
For the purpose of this subsection,
we describe the corresponding symmetric triads with multiplicities
by means of the double Satake diagram
(see \eqref{eqn:tildeSigma_description}
and Proposition \ref{pro:symm_determ_base}).

\subsubsection{Double Satake diagrams for commutative compact symmetric triads (Review)}

We start with recalling the notion of double Satake diagrams
for commutative compact symmetric triads.
Let $G$ be a compact connected semisimple Lie group
with Lie algebra $\mathfrak{g}$.
Let $(G,\theta_{1},\theta_{2})$ be a commutative compact symmetric triad.
The following lemma is fundamental in our argument.

\begin{lem}[{\cite[Lemma 5.6]{BI}}]\label{lem:mas_fundamental}
Under the above settings,
there exists a maximal abelian subalgebra $\mathfrak{t}$
of $\mathfrak{g}$ satisfying the following conditions:
\begin{enumerate}[(1)]
\item For each $i=1,2$,
the subspace
$\mathfrak{t}\cap\mathfrak{m}_{i}$
is maximal abelian in $\mathfrak{m}_{i}$.
In particular, $\mathfrak{t}$ is $(\theta_{1},\theta_{2})$-invariant.
\item $\mathfrak{t}\cap(\mathfrak{m}_{1}\cap\mathfrak{m}_{2})$
is a maximal abelian subspace of $\mathfrak{m}_{1}\cap\mathfrak{m}_{2}$.
\end{enumerate}
\end{lem}

Let $\mathfrak{t}$ be a maximal abelian subalgebra of $\mathfrak{g}$
as in Lemma \ref{lem:mas_fundamental},
and $\Delta$ denote the root system of $\mathfrak{g}$
with respect to $\mathfrak{t}$.
It follows from the $(\theta_{1},\theta_{2})$-invariance of $\mathfrak{t}$
that
$\sigma_{i}=-d\theta_{i}|_{\mathfrak{t}}$
gives an involutive linear isometry of $\mathfrak{t}$ satisfying
$\sigma_{i}(\Delta)=\Delta$.
Then $(\Delta,\sigma_{i})$
becomes a normal $\sigma$-system
in the sense of \cite[Definition, p.~21]{Warner}.
It is known that the triplet $(\Delta,\sigma_{1},\sigma_{2})$
is uniquely determined for $(G,\theta_{1},\theta_{2})$ (see \cite[Section 5]{BI} for the detail).
We call $(\Delta,\sigma_{1},\sigma_{2})$
the \textit{double $\sigma$-system} of $\mathfrak{t}$
corresponding to $(G,\theta_{1},\theta_{2})$.
It follows from \cite[Lemma 5.9]{BI}
that there exists a fundamental system $\Pi$ of $\Delta$
such that $\Pi$ is both a $\sigma_{1}$-fundamental system
and a $\sigma_{2}$-fundamental system.
We denote by $S_{i}$
the Satake diagram of $(G,\theta_{i})$
associated with $\Pi$.
It is known that the pair $(S_{1},S_{2})$
is uniquely determined for $(G,\theta_{1},\theta_{2})$,
that is,
$(S_{1},S_{2})$
is independent of the choices of
$\mathfrak{t}$ and $\Pi$ (see \cite[Sections 4 and 5]{BI}
for the detail).
We call $(S_{1},S_{2})$ the \textit{double Satake diagram}
for $(G,\theta_{1},\theta_{2})$.
It is shown that,
if $(G,\theta_{1}',\theta_{2}')$
is another commutative compact symmetric triad
satisfying
$(G,\theta_{1}',\theta_{2}')\sim(G,\theta_{1},\theta_{2})$,
then the double Satake diagram $(S_{1}',S_{2}')$
for $(G,\theta_{1}',\theta_{2}')$
is equivalent to that for $(G,\theta_{1},\theta_{2})$
in the sense of \cite[Definition 4.6]{BI},
that is, there exists a common isomorphism of Satake diagrams from $S_{i}$ to $S_{i}'$ for $i=1,2$,
which we write $(S_{1},S_{2})\sim(S_{1}',S_{2}')$.
This means that any commutative compact symmetric triad
uniquely determines the equivalence class of a double Satake diagram
up to this equivalence relation.

\subsubsection{Descriptions of symmetric triads with multiplicities by double Satake diagrams}\label{sec:description_symm}

Let $(G,\theta_{1},\theta_{2})$ be a commutative compact symmetric
triad with $\theta_{1}\not\sim\theta_{2}$.
Suppose that $G$ is simple.
Let $\mathfrak{t}$ be a maximal abelian subalgebra of $\mathfrak{g}$
as in Lemma \ref{lem:mas_fundamental}.
We put $\mathfrak{a}=\mathfrak{t}\cap(\mathfrak{m}_{1}\cap\mathfrak{m}_{2})$,
which is a maximal abelian subspace of $\mathfrak{m}_{1}\cap\mathfrak{m}_{2}$
by Lemma \ref{lem:mas_fundamental}, (2).
We denote by
$(\tilde{\Sigma},\Sigma,W;m,n)$ the symmetric triads
with multiplicities of $\mathfrak{a}$ corresponding to $(G,\theta_{1},\theta_{2})$.
We give a description of $(\tilde{\Sigma},\Sigma,W;m,n)$
by means of the double $\sigma$-system $(\Delta,\sigma_{1},\sigma_{2})$
of $\mathfrak{t}$
corresponding to $(G,\theta_{1},\theta_{2})$.
The double $\sigma$-system
is reconstructed from the double Satake diagram $(S_{1},S_{2})$ in a natural manner.

We first give a description of $\tilde{\Sigma}$ in terms of $(\Delta,\sigma_{1},\sigma_{2})$.
For each $\alpha\in\mathfrak{t}$,
we define the complex subspace
$\mathfrak{g}(\mathfrak{t},\alpha)$ of $\mathfrak{g}^{\mathbb{C}}$
as follows:
\begin{equation}
\mathfrak{g}(\mathfrak{t},\alpha)=\{
X\in\mathfrak{g}^{\mathbb{C}}
\mid [H,X]
=\sqrt{-1}\INN{\alpha}{H}X,\,
H\in\mathfrak{t}
\},
\end{equation}
where $\INN{\cdot}{\cdot}$
denotes the innder product on $\mathfrak{t}$
induced from the invariant inner product on $\mathfrak{g}$.
By definition,
$\mathfrak{g}(\mathfrak{t},-\alpha)$
coincides with the complex conjugate of $\mathfrak{g}(\mathfrak{t},\alpha)$
with respect to
the compact real form $\mathfrak{g}$ of $\mathfrak{g}^{\mathbb{C}}$,
which we write
%\begin{equation}
$\mathfrak{g}(\mathfrak{t},-\alpha)
=\overline{\mathfrak{g}(\mathfrak{t},\alpha)}$.
%\end{equation}
We have the following root space decomposition of
$\mathfrak{g}^{\mathbb{C}}$:
\begin{equation}
\mathfrak{g}^{\mathbb{C}}
=\mathfrak{t}^{\mathbb{C}}\oplus
\sum_{\alpha\in\Delta}\mathfrak{g}(\mathfrak{t},\alpha).
\end{equation}
We set $\Delta_{0}=\{\alpha\in\Delta\mid\INN{\alpha}{\mathfrak{a}}=\{0\}\}
=\{\alpha\in\Delta\mid\mathrm{pr}(\alpha)=0\}$,
where $\mathrm{pr}:\mathfrak{t}\to\mathfrak{a}$
denotes the orthogonal projection,
that is,
\begin{equation}
\mathrm{pr}(\alpha)=\dfrac{1}{4}(\alpha+\sigma_{1}(\alpha)+\sigma_{2}(\alpha)+\sigma_{1}\sigma_{2}(\alpha)).
\end{equation}
Then we have
\begin{equation}
\mathfrak{g}(\mathfrak{a},0)
=\mathfrak{t}^{\mathbb{C}}\oplus\sum_{\alpha\in\Delta_{0}}\mathfrak{g}(\mathfrak{t},\alpha),
\end{equation}
and, for each $\lambda\in\tilde{\Sigma}$,
\begin{equation}\label{eqn:galambda_decomp}
\mathfrak{g}(\mathfrak{a},\lambda)
%=\sum_{\alpha\in\Delta;\,\mathrm{pr}(\alpha)=\lambda}\mathfrak{g}
%(\mathfrak{t},\alpha)
=\sum_{\alpha\in\Delta-\Delta_{0};\,\mathrm{pr}(\alpha)=\lambda}\mathfrak{g}(\mathfrak{t},\alpha).
\end{equation}
Hence we obtain
\begin{equation}\label{eqn:tildeSigma_description}
\tilde{\Sigma}=\mathrm{pr}(\Delta-\Delta_{0}).
\end{equation}
Since $\Delta$, $\Delta_{0}$
and $\mathrm{pr}:\mathfrak{t}\to\mathfrak{a}$ are read off from $(S_{1},S_{2})$,
so is the right hand side of \eqref{eqn:tildeSigma_description}.
Therefore, $\tilde{\Sigma}$ is determined from $(S_{1},S_{2})$.

Next, we give descriptions of $\Sigma$ and $W$.
For this, we need some preparations.

\begin{dfn}\label{dfn:imgroot_cpxroot}
An element $\alpha$ of $\Delta-\Delta_{0}$
is called an \textit{imaginary root}
if $\sigma_{1}\sigma_{2}(\alpha)=\alpha$;
and a \textit{complex root} if $\sigma_{1}\sigma_{2}(\alpha)\neq\alpha$.
We denote by $\Delta_{\mathrm{im}}$ the set of all imaginary roots
of $\Delta-\Delta_{0}$
and by $\Delta_{\mathrm{cpx}}$ the set of all complex roots
of $\Delta-\Delta_{0}$.
\end{dfn}

By definition,
$\Delta=\Delta_{0}\cup\Delta_{\mathrm{im}}\cup\Delta_{\mathrm{cpx}}$
is a disjoint union.
It can be verified that
$\Delta_{\mathrm{im}}$ and
$\Delta_{\mathrm{cpx}}$
are invariant by the multiplication of $-1$,
and by the action
of $(\sigma_{1},\sigma_{2})$.
The sets $\Delta_{\mathrm{im}}$
and $\Delta_{\mathrm{cpx}}$
can be determined by the double Satake diagram.
Here, we remark that the group $\mathbb{Z}_{2}=\{1,\sigma_{1}\sigma_{2}\}$
acts freely on $\Delta_{\mathrm{cpx}}$.
Hence the cardinality 
$\#\Delta_{\mathrm{cpx}}$
of $\Delta_{\mathrm{cpx}}$ is even.

\begin{dfn}\label{dfn:cptroot_noncptroot}
An element $\alpha$ of $\Delta-\Delta_{0}$
is called a \textit{compact root}
if $\mathfrak{g}(\mathfrak{t},\alpha)$
is contained in $(\mathfrak{g}^{\theta_{1}\theta_{2}})^{\mathbb{C}}=((\mathfrak{k}_{1}\cap\mathfrak{k}_{2})\oplus(\mathfrak{m}_{1}\cap\mathfrak{m}_{2}))^{\mathbb{C}}$;
a \textit{noncompact root}
if $\mathfrak{g}(\mathfrak{t},\alpha)$
is contained in $(\mathfrak{g}^{-\theta_{1}\theta_{2}})^{\mathbb{C}}=((\mathfrak{k}_{1}\cap\mathfrak{m}_{2})\oplus(\mathfrak{m}_{1}\cap\mathfrak{k}_{2}))^{\mathbb{C}}$.
We denote by $\Delta_{\mathrm{cpt}}$
the set of all compact roots of $\Delta-\Delta_{0}$
and by $\Delta_{\mathrm{noncpt}}$
the set of all noncompact roots of $\Delta-\Delta_{0}$.
\end{dfn}

As will be discussed in Section \ref{sec:validity_imgcpx_cptnoncpt},
the terminologies as in Definitions
\ref{dfn:imgroot_cpxroot}
and \ref{dfn:cptroot_noncptroot}
are borrowed from Vogan diagrams for noncompact semisimple Lie algebras.

\begin{lem}\label{lem:acpt_-acpt}
We have:
\begin{enumerate}[(1)]
\item $\Delta_{\mathrm{cpt}}$
and $\Delta_{\mathrm{noncpt}}$
are invariant by the multiplication of $-1$.
\item $\Delta_{\mathrm{cpt}}$
and $\Delta_{\mathrm{noncpt}}$
are invariant by the action of $(\sigma_{1},\sigma_{2})$.
\end{enumerate}
\end{lem}

\begin{proof}
(1) If $\alpha\in\Delta-\Delta_{0}$
is compact,
then we have
$\mathfrak{g}(\mathfrak{t},-\alpha)
=\overline{\mathfrak{g}(\mathfrak{t},\alpha)}
\subset\overline{
(\mathfrak{g}^{\theta_{1}\theta_{2}})^{\mathbb{C}}}
=(\mathfrak{g}^{\theta_{1}\theta_{2}})^{\mathbb{C}}$,
from which $-\alpha$ is compact.
In a similar manner,
if $\alpha$ is noncompact,
then so is $-\alpha$.

(2) We have
$\mathfrak{g}(\mathfrak{t},\sigma_{i}(\alpha))
=\overline{\theta_{i}\mathfrak{g}(\mathfrak{t},\alpha)}$.
It follows from the commutativity of $\theta_{1}$
and $\theta_{2}$ that
$(\mathfrak{g}^{\theta_{1}\theta_{2}})^{\mathbb{C}}$
and $(\mathfrak{g}^{-\theta_{1}\theta_{2}})^{\mathbb{C}}$
are invariant under the action of $\theta_{i}$.
Hence we have the assertion.
\end{proof}

\begin{lem}\label{lem:img_cptnoncpt}
$\Delta_{\mathrm{im}}
=\Delta_{\mathrm{cpt}}\cup\Delta_{\mathrm{noncpt}}$.
\end{lem}

\begin{proof}
Let $\alpha\in\Delta_{\mathrm{im}}$.
From $\theta_{1}\theta_{2}(\mathfrak{g}(\mathfrak{t},\alpha))
=\mathfrak{g}(\mathfrak{t},\alpha)$, we have
$\mathfrak{g}(\mathfrak{t},\alpha)=(\mathfrak{g}(\mathfrak{t},\alpha)\cap
(\mathfrak{g}^{\theta_{1}\theta_{2}})^{\mathbb{C}})
\oplus(\mathfrak{g}(\mathfrak{t},\alpha)\cap
(\mathfrak{g}^{-\theta_{1}\theta_{2}})^{\mathbb{C}})$.
Since $\mathfrak{g}(\mathfrak{t},\alpha)$
has complex one dimension,
the involution $\theta_{1}\theta_{2}$
is $\pm 1$ on $\mathfrak{g}(\mathfrak{t},\alpha)$.
Thus, $\alpha$ is compact or noncompact.
Conversely,
let us consider
$\alpha\in\Delta_{\mathrm{cpt}}\cup\Delta_{\mathrm{noncpt}}$.
It follows from \cite[Theorem 4.2, (v), Chapter III]{Helgason}
that
$\alpha=[X,Y]$ holds for some
$X\in\mathfrak{g}(\mathfrak{t},\alpha)$
and $Y\in\mathfrak{g}(\mathfrak{t},-\alpha)$.
By the assumption,
there exists $\epsilon\in\{\pm 1\}$
satisfying
$\theta_{1}\theta_{2}(X)=\epsilon X$.
Lemma \ref{lem:acpt_-acpt}, (1)
yields $\theta_{1}\theta_{2}(Y)=\epsilon Y$.
Then we obtain
\begin{equation}
\sigma_{1}\sigma_{2}(\alpha)
=\theta_{1}\theta_{2}[X,Y]
=\epsilon^{2}[X,Y]=[X,Y]=\alpha.
\end{equation}
Thus, we have completed the proof.
\end{proof}

With the above preparations,
we have the following descriptions of $\Sigma$
and $W$.

\begin{pro}\label{pro:symm_determ_base}
We have:
\begin{align}
\Sigma &= \mathrm{pr}(\Delta_{\mathrm{cpt}})\cup\mathrm{pr}(\Delta_{\mathrm{cpx}}),\label{eqn:Sigma_formula}\\
W &= \mathrm{pr}(\Delta_{\mathrm{noncpt}})\cup\mathrm{pr}(\Delta_{\mathrm{cpx}}).\label{eqn:W_formula}
\end{align}
In addition, the multiplicities of $\lambda\in\tilde{\Sigma}$
is expressed as follows:
\begin{align}
m(\lambda)
&=\#\{
\alpha\in\Delta_{\mathrm{cpt}}
\mid \mathrm{pr}(\alpha)=\lambda
\}+\dfrac{1}{2}\#\{
\beta\in\Delta_{\mathrm{cpx}}
\mid
\mathrm{pr}(\beta)=\lambda\},\label{eqn:m_formula}\\
n(\lambda)
&=\#\{
\alpha\in\Delta_{\mathrm{noncpt}}
\mid \mathrm{pr}(\alpha)=\lambda
\}+\dfrac{1}{2}\#\{
\beta\in\Delta_{\mathrm{cpx}}
\mid
\mathrm{pr}(\beta)=\lambda\}.\label{eqn:n_formula}
\end{align}
\end{pro}

\begin{proof}
We show the descriptions
\eqref{eqn:Sigma_formula}
and \eqref{eqn:m_formula}
of $\Sigma$
and $m$, respectively.
Let $\lambda\in\tilde{\Sigma}$.
By using \eqref{eqn:galambda_decomp},
we obtain the following decomposition
of $\mathfrak{g}(\mathfrak{a},\lambda,1)
=\mathfrak{g}(\mathfrak{a},\lambda)\cap
(\mathfrak{g}^{\theta_{1}\theta_{2}})^{\mathbb{C}}$:
\begin{equation}\label{eqn:galambda1_decomp}
\mathfrak{g}(\mathfrak{a},\lambda,1)
=\sum_{\substack{\alpha\in\Delta_{\mathrm{im}};\\\mathrm{pr}(\alpha)=\lambda}}(\mathfrak{g}(\mathfrak{t},\alpha)\cap(\mathfrak{g}^{\theta_{1}\theta_{2}})^{\mathbb{C}})
\oplus
\sum_{\substack{\beta\in\Delta_{\mathrm{cpx}};\\\mathrm{pr}(\beta)=\lambda}}((\mathfrak{g}(\mathfrak{t},\beta)\oplus\mathfrak{g}(\mathfrak{t},\sigma_{1}\sigma_{2}(\beta)))\cap(\mathfrak{g}^{\theta_{1}\theta_{2}})^{\mathbb{C}}).
\end{equation}
For any $\alpha\in\Delta_{\mathrm{im}}$,
we have $(\mathfrak{g}(\mathfrak{t},\alpha)\cap(\mathfrak{g}^{\theta_{1}\theta_{2}})^{\mathbb{C}})=\mathfrak{g}(\mathfrak{t},\alpha)$
if $\alpha$ is compact;
$(\mathfrak{g}(\mathfrak{t},\alpha)\cap(\mathfrak{g}^{\theta_{1}\theta_{2}})^{\mathbb{C}})=\{0\}$ if $\alpha$ is noncompact.
This implies that
the first term of the right hand side in \eqref{eqn:galambda1_decomp}
coincides with
\begin{equation}\label{eqn:gta_formula}
\sum_{\alpha\in\Delta_{\mathrm{cpt}};\,\mathrm{pr}(\alpha)=\lambda}\mathfrak{g}(\mathfrak{t},\alpha).
\end{equation}
On the other hand,
for $\beta\in\Delta_{\mathrm{cpx}}$,
the subspace
\begin{equation}\label{eqn:gtb_formula}
(\mathfrak{g}(\mathfrak{t},\beta)\oplus
\mathfrak{g}(\mathfrak{t},\sigma_{1}\sigma_{2}(\beta)))
\cap(\mathfrak{g}^{\theta_{1}\theta_{2}})^{\mathbb{C}}
=\{X+\theta_{1}\theta_{2}X\mid X\in\mathfrak{g}(\mathfrak{t},\beta)\}
\end{equation}
has complex one dimension.
From the above arguments,
$\lambda$ is in $\Sigma$
if and only if
there exists $\alpha\in\Delta_{\mathrm{cpt}}\cup\Delta_{\mathrm{cpx}}$
satisfying $\mathrm{pr}(\alpha)=\lambda$,
from which \eqref{eqn:Sigma_formula} holds.
We also have \eqref{eqn:m_formula}.
A similar argument shows \eqref{eqn:W_formula} and \eqref{eqn:n_formula}.
Thus, we have completed the proof.
\end{proof}

On the right hand sides in \eqref{eqn:Sigma_formula}
and
\eqref{eqn:W_formula},
$\mathrm{pr}(\Delta_{\mathrm{cpx}})$
is obtained by the double Satake diagram $(S_{1},S_{2})$.
In addition, 
we also have
the value of
$\#\{
\beta\in\Delta_{\mathrm{cpx}}
\mid
\mathrm{pr}(\beta)=\lambda\}$
in the right hand sides in \eqref{eqn:m_formula}
and
\eqref{eqn:n_formula} by $(S_{1},S_{2})$.
In order to determine $(\tilde{\Sigma},\Sigma,W;m,n)$
by Proposition \ref{pro:symm_determ_base},
our remaining tasks are to obtain
$\Delta_{\mathrm{cpt}}$
and $\Delta_{\mathrm{noncpt}}$.
The following three propositions
provide us to study
the compactness/noncompactness
for imaginary roots.

\begin{pro}\label{pro:cpt+cpt_cpt_etc}
Let $\alpha,\beta$ be imaginary roots with $\alpha+\beta\in\Delta-\Delta_{0}$.
Then $\alpha+\beta$ is also imaginary.
In addition, we have:
\begin{enumerate}[(1)]
\item If $\alpha,\beta$ are compact,
then $\alpha+\beta$ is compact.
\item If $\alpha,\beta$ are noncompact,
then $\alpha+\beta$ is compact.
\item If $\alpha$ is compact
and $\beta$ is noncompact,
then $\alpha+\beta$ is noncompact.
\end{enumerate}
\end{pro}

\begin{proof}
Let $\alpha,\beta$ be imaginary roots with $\alpha+\beta\in\Delta-\Delta_{0}$.
Then we have
$\sigma_{1}\sigma_{2}(\alpha+\beta)
=\sigma_{1}\sigma_{2}(\alpha)+\sigma_{1}\sigma_{2}(\beta)=\alpha+\beta$.
Thus, $\alpha+\beta$ is imaginary.
If $\alpha,\beta$ are compact, then we have
\begin{equation}
\mathfrak{g}(\mathfrak{t},\alpha+\beta)
=[\mathfrak{g}(\mathfrak{t},\alpha),\mathfrak{g}(\mathfrak{t},\beta)]
\subset
[(\mathfrak{g}^{\theta_{1}\theta_{2}})^{\mathbb{C}},(\mathfrak{g}^{\theta_{1}\theta_{2}})^{\mathbb{C}}]
\subset
(\mathfrak{g}^{\theta_{1}\theta_{2}})^{\mathbb{C}}.
\end{equation}
Hence we get (1).
A similar argument shows (2) and (3).
\end{proof}

\begin{pro}\label{pro:delta_inductive2}
Let $\delta\in\Delta$.
Assume that
there exist $\alpha,\beta\in\Delta$ satisfying the following conditions:
(i)
$\delta=(\alpha+\beta)+\sigma_{1}\sigma_{2}(\alpha)$;
(ii)
$\alpha+\sigma_{1}\sigma_{2}(\alpha)\not\in\Delta$;
and (iii)
$\beta\in\Delta_{\mathrm{im}}$ with $\alpha+\beta\in\Delta$.
Then, $\delta$ is an imaginary root.
Furthermore, $\delta$ is compact if and only if $\beta$ is compact.
\end{pro}

\begin{proof}
It is clear that $\delta$ is an imaginary root.
From $\delta=(\alpha+\beta)+\sigma_{1}\sigma_{2}(\alpha)$,
we have
$\mathfrak{g}(\mathfrak{t},\delta)=[[\mathfrak{g}(\mathfrak{t},\alpha),\mathfrak{g}(\mathfrak{t},\beta)],\theta_{1}\theta_{2}\mathfrak{g}(\mathfrak{t},\alpha)]$.
There exist $X_{\alpha}\in\mathfrak{g}(\mathfrak{t},\alpha)$
and $X_{\beta}\in\mathfrak{g}(\mathfrak{t},\beta)$
satisfying $[[X_{\alpha},X_{\beta}],\theta_{1}\theta_{2}(X_{\alpha})]\neq 0$.
Then, by using the Jacobi identity for $[\cdot,\cdot]$,
we get
\begin{align}
\theta_{1}\theta_{2}([[X_{\alpha},X_{\beta}],\theta_{1}\theta_{2}(X_{\alpha})])
&= [[\theta_{1}\theta_{2}(X_{\alpha}),\theta_{1}\theta_{2}(X_{\beta})],X_{\alpha}]\\
&= -([[X_{\alpha},\theta_{1}\theta_{2}(X_{\alpha})],\theta_{1}\theta_{2}(X_{\beta})]+[[\theta_{1}\theta_{2}(X_{\beta}),X_{\alpha}],\theta_{1}\theta_{2}(X_{\alpha})])\\
&=[[X_{\alpha},\theta_{1}\theta_{2}(X_{\beta})],\theta_{1}\theta_{2}(X_{\alpha})],
\end{align}
from which
$[[X_{\alpha},X_{\beta}],\theta_{1}\theta_{2}(X_{\alpha})]$
is in $(\mathfrak{g}^{\theta_{1}\theta_{2}})^{\mathbb{C}}$
if and only if so is $X_{\beta}$ by Lemma \ref{lem:img_cptnoncpt}.
Thus, we have the assertion.
\end{proof}

Although the following lemma is well-known
(see \cite{Takeuchi}, for example),
we give a proof for the sake of completeness.

\begin{lem}[]\label{lem:black_ki}
Fix $i=1,2$.
If $\sigma_{i}(\beta)=-\beta$,
then $\mathfrak{g}(\mathfrak{t},\beta)$
is contained in $\mathfrak{k}_{i}^{\mathbb{C}}$.
\end{lem}

\begin{proof}
We have
$\theta_{i}(\mathfrak{g}(\mathfrak{t},\beta))
=\mathfrak{g}(\mathfrak{t},\beta)$,
from which $\mathfrak{g}(\mathfrak{t},\beta)$
is contained in either
$\mathfrak{k}_{i}^{\mathbb{C}}$ or $\mathfrak{m}_{i}^{\mathbb{C}}$.
Suppose for contradiction that
$\mathfrak{g}(\mathfrak{t},\beta)\subset\mathfrak{m}_{i}^{\mathbb{C}}$ holds.
Then $\mathfrak{g}(\mathfrak{t},-\beta)$ is also contained in $\mathfrak{m}_{i}^{\mathbb{C}}$.
Let $X$ be a nonzero element in $\mathfrak{g}(\mathfrak{t},\beta)$.
Then we have $\theta_{i}(X+\overline{X})=-(X+\overline{X})$,
from which $X+\overline{X}\in\mathfrak{m}_{i}$.
If we put $\mathfrak{a}_{i}'=\mathbb{R}(X+\overline{X})\oplus\mathfrak{a}_{i}$,
then $\mathfrak{a}_{i}'$ becomes an abelian subspace of $\mathfrak{m}_{i}$
containing $\mathfrak{a}_{i}$.
This contradicts to the maximality of $\mathfrak{a}_{i}$ in $\mathfrak{m}_{i}$
(Lemma \ref{lem:mas_fundamental}, (1)).
Hence we have shown that $\mathfrak{g}(\mathfrak{t},\beta)$ is contained in $\mathfrak{k}_{i}^{\mathbb{C}}$.
\end{proof}

\begin{pro}\label{pro:cpt+black_cpt}
Let $\alpha\in\Delta_{\mathrm{im}}$.
Assume that $\beta\in\Delta$ satisfies
$\sigma_{1}(\beta)=\sigma_{2}(\beta)=-\beta$
and $\alpha+\beta$ is in $\Delta$.
Then $\alpha+\beta$ is imaginary.
In addition,
$\alpha$ is compact
if and only if $\alpha+\beta$
is compact.
\end{pro}

\begin{proof}
From $\sigma_{1}(\beta)=\sigma_{2}(\beta)=-\beta$,
we have $\mathrm{pr}(\beta)=0$.
This yields
$\mathrm{pr}(\alpha+\beta)=\mathrm{pr}(\alpha)\neq 0$.
Furthermore,
by $\sigma_{1}\sigma_{2}(\beta)=\beta$,
we get $\sigma_{1}\sigma_{2}(\alpha+\beta)=\alpha+\beta$.
Thus, $\alpha+\beta$ is an imaginary root.

It follows from Lemma \ref{lem:black_ki}
that the assumption yields
$\mathfrak{g}(\mathfrak{t},\beta)\subset(\mathfrak{k}_{1}\cap\mathfrak{k}_{2})^{\mathbb{C}}$.
Since 
$[(\mathfrak{g}^{\pm\theta_{1}\theta_{2}})^{\mathbb{C}},(\mathfrak{k}_{1}\cap\mathfrak{k}_{2})^{\mathbb{C}}]\subset
(\mathfrak{g}^{\pm\theta_{1}\theta_{2}})^{\mathbb{C}}$ holds,
it is verified that
$\alpha$ is compact if and only if $\alpha+\beta$
is compact.
Thus, we have completed the proof.
\end{proof}

It is verified that
$\Delta_{\mathrm{cpt}}$
and $\Delta_{\mathrm{noncpt}}$
depend on the choice of
the representative $(G,\theta_{1},\theta_{2})$
in its isomorphism class with respect to $\sim$.
Let
$(G,\theta_{1}',\theta_{2}')$
be another
commutative compact symmetric triad
satisfying $(G,\theta_{1}',\theta_{2}')\sim(G,\theta_{1},\theta_{2})$.

We observe the relation between
$\Delta_{\mathrm{cpt}}$
(resp.~$\Delta_{\mathrm{noncpt}}$) for $(G,\theta_{1},\theta_{2})$
and 
$\Delta_{\mathrm{cpt}}'$
(resp.~$\Delta'_{\mathrm{noncpt}}$) for $(G,\theta_{1}',\theta_{2}')$.
Without loss of generalities,
we assume that $\theta_{1}'=\theta_{1}$
and $\theta_{2}'=\tau_{\exp(Y)}\theta_{2}\tau_{\exp(Y)}^{-1}$
for some $Y\in\Gamma$.
Then the maximal abelian subalgebra $\mathfrak{t}$ of $\mathfrak{g}$
satisfies
(1) $\mathfrak{t}\cap\mathfrak{m}_{i}'$
is a maximal abelian subspace of $\mathfrak{m}_{i}'$;
(2) $\mathfrak{t}\cap(\mathfrak{m}_{1}'\cap\mathfrak{m}_{2}')$
is a maximal abelian subspace of $\mathfrak{m}_{1}'\cap\mathfrak{m}_{2}'$
(cf.~Lemma \ref{lem:mas_fundamental}).
We denote by $(\Delta',\sigma_{1}',\sigma_{2}')$ the corresponding double $\sigma$-system of $\mathfrak{t}$ for $(G,\theta_{1}',\theta_{2}')$.
Since $\theta_{i}'=\theta_{i}$ holds on $\mathfrak{t}$,
we have $\Delta'=\Delta$ and $\sigma_{i}'=\sigma_{i}$.
In particular,
we have $\Delta_{0}'=\Delta_{0}$,
$\Delta_{\mathrm{im}}'=\Delta_{\mathrm{im}}$
and $\Delta_{\mathrm{cpx}}'=\Delta_{\mathrm{cpx}}$.
From the above setting,
$\Delta_{\mathrm{cpt}}'$
and $\Delta_{\mathrm{noncpt}}'$
are expressed as follows:
\begin{align*}
\Delta_{\mathrm{cpt}}'
&=\{\alpha\in\Delta_{\mathrm{cpt}}\mid \INN{\alpha}{2Y}\in 2\pi\mathbb{Z}\}
\cup\{\alpha\in\Delta_{\mathrm{noncpt}}\mid \INN{\alpha}{2Y}\in \pi+2\pi\mathbb{Z}\},\\
\Delta_{\mathrm{noncpt}}'
&=\{\alpha\in\Delta_{\mathrm{cpt}}\mid \INN{\alpha}{2Y}\in \pi+2\pi\mathbb{Z}\}
\cup\{\alpha\in\Delta_{\mathrm{noncpt}}\mid \INN{\alpha}{2Y}\in 2\pi\mathbb{Z}\}.
\end{align*}

\subsection{Determinations of symmetric triads with multiplicities}

The purpose of this subsection
is to determine
the corresponding symmetric triad
$(\tilde{\Sigma},\Sigma,W;m,n)$
with multiplicities
up to isomorphism
for a commutative compact symmetric triad $(G,\theta_{1},\theta_{2})$ with $\theta_{1}\not\sim\theta_{2}$.
This determination is given by a case-by-case argument based on
the classification of commutative compact symmetric triads
with respect to $\sim$.
Our result will be accomplished
in Table \ref{table:symmtriext}
(see Theorem \ref{thm:determ_symmetrictriads}).
Here, we note that
the local isomorphism class of $(G,\theta_{1},\theta_{2})$
is uniquely determined by the fixed-point subgroup
$K_{i}$ of $\theta_{i}$ ($i=1,2$)
(cf.~\cite[Corollary 5.19]{BI}).
In this subsection,
we often use the notation
$(G,K_{1},K_{2})$ insted of $(G,\theta_{1},\theta_{2})$
if there is no confusion.

Let us consider the case when $(G,\theta_{1},\theta_{2})$
satisfies $\Delta_{\mathrm{im}}=\emptyset$,
that is, $\Delta_{\mathrm{cpt}}=\Delta_{\mathrm{noncpt}}=\emptyset$.
In this case, 
we have $\Delta_{\mathrm{cpx}}=\Delta-\Delta_{0}$.
By Proposition \ref{pro:symm_determ_base},
we have
$\Sigma=W=\tilde{\Sigma}=\mathrm{pr}(\Delta_{\mathrm{cpx}})$
and
\begin{equation}\label{eqn:multi_im_empty}
m(\lambda)=n(\lambda)=\dfrac{1}{2}\#\{\beta\in\Delta_{\mathrm{cpx}}\mid
\mathrm{pr}(\beta)=\lambda\}\quad
(\lambda\in\tilde{\Sigma}).
\end{equation}
In particular,
$(\tilde{\Sigma},\Sigma,W)$
is of type (III)
and $m(\lambda)=n(\lambda)$ ($\lambda\in\tilde{\Sigma}$) holds.

\begin{ex}\label{ex:EI-IV}
Let us consider the case when $(G,K_{1},K_{2})=(E_{6},Sp(4),F_{4})$.
Table \ref{table:EI-IV_2satake}
shows the double Satake diagram
for $(E_{6},Sp(4),F_{4})$,
from which the action of $(\sigma_{1},\sigma_{2})$
on $\Delta$ is reconstructed from
$\sigma_{1}(\alpha_{j})=\alpha_{j}$
($j=1,\dotsc,6$)
and
\begin{equation}
\begin{cases}
\sigma_{2}(\alpha_{1})=\alpha_{1}+\alpha_{2}+2\alpha_{3}+2\alpha_{4}+\alpha_{5}, & \sigma_{2}(\alpha_{k})
=-\alpha_{k}\quad
(k=2,3,4,5),\\
\sigma_{2}(\alpha_{6})=\alpha_{2}+\alpha_{3}+2\alpha_{4}+2\alpha_{5}+\alpha_{6}.
\end{cases}
\end{equation}
Then $\Delta_{0}$
is obtained as the root system generated by $\alpha_{k}$
($k=2,3,4,5$).
We also have $\Delta_{\mathrm{im}}=\emptyset$.
By \eqref{eqn:tildeSigma_description},
a direct calculation shows $\tilde{\Sigma}\simeq A_{2}$.
Hence we find $(\tilde{\Sigma},\Sigma,W)=(\mbox{III-}A_{2})$.
For each $\lambda\in\tilde{\Sigma}$,
since $\{\alpha\in\Delta-\Delta_{0}\mid \mathrm{pr}(\alpha)=\lambda\}$ consists of eight complex roots,
we have $(m(\lambda),n(\lambda))=(4,4)$.
Thus, we have determined $(\tilde{\Sigma},\Sigma,W;m,n)$.

\begin{table}[!!ht]
\renewcommand{\arraystretch}{1.5}
\centering
\caption{Double Satake diagram
$(S_{1},S_{2})$ of $(E_{6},Sp(4),F_{4})$}\label{table:EI-IV_2satake}
\begin{tabular}{cc}
\hline
\hline
Satake diagram $S_{1}$ & Satake diagram $S_{2}$ \\
\hline
\hline
\begin{xy}
\ar@{-}(-20,0)*++!D{\alpha_{6}}*{\circ}="a6";(-10,0)*++!D{\alpha_{5}}*{\circ}="a5"
\ar@{-}"a5";(0,0)*++!DL{\alpha_{4}}*{\circ}="a4"
\ar@{-}"a4";(10,0)*++!D{\alpha_{3}}*{\circ}="a3"
\ar@{-}"a3";(20,0)*++!D{\alpha_{1}}*{\circ}="a1"
\ar@{-}"a4";(0,10)*++!D{\alpha_{2}}*{\circ}="a2"
%%% arrows
%\ar@/_4mm/@{<->} "a6";"a1"
\end{xy}
&
\begin{xy}
\ar@{-}(-20,0)*++!D{\alpha_{6}}*{\circ}="a6";(-10,0)*++!D{\alpha_{5}}*{\bullet}="a5"
\ar@{-}"a5";(0,0)*++!DL{\alpha_{4}}*{\bullet}="a4"
\ar@{-}"a4";(10,0)*++!D{\alpha_{3}}*{\bullet}="a3"
\ar@{-}"a3";(20,0)*++!D{\alpha_{1}}*{\circ}="a1"
\ar@{-}"a4";(0,10)*++!D{\alpha_{2}}*{\bullet}="a2"
\end{xy}\\
\hline
\hline
\end{tabular}
\end{table}
\end{ex}

It is shown that
the other $(G,\theta_{1},\theta_{2})$
satisfying $\Delta_{\mathrm{im}}=\emptyset$
is $(SU(2m),SO(2m),Sp(m))$.
A similar argument as in
Example \ref{ex:EI-IV}
shows
$(\tilde{\Sigma},\Sigma,W)=(\mbox{III-}A_{m-1})$
and $(m(\lambda),n(\lambda))=(2,2)$ for all $\lambda\in\tilde{\Sigma}$
for this commutative compact symmetric triad.

We give some observations
before determing 
the corresponding symmetric triads with multiplicities
for commutative compact symmetric triads satisfying $\Delta_{\mathrm{im}}\neq\emptyset$.
For an abstract symmetric triad
$(\tilde{\Sigma},\Sigma,W)$,
we note that $\Sigma\cap W$
is independent of the choice of
a representative of its isomorphism class
with respect to $\sim$.
Furthermore, we have the following lemma
by the classification.

\begin{lem}\label{lem:tildeSigma_SigmaW_determ_sufficient}
Let
$(\tilde{\Sigma},\Sigma,W)$,
$(\tilde{\Sigma}',\Sigma',W')$
be two abstract symmetric triads.
Suppose that
$\tilde{\Sigma}=\tilde{\Sigma}'$
and
$\Sigma\cap W=\Sigma'\cap W'$
hold.
\begin{enumerate}[(1)]
\item If $\tilde{\Sigma}$ is not of $BC$-type,
then $(\tilde{\Sigma},\Sigma,W)\sim
(\tilde{\Sigma}',\Sigma',W')$ holds.
\item 
 Suppose that $\tilde{\Sigma}$ is of $BC$-type.
If $\Sigma\cap W$ is not of $B$-type,
then
$(\tilde{\Sigma},\Sigma,W)\sim(\tilde{\Sigma}',\Sigma',W')$
holds.
\end{enumerate}
\end{lem}

We will make use of this lemma
to determine the isomorphism class
of the corresponding symmetric triads
for commutative compact symmetric triads $(G,\theta_{1},\theta_{2})$.
In the case when $(\tilde{\Sigma},\Sigma,W)$
is constructed from
a commutative compact symmetric triad,
it follows from \eqref{eqn:Sigma_formula}
and \eqref{eqn:W_formula}
that $\Sigma\cap W$ contains
$\mathrm{pr}(\Delta_{\mathrm{cpx}})$.
In fact,
we can show
the following proposition.

\begin{pro}
Let $(\tilde{\Sigma},\Sigma,W)$
be the symmetric triad
corresponding to a commutative compact symmetric
triad. Then we have
$\Sigma\cap W=\mathrm{pr}(\Delta_{\mathrm{cpx}})$.
\end{pro}

\begin{proof}
By \eqref{eqn:Sigma_formula} and \eqref{eqn:W_formula},
it is sufficient to show $\Sigma\cap W \subset \mathrm{pr}(\Delta_{\mathrm{cpx}})$.
Let $\lambda\in\Sigma\cap W$ and $\Delta_{\lambda}:=\{\alpha\in\Delta \mid \mathrm{pr}(\alpha)=\lambda\}$.
For simplicity, let us consider the case when $\lambda$ is a positive root.
Suppose for contradiction that
there exists $\lambda\in\Sigma\cap W$ such that
$\Delta_{\lambda}$ consists of
imaginary roots.
Since $\lambda\in \Sigma\cap W$,
we obtain
\begin{equation}\label{eqn:Deltalambda_cpt_noncpt}
\Delta_{\lambda}\cap\Delta_{\mathrm{cpt}}\neq\emptyset,\quad
\Delta_{\lambda}\cap\Delta_{\mathrm{noncpt}}\neq\emptyset.
\end{equation}
On the other hand, 
for any $\alpha\in\Delta_{\lambda}$,
we have $\theta_{1}\theta_{2}(\alpha)=\alpha$
by using the assumption for contradiction.
Then we get $\lambda=\mathrm{pr}(\alpha)=(1/2)(\alpha-\theta_{1}(\alpha))$.
Hence $\lambda$ can be regarded as a restricted root of
the compact symmetric pair $(G,\theta_{1})$ associated with $\mathfrak{a}_{1}$
and $\Delta_{\lambda}=\{\alpha\in\Delta\mid\mathrm{pr}_{1}(\alpha)=\lambda\}$,
where $\mathrm{pr}_{1}:\mathfrak{t}\to\mathfrak{a}_{1}$ denotes the orthogonal projection.
By the restricted root system theory of compact symmetric pairs,
there exists $\delta\in\Delta_{\lambda}$ such that 
any root $\alpha\in\Delta_{\lambda}$ is expressed as follows:
(i) $\alpha\in\{\delta,\sigma_{1}(\delta)\}$ or
(ii) $\alpha\in\{(\dotsb((\delta+\gamma_{1})+\gamma_{2})+\dotsb)+\gamma_{k},
(\dotsb((\sigma_{1}(\delta)-\gamma_{1})-\gamma_{2})-\dotsb)-\gamma_{k}\}$
for some $\gamma_{1},\dotsc,\gamma_{k}\in\Pi_{1,0}$.
It follows from Lemma \ref{lem:acpt_-acpt}, (2)
that $\sigma_{1}(\delta)$ is compact
if so is $\delta$.
If $\delta+\gamma\in\Delta_{\lambda}$
for $\gamma\in\Pi_{1,0}$,
then we have $\sigma_{1}(\gamma)=\sigma_{2}(\gamma)=-\gamma$.
From Proposition \ref{pro:cpt+black_cpt},
$\delta+\gamma$ is compact
if so is $\delta$.
By induction,
$\alpha\in\Delta_{\lambda}$ expressed as the above (ii)
is also compact if so is $\delta$.
Hence we have
$\Delta_{\lambda}\subset \Delta_{\mathrm{cpt}}$
if $\delta$ is compact;
$\Delta_{\lambda}\subset \Delta_{\mathrm{noncpt}}$
if $\delta$ is noncompact.
This contradicts to \eqref{eqn:Deltalambda_cpt_noncpt}.
Therefore $\Delta_{\lambda}$ contains a complex root $\alpha$,
from which $\lambda=\mathrm{pr}(\alpha)\in\mathrm{pr}(\Delta_{\mathrm{cpx}})$ holds.
Since $\Sigma\cap W$ and $\mathrm{pr}(\Delta_{\mathrm{cpx}})$
are invariant under the multiplication of $-1$,
we obtain
$\Sigma \cap W =\mathrm{pr}(\Delta_{\mathrm{cpx}})$.
\end{proof}

%The proof of this proposition
%is given by a case-by-case
%verification based on the classification
%of commutative compact symmetric triads,
%which is completed within this subsection.

Let us consider
the case when
$(G,\theta_{1},\theta_{2})$
satisfies that
$\tilde{\Sigma}$ is not of $BC$-type.

\begin{ex}\label{ex:EI-II_SYMMm}
Let us consider the case when $(G,K_{1},K_{2})=(E_{6},Sp(4),SU(6)\cdot SU(2))$.
Table \ref{table:EI-II_2satake}
shows the double Satake diagram
for $(E_{6},Sp(4),SU(6)\cdot SU(2))$,
from which we obtain
$\Delta_{0}=\emptyset$, $\Delta_{\mathrm{im}}$
and $\Delta_{\mathrm{cpx}}$.
By a similar argument as in Example \ref{ex:EI-IV},
we have $\tilde{\Sigma}=\mathrm{pr}(\Delta)=F_{4}$.
We also get
$\Sigma\cap W
=\mathrm{pr}(\Delta_{\mathrm{cpx}})
=\{\text{the short roots in $F_{4}$}\}\simeq D_{4}$.
By Lemmas \ref{lem:CSTSYMM_SYMM'_CST'}
and \ref{lem:tildeSigma_SigmaW_determ_sufficient},
without loss of generalities,
we assume $(\tilde{\Sigma},\Sigma,W)=(\mbox{I-}F_{4})$.
For each short root $\lambda\in\tilde{\Sigma}$,
since $\{\alpha\in\Delta-\Delta_{0}\mid\mathrm{pr}(\alpha)=\lambda\}$ consists of two complex roots,
we have
$(m(\lambda),n(\lambda))=(1,1)$.
For each long root $\mu\in\tilde{\Sigma}$,
it follows from $\mu\in\Sigma$
that $\{\alpha\in\Delta-\Delta_{0}\mid \mathrm{pr}(\alpha)=\mu\}$
consists of one compact root.
Hence we get
$(m(\mu),n(\mu))=(1,0)$.
Thus, we have determined
$(\tilde{\Sigma},\Sigma,W;m,n)$.

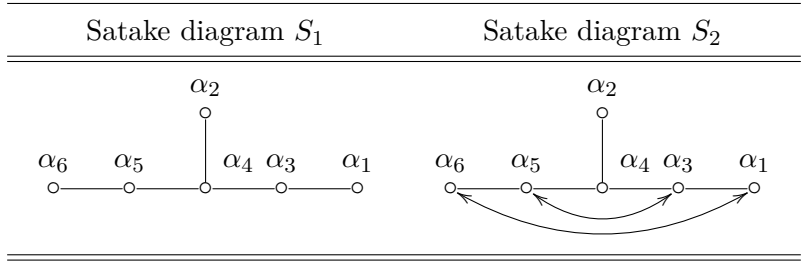
\begin{table}[!!ht]
\renewcommand{\arraystretch}{1.5}
\centering
\caption{Double Satake diagram
$(S_{1},S_{2})$ of $(E_{6},Sp(4),SU(6)\cdot SU(2))$}\label{table:EI-II_2satake}
\begin{tabular}{cc}
\hline
\hline
Satake diagram $S_{1}$ & Satake diagram $S_{2}$ \\
\hline
\hline
\begin{xy}
\ar@{-}(-20,0)*++!D{\alpha_{6}}*{\circ}="a6";(-10,0)*++!D{\alpha_{5}}*{\circ}="a5"
\ar@{-}"a5";(0,0)*++!DL{\alpha_{4}}*{\circ}="a4"
\ar@{-}"a4";(10,0)*++!D{\alpha_{3}}*{\circ}="a3"
\ar@{-}"a3";(20,0)*++!D{\alpha_{1}}*{\circ}="a1"
\ar@{-}"a4";(0,10)*++!D{\alpha_{2}}*{\circ}="a2"
%%% arrows
%\ar@/_4mm/@{<->} "a6";"a1"
\end{xy}
&
\begin{xy}
\ar@{-}(-20,0)*++!D{\alpha_{6}}*{\circ}="a6";(-10,0)*++!D{\alpha_{5}}*{\circ}="a5"
\ar@{-}"a5";(0,0)*++!DL{\alpha_{4}}*{\circ}="a4"
\ar@{-}"a4";(10,0)*++!D{\alpha_{3}}*{\circ}="a3"
\ar@{-}"a3";(20,0)*++!D{\alpha_{1}}*{\circ}="a1"
\ar@{-}"a4";(0,10)*++!D{\alpha_{2}}*{\circ}="a2"
%%% arrows
\ar@/^4mm/@{<->} "a3";"a5"
\ar@/_6mm/@{<->} "a6";"a1"
\end{xy}\\[4ex]
\hline
\hline
\end{tabular}
\end{table}
\end{ex}

In a similar manner,
we can determine
$(\tilde{\Sigma},\Sigma,W;m,n)$
with multiplicities
for $(G,\theta_{1},\theta_{2})$
such that $\tilde{\Sigma}$ is not of $BC$-type.

Next, let us consider the case when
$(G,\theta_{1},\theta_{2})$ satisfies that
$\tilde{\Sigma}$ is of $BC$-type.
Then, by the classification,
the type of
$(\tilde{\Sigma},\Sigma,W)$
becomes one of
$(\mbox{I-}BC_{r}\mbox{-}A_{1}^{r})$,
$(\mbox{I-}BC_{r}\mbox{-}B_{r})$,
$(\mbox{II-}BC_{r})$ or
$(\mbox{III-}BC_{r})$.
In what follows,
we demonstrate
our determination of
$(\tilde{\Sigma},\Sigma,W;m,n)$.

We first give an example
of $(G,\theta_{1},\theta_{2})$
satisfying
$(\tilde{\Sigma},\Sigma,W)=(\mbox{I-}BC_{r}\mbox{-}A_{1}^{r})$.

\begin{ex}\label{ex:SOUU_SYMMm}
Let us determine the symmetric triad
with multiplicities of $(G,K_{1},K_{2})=(SO(4m),U(2m),U(2m)')$,
which is a commutative compact symmetric triad
$(G,\theta,\theta')$
such that the fixed-point subgroups of $\theta$
and $\theta'$ are isomorphic to the unitary group $U(2m)$
and that $(G,\theta,\theta')\not\sim(G,\theta,\theta)$
holds.
Table \ref{table:SOUU_SYMMm}
shows the double Satake diagram
for $(SO(4m),U(2m),U(2m)')$,
from which we obtain $\Delta_{0}\simeq A_{1}^{m+1}$,
$\Delta_{\mathrm{im}}$ and $\Delta_{\mathrm{cpx}}$.
By \eqref{eqn:tildeSigma_description},
a direct calculation shows
$\tilde{\Sigma}\simeq BC_{r}$
($r=m-1$),
which we write $\tilde{\Sigma}=\{\pm e_{i}\pm e_{j}\mid
1\leq i<j\leq r\}\cup\{\pm e_{i},\pm 2e_{i}\mid 1\leq i\leq r\}$.
We also get $\Sigma\cap W=\mathrm{pr}(\Delta_{\mathrm{cpx}})=A_{1}^{r}$.
Without loss generalities,
we assume that $(\tilde{\Sigma},\Sigma,W)=(\mbox{I-}BC_{r}\mbox{-}A_{1}^{r})$
holds.
Since $\{\alpha\in\Delta\mid \mathrm{pr}(\alpha)=e_{i}\}$
consists of eight complex roots,
we have
$(m(e_{i}),n(e_{i}))=(4,4)$.
It follows from 
Proposition \ref{pro:cpt+black_cpt}
and $e_{i}\pm e_{j}\in\Sigma$
that 
$(m(e_{i}\pm e_{j}),n(e_{i}\pm e_{j}))=(4,0)$ holds
for $1\leq i<j\leq r$.
Since $2e_{i}\in\Sigma$ holds,
we have $(m(2e_{i}),n(2e_{i}))=(1,0)$
for $1\leq i\leq r$.

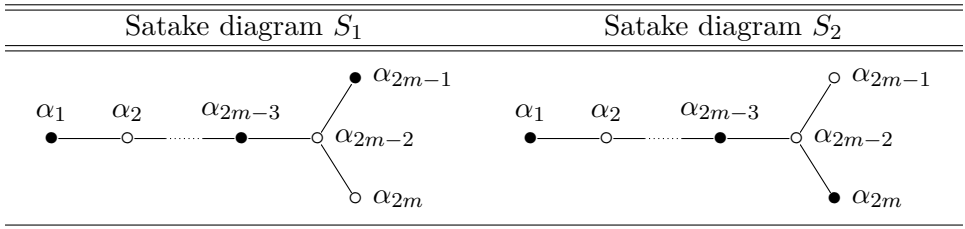
\begin{table}[!!ht]
\centering
\caption{Double Satake diagram $(S_{1},S_{2})$ of $(SO(4m),\theta,\theta')$ with $m\geq 3$}\label{table:SOUU_SYMMm}
\begin{tabular}{cc}
\hline
\hline
Satake diagram $S_{1}$ & Satake diagram $S_{2}$\\
\hline
\hline
\begin{xy}
\ar@{-}(0,0)*++!D{\alpha_{1}}*{\bullet};(10,0)*++!D{\alpha_{2}}*{\circ}="a2"
\ar@{-}"a2";(15,0)*{}
\ar@{.}(15,0)*{};(20,0)*{}
\ar@{-}(20,0)*{};(25,0)*++!D{\alpha_{2m-3}}*{\bullet}="a2m-3"
\ar@{-}"a2m-3";(35,0)*++!L{\alpha_{2m-2}}*{\circ}="a2m-2"
\ar@{-}"a2m-2";(40,8)*++!L{\alpha_{2m-1}}*{\bullet}
\ar@{-}"a2m-2";(40,-8)*++!L{\alpha_{2m}}*{\circ}
%%% arrows
\end{xy}
&
\begin{xy}
\ar@{-}(0,0)*++!D{\alpha_{1}}*{\bullet};(10,0)*++!D{\alpha_{2}}*{\circ}="a2"
\ar@{-}"a2";(15,0)*{}
\ar@{.}(15,0)*{};(20,0)*{}
\ar@{-}(20,0)*{};(25,0)*++!D{\alpha_{2m-3}}*{\bullet}="a2m-3"
\ar@{-}"a2m-3";(35,0)*++!L{\alpha_{2m-2}}*{\circ}="a2m-2"
\ar@{-}"a2m-2";(40,8)*++!L{\alpha_{2m-1}}*{\circ}
\ar@{-}"a2m-2";(40,-8)*++!L{\alpha_{2m}}*{\bullet}
%%% arrows
\end{xy}\\
\hline
\hline
\end{tabular}
\end{table}
\end{ex}

It is shown that
the other $(G,\theta_{1},\theta_{2})$
with $(\tilde{\Sigma},\Sigma,W)=(\mbox{I-}BC_{r}\mbox{-}A_{1}^{r})$
are $(SU(n),S(U(a)\times U(b)),S(U(c)\times U(d)))$
and $(Sp(n),Sp(a)\times Sp(b),Sp(c)\times Sp(d))$
($a<c\leq b<d$).
In fact, their symmetric triads with multiplicities
are determined
by a similar argument as in Example \ref{ex:SOUU_SYMMm}.

Second, we give examples of $(G,\theta_{1},\theta_{2})$
satisfying $(\tilde{\Sigma},\Sigma,W)=(\mbox{I-}BC_{r}\mbox{-}B_{r})$.

\begin{ex}\label{ex:EII-III_SYMMm}
Let us consider the case when
$(G,K_{1},K_{2})=(E_{6},SU(6)\cdot SU(2),SO(10)\cdot U(1))$.
Table \ref{table:EII-III_2satake}
shows the double Satake diagram for $(E_{6},SU(6)\cdot SU(2),SO(10)\cdot U(1))$,
from which we obtain $\Delta_{0}$,
$\Delta_{\mathrm{im}}$
and $\Delta_{\mathrm{cpx}}$.
By \eqref{eqn:tildeSigma_description},
a direct calculation shows
$\tilde{\Sigma}\simeq BC_{2}$,
which we write $\tilde{\Sigma}=\{\pm e_{1}\pm e_{2}\}\cup\{\pm e_{i},\pm 2e_{i}\mid i=1,2\}$.
We also have $\Sigma\cap W=\mathrm{pr}(\Delta_{\mathrm{cpx}})=B_{2}$.
Let us show that $2e_{i}\in \Sigma$ holds for $i=1,2$.
We put $e_{1}-e_{2}=\mathrm{pr}(\gamma)$
and $2e_{2}=\mathrm{pr}(\delta)$,
where
$\gamma=\alpha_{3}+\alpha_{4}+\alpha_{5}$ and
$\delta=\alpha_{1}+\alpha_{3}+\alpha_{4}+\alpha_{5}+\alpha_{6}$.
We get $\gamma, \delta\in\Delta_{\mathrm{im}}$.
Since $(m(e_{1}-e_{2}),n(e_{1}-e_{2}))=(m(e_{1}+e_{2}),n(e_{1}+e_{2}))$
holds,
$\gamma$ is compact if and only if so is $\gamma+\delta$.
Then $\delta$ must be compact (cf.~Proposition \ref{pro:cpt+cpt_cpt_etc}),
from which we get $2e_{2}\in\Sigma$.
This implies that $2e_{1}$ is in $\Sigma$.
Hence we have $(\tilde{\Sigma},\Sigma,W)=(\mbox{I-}BC_{2}\mbox{-}B_{2})$.
Since $\{\alpha\in\Delta\mid \mathrm{pr}(\alpha)=e_{i}\}$
consists of eight complex roots,
we have
$(m(e_{i}),n(e_{i}))=(4,4)$.
From the above,
$\{\alpha\in\Delta-\Delta_{0}\mid \mathrm{pr}(\alpha)=2e_{i}\}$ consists of one compact root,
from which
we have $(m(2e_{i}),n(2e_{i}))=(1,0)$.
By Lemma \ref{lem:acpt_-acpt}, (2),
it is verified that $(m(e_{1}\pm e_{2}),n(e_{1}\pm e_{2}))$
is equal to $(4,2)$ or $(2,4)$.
By the classification as in Example \ref{ex.i-bcr-br},
we may conclude that $(m(e_{1}\pm e_{2}),n(e_{1}\pm e_{2})=(4,2)$.
Thus, we have determined $(\tilde{\Sigma},\Sigma,W;m,n)$.

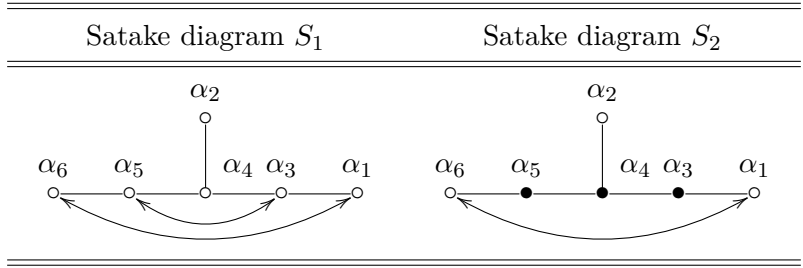
\begin{table}[!!ht]
\renewcommand{\arraystretch}{1.5}
\centering
\caption{Double Satake diagram
$(S_{1},S_{2})$ of $(E_{6},SU(6)\cdot SU(2),SU(10)\cdot U(1))$}\label{table:EII-III_2satake}
\begin{tabular}{cc}
\hline
\hline
Satake diagram $S_{1}$ & Satake diagram $S_{2}$ \\
\hline
\hline
\begin{xy}
\ar@{-}(-20,0)*++!D{\alpha_{6}}*{\circ}="a6";(-10,0)*++!D{\alpha_{5}}*{\circ}="a5"
\ar@{-}"a5";(0,0)*++!DL{\alpha_{4}}*{\circ}="a4"
\ar@{-}"a4";(10,0)*++!D{\alpha_{3}}*{\circ}="a3"
\ar@{-}"a3";(20,0)*++!D{\alpha_{1}}*{\circ}="a1"
\ar@{-}"a4";(0,10)*++!D{\alpha_{2}}*{\circ}="a2"
%%% arrows
\ar@/^4mm/@{<->} "a3";"a5"
\ar@/_6mm/@{<->} "a6";"a1"
\end{xy}
&
\begin{xy}
\ar@{-}(-20,0)*++!D{\alpha_{6}}*{\circ}="a6";(-10,0)*++!D{\alpha_{5}}*{\bullet}="a5"
\ar@{-}"a5";(0,0)*++!DL{\alpha_{4}}*{\bullet}="a4"
\ar@{-}"a4";(10,0)*++!D{\alpha_{3}}*{\bullet}="a3"
\ar@{-}"a3";(20,0)*++!D{\alpha_{1}}*{\circ}="a1"
\ar@{-}"a4";(0,10)*++!D{\alpha_{2}}*{\circ}="a2"
%%% arrows
%\ar@/^4mm/@{<->} "a3";"a5"
\ar@/_6mm/@{<->} "a6";"a1"
\end{xy}\\[4ex]
\hline
\hline
\end{tabular}
\end{table}
\end{ex}

It is shown that the other $(G,\theta_{1},\theta_{2})$
with $(\tilde{\Sigma},\Sigma,W)=(\mbox{I-}BC_{r}\mbox{-}B_{r})$
are $(SO(2m),SO(a)\times SO(b),U(m))$ ($a$: even, $m>a$)
and $(E_{7},SO(12)\cdot SU(2),E_{6}\cdot U(1))$.
In order to determine their symmetric triads
with multiplicities,
we make use of the following proposition.

\begin{pro}\label{pro:delta_compact}
Let $\delta\in\Delta$.
Assume that there exist $\alpha,\beta\in\Delta$
satisfying the following conditions:
(i) $\delta=\alpha+(\beta+\sigma_{1}\sigma_{2}(\alpha))$
with $\beta+\sigma_{1}\sigma_{2}(\alpha)\in\Delta$;
(ii) $\alpha+\sigma_{1}\sigma_{2}(\alpha)\not\in\Delta$;
(iii) $\sigma_{1}(\beta)=\sigma_{2}(\beta)=-\beta$.
Then $\delta$ is a compact root.
\end{pro}

\begin{proof}
From the assumption (i),
we have $\mathfrak{g}(\mathfrak{t},\delta)=[\mathfrak{g}(\mathfrak{t},\alpha),[\mathfrak{g}(\mathfrak{t},\beta),\theta_{1}\theta_{2}\mathfrak{g}(\mathfrak{t},\alpha)]]$.
There exist $X_{\alpha}\in\mathfrak{g}(\mathfrak{t},\alpha)$
and $X_{\beta}\in\mathfrak{g}(\mathfrak{t},\beta)$
satisfying $[X_{\alpha},[X_{\beta},\theta_{1}\theta_{2}(X_{\alpha})]]\neq 0$.
The assumption (iii) yields $\theta_{1}(X_{\beta})=\theta_{2}(X_{\beta})=X_{\beta}$.
Then we have
\begin{align}
\theta_{1}\theta_{2}([X_{\alpha},[X_{\beta},\theta_{1}\theta_{2}(X_{\alpha})]])
&= [\theta_{1}\theta_{2}(X_{\alpha}),[X_{\beta},X_{\alpha}]]\\
&=-([X_{\beta},[X_{\alpha},\theta_{1}\theta_{2}(X_{\alpha})]]+[X_{\alpha},[\theta_{1}\theta_{2}(X_{\alpha}),X_{\beta}]])\\
&= [X_{\alpha},[X_{\beta},\theta_{1}\theta_{2}(X_{\alpha})]].
\end{align}
Here, we have used the assumption (ii)
in the last equality.
Thus, we have completed the proof.
\end{proof}

\begin{ex}\label{ex:E7_SO12SU2_E6U1_SYMM}
Let us consider the case when
$(G,K_{1},K_{2})=(E_{7},SO(12)\cdot SU(2),E_{6}\cdot U(1))$.
Table \ref{table:EVI-VII_2satake} shows
the double Satake diagram
for $(E_{7},SO(12)\cdot SU(2),E_{6}\cdot U(1))$,
from which we obtain $\Delta_{0}, \Delta_{\mathrm{im}}$
and $\Delta_{\mathrm{cpx}}$.
By \eqref{eqn:tildeSigma_description},
a direct calculation shows
$\tilde{\Sigma}\simeq BC_{2}$,
which we write $\tilde{\Sigma}=\{\pm e_{1}\pm e_{2}\}\cup\{\pm e_{i},\pm 2e_{i}\mid i=1,2\}$.
We also have $\Sigma\cap W=\mathrm{pr}(\Delta_{\mathrm{cpx}})=B_{2}$.
Let us show that $2e_{i}\in \Sigma$ holds for $i=1,2$.
We put $2e_{2}=\mathrm{pr}(\delta)$
with
$\delta=\alpha_{2}+\alpha_{3}+2\alpha_{4}+2\alpha_{5}+2\alpha_{6}+\alpha_{7}$.
Then we have $\delta\in\Delta_{\mathrm{cpt}}$
by Proposition \ref{pro:delta_compact}.
Indeed,
if we put $\alpha=\alpha_{6}$ and $\beta=\alpha_{5}$,
then $\delta, \alpha, \beta$ satisfy the assumption
of Proposition \ref{pro:delta_compact}.
Hence $\delta=\alpha+(\beta+\sigma_{1}\sigma_{2}(\alpha))$
is compact, so that $2e_{2}\in \Sigma$ holds.
Since this also yields $2e_{1}\in\Sigma$,
we have $(\tilde{\Sigma},\Sigma,W)=(\mbox{I-}BC_{2}\mbox{-}B_{2})$.
By a similar argument as in Example \ref{ex:EII-III_SYMMm},
we get $(m(e_{i}),n(e_{i}))=(8,8)$ and
$(m(2e_{i}),n(2e_{i}))=(1,0)$ for $i=1,2$.
In addition, we may conclude that
$(m(e_{1}\pm e_{2}),n(e_{1}\pm e_{2}))=(6,2)$
by the classification as in Example \ref{ex.i-bcr-br}.

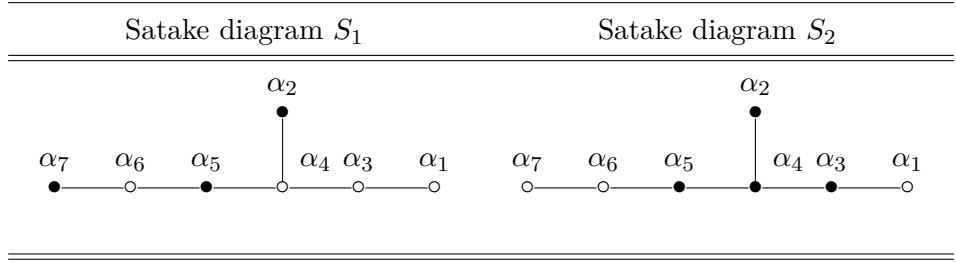
\begin{table}[!!ht]
\renewcommand{\arraystretch}{1.5}
\centering
\caption{Double Satake diagram
$(S_{1},S_{2})$ of $(E_{7},SO(12)\cdot SU(2),E_{6}\cdot U(1))$}\label{table:EVI-VII_2satake}
\begin{tabular}{cc}
\hline
\hline
Satake diagram $S_{1}$ & Satake diagram $S_{2}$ \\
\hline
\hline
\begin{xy}
\ar@{-} (-30,0)*++!D{\alpha_{7}}*{\bullet}="a7";(-20,0)*++!D{\alpha_{6}}*{\circ}="a6"
\ar@{-}"a6";(-10,0)*++!D{\alpha_{5}}*{\bullet}="a5"
\ar@{-}"a5";(0,0)*++!DL{\alpha_{4}}*{\circ}="a4"
\ar@{-}"a4";(10,0)*++!D{\alpha_{3}}*{\circ}="a3"
\ar@{-}"a3";(20,0)*++!D{\alpha_{1}}*{\circ}="a1"
\ar@{-}"a4";(0,10)*++!D{\alpha_{2}}*{\bullet}="a2"
%%% arrows
%\ar@/^4mm/@{<->} "a3";"a5"
%\ar@/_6mm/@{<->} "a6";"a1"
\end{xy}&
\begin{xy}
\ar@{-} (-30,0)*++!D{\alpha_{7}}*{\circ}="a7";(-20,0)*++!D{\alpha_{6}}*{\circ}="a6"
\ar@{-}"a6";(-10,0)*++!D{\alpha_{5}}*{\bullet}="a5"
\ar@{-}"a5";(0,0)*++!DL{\alpha_{4}}*{\bullet}="a4"
\ar@{-}"a4";(10,0)*++!D{\alpha_{3}}*{\bullet}="a3"
\ar@{-}"a3";(20,0)*++!D{\alpha_{1}}*{\circ}="a1"
\ar@{-}"a4";(0,10)*++!D{\alpha_{2}}*{\bullet}="a2"
%%% arrows
%\ar@/^4mm/@{<->} "a3";"a5"
%\ar@/_6mm/@{<->} "a6";"a1"
\end{xy}\\[4ex]
\hline
\hline
\end{tabular}
\end{table}
\end{ex}

We can determine the corresponding $(\tilde{\Sigma},\Sigma,W;m,n)$
for $(SO(2m),SO(a)\times SO(b),U(m))$ ($a$: even, $m>a$)
by a similar argument as in Example \ref{ex:E7_SO12SU2_E6U1_SYMM}.

Third, we give an example of $(G,\theta_{1},\theta_{2})$
satisfying $(\tilde{\Sigma},\Sigma,W)=(\mbox{II-}BC_{r})$.
In this explanation,
we make use of the following proposition.

\begin{pro}\label{pro:delta_noncompact}
Let $\delta\in\Delta$.
If there exists $\alpha\in\Delta$ such that
\begin{equation}\label{eqn:delta_as12a}
\delta=\alpha+\sigma_{1}\sigma_{2}(\alpha),
\end{equation}
then $\delta$ is noncompact.
\end{pro}

\begin{proof}
It follows from \eqref{eqn:delta_as12a}
that there exists $X\in\mathfrak{g}(\mathfrak{t},\alpha)$
such that $\mathfrak{g}(\mathfrak{t},\delta)=\mathbb{C}[X,\theta_{1}\theta_{2}(X)]$ holds.
Then we have
$\theta_{1}\theta_{2}[X_{\alpha},\theta_{1}\theta_{2}(X_{\alpha})]
=[\theta_{1}\theta_{2}(X_{\alpha}),X_{\alpha}]=-[X_{\alpha},\theta_{1}\theta_{2}(X_{\alpha})]$,
from which $\mathfrak{g}(\mathfrak{t},\delta)$ is contained in $(\mathfrak{g}^{-\theta_{1}\theta_{2}})^{\mathbb{C}}$,
that is, $\delta$ is noncompact.
\end{proof}

Here, we emphasize that
the assumptions of Propositions \ref{pro:delta_compact} and \ref{pro:delta_noncompact}
are given by the action
of $(\sigma_{1},\sigma_{2})$
on $\Delta$,
although their conclusions
are represented by the action
of $(\theta_{1},\theta_{2})$ on $\mathfrak{g}$.

\begin{ex}\label{ex:EI-III_determ}
Let us consider the case when $(G,K_{1},K_{2})=(E_{6},Sp(4),SO(10)\cdot U(1))$.
Table \ref{table:EI-III_2satake}
shows the double Satake diagram for$(E_{6},Sp(4),SO(10)\cdot U(1))$,
from which we obtain $\Delta_{0}$,
$\Delta_{\mathrm{im}}$ and $\Delta_{\mathrm{cpx}}$.
By \eqref{eqn:tildeSigma_description},
a direct calculation shows
$\tilde{\Sigma}\simeq BC_{2}$,
which we write $\tilde{\Sigma}=\{\pm e_{1}\pm e_{2}\}\cup\{\pm e_{i},\pm 2e_{i}\mid i=1,2\}$.
We also have $\Sigma\cap W=\mathrm{pr}(\Delta_{\mathrm{cpx}})=B_{2}$.
Let us show that $2e_{i}\in W$ holds for $i=1,2$.
We put $2e_{2}=\mathrm{pr}(\delta)$
with
$\delta=\alpha_{1}+\alpha_{3}+\alpha_{4}+\alpha_{5}+\alpha_{6}$.
Then we have $\delta=\alpha_{1}+\sigma_{1}\sigma_{2}(\alpha_{1})$,
from which $\delta$ is a noncompact root
by Proposition \ref{pro:delta_noncompact}.
This also yields $2e_{1}\in W$.
Hence $(\tilde{\Sigma},\Sigma,W)=(\mbox{II-}BC_{2})$ holds.
From the above, we have
$(m(2e_{i}),n(2e_{i}))=(0,1)$ for $i=1,2$.
Since $\{\alpha\in\Delta\mid \mathrm{pr}(\alpha)=e_{i}\}$
consists of eight complex roots,
we have $(m(e_{i}),n(e_{i}))=(4,4)$.
Similarly,
we get $(m(e_{1}\pm e_{2}),n(e_{1}\pm e_{2}))=(3,3)$.
Thus, we have determined $(\tilde{\Sigma},\Sigma,W;m,n)$.

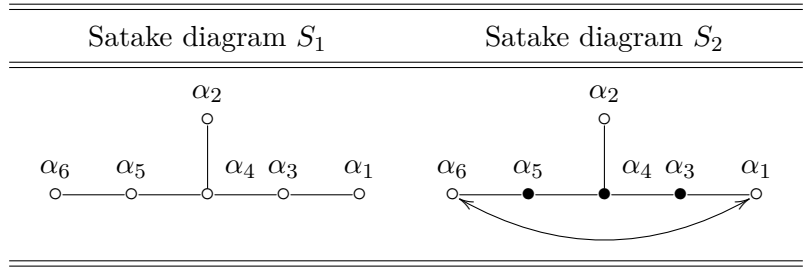
\begin{table}[!!ht]
\renewcommand{\arraystretch}{1.5}
\centering
\caption{Double Satake diagram
$(S_{1},S_{2})$ of $(E_{6},Sp(4),SO(10)\cdot U(1))$}\label{table:EI-III_2satake}
\begin{tabular}{cc}
\hline
\hline
Satake diagram $S_{1}$ & Satake diagram $S_{2}$ \\
\hline
\hline
\begin{xy}
\ar@{-}(-20,0)*++!D{\alpha_{6}}*{\circ}="a6";(-10,0)*++!D{\alpha_{5}}*{\circ}="a5"
\ar@{-}"a5";(0,0)*++!DL{\alpha_{4}}*{\circ}="a4"
\ar@{-}"a4";(10,0)*++!D{\alpha_{3}}*{\circ}="a3"
\ar@{-}"a3";(20,0)*++!D{\alpha_{1}}*{\circ}="a1"
\ar@{-}"a4";(0,10)*++!D{\alpha_{2}}*{\circ}="a2"
%%% arrows
%\ar@/_4mm/@{<->} "a6";"a1"
\end{xy}
&
\begin{xy}
\ar@{-}(-20,0)*++!D{\alpha_{6}}*{\circ}="a6";(-10,0)*++!D{\alpha_{5}}*{\bullet}="a5"
\ar@{-}"a5";(0,0)*++!DL{\alpha_{4}}*{\bullet}="a4"
\ar@{-}"a4";(10,0)*++!D{\alpha_{3}}*{\bullet}="a3"
\ar@{-}"a3";(20,0)*++!D{\alpha_{1}}*{\circ}="a1"
\ar@{-}"a4";(0,10)*++!D{\alpha_{2}}*{\circ}="a2"
%%% arrows
%\ar@/^4mm/@{<->} "a3";"a5"
\ar@/_6mm/@{<->} "a6";"a1"
\end{xy}\\[4ex]
\hline
\hline
\end{tabular}
\end{table}
\end{ex}

We can determine $(\tilde{\Sigma},\Sigma,W;m,n)=(\mbox{II-}BC_{r};m,n)$
for $(SU(2m),Sp(m),S(U(a)\times U(b)))$
($n>2a$),
$(SO(2m),SO(a)\times SO(b), U(m))$
($a$: odd, $m=a$)
by a similar argument as in Example \ref{ex:EI-III_determ}.

Finally, we explain
our determination of $(\tilde{\Sigma},\Sigma,W;m,n)=(\mbox{III-}BC_{r})$.
It is shown that
the symmetric triads
with multiplicities
becomes $(\mbox{III-}BC_{r})$-type
for the following commutative compact symmetric triads,
which is verified by using Lemma \ref{lem:tildeSigma_SigmaW_determ_sufficient}, (2):
\begin{enumerate}[C{a}se 1:]
\item $(SU(2m),Sp(m),S(U(a)\times U(b)))$ ($a$: odd, $m=a$),
$(Sp(n),U(n),Sp(a)\times Sp(b))$ ($n>2a$),
$(E_{6},SU(6)\cdot SU(2),F_{4})$,
$(F_{4},SU(2)\cdot Sp(3), SO(9))$;
\item $(SU(2m),Sp(m),S(U(a)\times U(b)))$ ($a$: even, $m>a$),
$(E_{6},SO(10)\cdot U(1),F_{4})$.
\end{enumerate}
We can determine their multiplicities
by Proposition \ref{pro:delta_compact} for Case 1
and by Proposition \ref{pro:delta_noncompact} for Case 2.
We omit the details of the determination of their multiplicities.

From the above argument,
we have the following theorem.

\begin{thm}\label{thm:determ_symmetrictriads}
Table \ref{table:symmtriext} exhibits explicit descriptions of the types of the symmetric triads with multiplicities
corresponding to commutative compact symmetric triads $(G,\theta_{1},\theta_{2})$
in the case when $G$ is simple and $\theta_{1}\not\sim\theta_{2}$.
\end{thm}

%\begin{minipage}{0.974\textwidth}
%\vspace{5\baselineskip}
\begin{table}[!!ht]
%\footnotesize
\centering
\caption{The symmetric triads $(\tilde{\Sigma},\Sigma,W;m,n)$
with multiplicities corresponding to commutative compact symmetric triads $(G,K_{1},K_{2})$ such that $G$ is simple and that $\theta_{1}\not\sim\theta_{2}$}\label{table:symmtriext}
\renewcommand\arraystretch{1.5}
\scalebox{0.77}{
\begin{tabular}{cccc}
\hline
\hline
$(G, K_{1},K_{2})$ & $(\tilde{\Sigma},\Sigma,W)$ & $m,n$ & Remark \\
\hline
\hline
$(SU(2m), SO(2m), Sp(m))$ & $(\mbox{III-}A_{m-1})$ & $2,2$ & \\
\hline
\multirow{3}{*}{$(SU(n), SO(n), S(U(a) \times U(b)))$} & $(\mbox{I-}C_{a})$ & $\begin{cases}1,1 & (\text{short})\\1,0 & (\text{long})\end{cases}$ & $n=2a$\\
& $(\mbox{II-}BC_{a})$ & $\begin{cases}n-2a,n-2a & (\text{shortest})\\1,1 & (\text{middle})\\0,1 & (\text{longest})\end{cases}$ & $n>2a$ \\
\hline
\multirow{5}{*}{$(SU(2m), Sp(m), S(U(a)\times U(b)))$} & $(\mbox{III-}C_{a/2})$ & $\begin{cases}
4,4 & (\text{short})\\
3,1 & (\text{long})
\end{cases}$ & $\begin{cases}a:\text{even},\\m=a\end{cases}$\\
& $(\mbox{III-}BC_{(a-1)/2})$ & $\begin{cases}4,4 & (\text{shortest})\\4,4 & (\text{middle})\\1,3 & (\text{longest})\end{cases}$ & $\begin{cases}a:\text{odd},\\m=a\end{cases}$ \\
%\cline{2-4}
& $(\mbox{III-}BC_{a/2})$ & $\begin{cases}4(m-a),4(m-a) & (\text{shortest})\\4,4 & (\text{middle})\\3,1 & (\text{longest})\end{cases}$ & $\begin{cases}a:\text{even},\\m>a\end{cases}$ \\
\hline
$(SU(n), S(U(a)\times U(b)),S(U(c) \times U(d)))$ & $(\mbox{I-}BC_{a}\mbox{-}A_{1}^{a})$ & $\begin{cases}2(d-a),2(c-a) & (\text{shortest})\\2,0 & (\text{middle})\\1,0 & (\text{longest})\end{cases}$ & $a<c \leq d<b$ \\
\hline
$(SO(n), SO(a)\times SO(b), SO(c) \times SO(d))$ & $(\mbox{I-}B_{a})$ & $\begin{cases}d-a,c-a & (\text{short})\\1,0 & (\text{long})\end{cases}$ &  $a<c \leq d<b$ \\
\hline
\multirow{5}{*}{$(SO(2m), SO(a)\times SO(b), U(m))$} & $(\mbox{I-}C_{a/2})$ & $\begin{cases}2,2 & (\text{short})\\1,0 & (\text{long})\end{cases}$ & $\begin{cases}a:\text{even},\\m=a\end{cases}$\\
& $(\mbox{II-}BC_{(a-1)/2})$ & $\begin{cases}2,2 & (\text{shortest})\\2,2 & (\text{middle})\\0,1 & (\text{longest})\end{cases}$ & $\begin{cases}a:\text{odd},\\m=a\end{cases}$ \\
& $(\mbox{I-}BC_{a/2}\mbox{-}B_{a/2})$ & $\begin{cases}2(m-a),2(m-a) & (\text{shortest})\\2,2 & (\text{middle})\\1,0 & (\text{longest})\end{cases}$ & $\begin{cases}a:\text{even},\\m>a\end{cases}$ \\
\hline
$(SO(4m), U(2m), U(2m)')$ & $(\mbox{I-}BC_{m-1}\mbox{-}A_{1}^{m-1})$ & $\begin{cases}4,4 & (\text{shortest})\\4,0 & (\text{middle})\\1,0 & (\text{longest})\end{cases}$ & \\
\hline
\multirow{3}{*}{$(Sp(n), U(n), Sp(a)\times Sp(b))$} & $(\mbox{III-}C_{a})$ & $\begin{cases}2,2 & (\text{short})\\2,1 & (\text{long})\end{cases}$ & $n=2a$\\
%\cline{2-4}
& $(\mbox{III-}BC_{a})$ & $\begin{cases}2(n-2a),2(n-2a) & (\text{shortest})\\2,2 & (\text{middle})\\1,2 & (\text{longest})\end{cases}$ & $n>2a$ \\
\hline
$(Sp(n), Sp(a)\times Sp(b), Sp(c)\times Sp(d))$ & $(\mbox{I-}BC_{a}\mbox{-}A_{1}^{a})$ & $\begin{cases}4(d-a),4(c-a) & (\text{shortest})\\4,0 & (\text{middle})\\3,0 & (\text{longest})\end{cases}$ & $a<c \leq d<b$ \\
\hline
\hline
\end{tabular}
}
\renewcommand\arraystretch{1.0}
\end{table}
%\end{minipage}

\begin{table}[!!ht]
\centering
%\footnotesize
%\scriptsize
%\tiny
%\centering
\contcaption{(continued)}
%\scalebox{0.8}{
\renewcommand\arraystretch{1.5}
\begin{tabular}{ccc}
\hline
\hline
$(G,K_{1},K_{2})$ & $(\tilde{\Sigma},\Sigma,W)$ & $m,n$  \\
\hline
\hline
$(E_{6}, Sp(4), SU(6)\cdot SU(2))$ & $(\mbox{I-}F_{4})$ & $\begin{cases}1,1 & (\text{short})\\1,0 & (\text{long})\end{cases}$ \\
\hline
$(E_{6}, Sp(4), SO(10)\cdot U(1))$ & $(\mbox{II-}BC_{2})$ & $\begin{cases}4,4 & (\text{shortest})\\3,3 & (\text{middle})\\0,1 & (\text{longest})\end{cases}$  \\
\hline
$(E_{6}, Sp(4), F_{4})$ & $(\mbox{III-}A_{2})$ & $4,4$  \\
\hline
$(E_{6}, SU(6)\cdot SU(2), SO(10)\cdot U(1))$ & $(\mbox{I-}BC_{2}\mbox{-}B_{2})$ & $\begin{cases}4,4 & (\text{shortest})\\4,2 & (\text{middle})\\1,0 & (\text{longest})\end{cases}$ \\
\hline
$(E_{6}, SU(6)\cdot SU(2), F_{4})$ & $(\mbox{III-}BC_{1})$ & $\begin{cases}8,8 & (\text{short})\\3,5 & (\text{long})\end{cases}$ \\
\hline
$(E_{6}, SO(10)\cdot U(1), F_{4})$ & $(\mbox{III-}BC_{1})$ & $\begin{cases}8,8 & (\text{short})\\7,1 & (\text{long})\end{cases}$  \\
\hline
$(E_{7}, SU(8), SO(12)\cdot SU(2))$ & $(\mbox{I-}F_{4})$ & $\begin{cases}2,2 & (\text{short})\\1,0 & (\text{long})\end{cases}$  \\
\hline
$(E_{7}, SU(8), E_{6}\cdot U(1))$ & $(\mbox{I-}C_{3})$ & $\begin{cases}4,4 & (\text{short})\\1,0 & (\text{long})\end{cases}$  \\
\hline
$(E_{7}, SO(12)\cdot SU(2), E_{6}\cdot U(1))$ & $(\mbox{I-}BC_{2}\mbox{-}B_{2})$ & $\begin{cases}8,8 & (\text{shortest})\\6,2 & (\text{middle})\\1,0 & (\text{longest})\end{cases}$  \\
\hline
$(E_{8}, SO(16), E_{7}\cdot SU(2))$ & $(\mbox{I-}F_{4})$ & $\begin{cases}4,4 & (\text{short})\\1,0 & (\text{long})\end{cases}$  \\
\hline
$(F_{4}, SU(2)\cdot Sp(3), SO(9))$ & $(\mbox{III-}BC_{1})$ & $\begin{cases}4,4 & (\text{short})\\3,4 & (\text{long})\end{cases}$  \\
\hline
\hline
\end{tabular}
\renewcommand\arraystretch{1.0}
%} %%% scalebox
\end{table}

As an application of Theorem \ref{thm:determ_symmetrictriads},
we have the following corollary.

\begin{cor}\label{cor:simsim_simplyconnected}
Under the same settings as in Proposition \ref{pro:cstsim_symmsim},
if $(\tilde{\Sigma},\Sigma,W;m,n)\sim
(\tilde{\Sigma}',\Sigma',W';m',n')$,
then $(G,\theta_{1},\theta_{2})$
is isomorphic to $(G,\theta_{1}',\theta_{2}')$
or $(G,\theta_{2}',\theta_{1}')$
with respect to $\sim$ at the Lie algebra level.
Then, in the case when $G$ is simply-connected,
$(G,\theta_{1},\theta_{2})\sim(G,\theta_{1}',\theta_{2}')$
or $(G,\theta_{1}',\theta_{2}')$
holds.
\end{cor}

This corollary follows from the classification of commutative compact symmetric triads 
with respect to $\sim$
and Theorem \ref{thm:determ_symmetrictriads}

\begin{ex}
Let us consider the case when:
\begin{equation*}
\begin{cases}
(G,\theta_{1},\theta_{2})=(SO(8),SO(4)\times SO(4),U(4)),\\
(G,\theta_{1}',\theta_{2}')=(SO(8),SO(4)\times SO(4),SO(2)\times SO(6)).
\end{cases}
\end{equation*}
It follows from \cite[Corollary 5.23, (1)]{BI}
that
$(SO(8),SO(4)\times SO(4),U(4))$
and $(SO(8),SO(4)\times SO(4),SO(2)\times SO(6))$
are isomorphic with respect to $\sim$ at the Lie algebra level.
Furthermore,
from Example \ref{ex:stmSI}, (1),
we have
\[
(\tilde{\Sigma},\Sigma,W;m,n)=(\mbox{I-}C_{2};m,n)
\equiv(\mbox{I-}B_{2};m',n')=(\tilde{\Sigma}',\Sigma',W';m',n').
\]
However, 
$(SO(8),SO(4)\times SO(4),U(4))$
and $(SO(8),SO(4)\times SO(4),SO(2)\times SO(6))$
are \textit{not} isomorphic with respect to $\sim$
at the Lie group level.
Indeed, the centers of $U(4)$
and $SO(2)\times SO(6)$
are $SO(2)$ and $SO(2) \times \mathbb{Z}_{2}$,
respectively.
\end{ex}

\subsection{Symmetric triads of type (IV)
with multiplicities for commutative compact symmetric triads}\label{sec:typeIV_CST}

In this subsection,
we study the symmetric triads
of type (IV) with multiplicities
for commutative compact symmetric triads
$(G,\theta_{1},\theta_{2})$ with $\theta_{1}\sim\theta_{2}$.
We first explain a construction
of a symmetric triad of type (IV)
with multiplicities from $(G,\theta_{1},\theta_{2})$
in a similar manner
as in Section \ref{sec:ccst_symm}.
From $\theta_{1}\sim\theta_{2}$,
there exists $g\in G$ satisfying
$\theta_{2}=\tau_{g}\theta_{1}\tau_{g}^{-1}$.
Let $\mathfrak{a}$ be a maximal abelian
subspace of $\mathfrak{m}_{1}$.
We set $A=\exp(\mathfrak{a})\subset G$.
Since we have $G=K_{1}AK_{1}$,
there exist $k,k'\in K_{1}$
and $Y\in\mathfrak{a}$ satisfying
$g=k\exp(Y)k'$,
from which
$(G,\theta_{1},\theta_{2})\equiv
(G,\theta_{1},\tau_{\exp(Y)}\theta_{1}\tau_{\exp(Y)}^{-1})$ holds.
Let $(\bar{\Sigma};\bar{m})$
denote the restricted root system
with multiplicity for
$(G,\theta_{1})$ associated with $\mathfrak{a}$.
By the commutativity of $\theta_{1}$
and $\theta_{2}$,
$Y$ is an element of
the lattice $\Gamma$ for $\bar{\Sigma}$
as in \eqref{eqn:Gamma}.
Without loss of generalities,
we may assume that $\theta_{2}=\tau_{\exp(Y)}\theta_{1}\tau_{\exp(Y)}^{-1}$ for $Y\in\Gamma$.
Then we have $\mathfrak{k}_{2}=e^{\mathrm{ad}(Y)}\mathfrak{k}_{1}$
and $\mathfrak{m}_{2}=e^{\mathrm{ad}(Y)}\mathfrak{m}_{1}$.
This implies that $\mathfrak{a}$
is also a maximal abelian subspace of $\mathfrak{m}_{2}$.
In particular, $\mathfrak{a}$ gives a maximal abelian subspace of $\mathfrak{m}_{1}\cap\mathfrak{m}_{2}$.
As in Section \ref{sec:ccst_symm},
we can construct $(\tilde{\Sigma},\Sigma,W;m,n)$
from $(G,\theta_{1},\theta_{2})$ and $\mathfrak{a}$.

\begin{pro}\label{pro:symmiv_theta12sim}
Assume that $G$ is simple and
that $\theta_{1}\sim\theta_{2}$ holds.
Under the above settings,
the obtaining
$(\tilde{\Sigma},\Sigma,W;m,n)$
is the symmetric triad 
$(\bar{\Sigma},\Sigma_{Y},W_{Y};m,n)$
of type (IV) with multiplicities as in Definition
\ref{def:symiv}, that is,
\begin{equation}
\Sigma_{Y}=\{\lambda\in\bar{\Sigma}\mid
\INN{\lambda}{2Y}\in 2\pi\mathbb{Z}\},\quad
W_{Y}=\bar{\Sigma}-\Sigma_{Y},
\end{equation}
and $(m(\lambda),n(\lambda))=(\bar{m}_{\lambda},0)$
for $\lambda\in\Sigma_{Y}$;
$(m(\lambda),n(\lambda))=(0,\bar{m}_{\lambda})$
for $\lambda\in W_{Y}$.
In particular,
in the case when $Y=0$,
we have
$\theta_{1}=\theta_{2}$
and $(\tilde{\Sigma},\Sigma,W)=
(\bar{\Sigma},\bar{\Sigma},\emptyset)$.
\end{pro}

In order to prove it
we need some preparations.
For $\lambda\in\bar{\Sigma}$,
we define the subspace $\mathfrak{k}_{\lambda}$
of $\mathfrak{k}_{1}$
and $\mathfrak{m}_{\lambda}$
of $\mathfrak{m}_{1}$ as follows:
\begin{align}
\mathfrak{k}_{\lambda}
&=\{X\in\mathfrak{k}_{1}\mid
[H,[H,X]]=-\INN{\lambda}{H}^{2}X,H\in\mathfrak{a}\},\\
\mathfrak{m}_{\lambda}
&=\{X\in\mathfrak{m}_{1}\mid
[H,[H,X]]=-\INN{\lambda}{H}^{2}X,H\in\mathfrak{a}\}.
\end{align}
Then we have root space decompositions
of $\mathfrak{k}_{1}$
and $\mathfrak{m}_{1}$ as follows:
\begin{equation}
\mathfrak{k}_{1}
=\mathfrak{k}_{0}\oplus\sum_{\lambda\in\bar{\Sigma}^{+}}\mathfrak{k}_{\lambda},\quad
\mathfrak{m}_{1}
=\mathfrak{a}\oplus\sum_{\lambda\in\bar{\Sigma}^{+}}\mathfrak{m}_{\lambda},
\end{equation}
where $\mathfrak{k}_{0}=\{X\in\mathfrak{k}_{1}\mid
[\mathfrak{a},X]=\{0\}\}$ is a subalgebra of $\mathfrak{k}_{1}$
and $\bar{\Sigma}^{+}$
is the set of positive roots of $\bar{\Sigma}$
for an ordering on $\bar{\Sigma}$.
Let $\mathfrak{t}$ be a maximal abelian subalgebra
of $\mathfrak{g}$ containing $\mathfrak{a}$,
and $\Delta$ denote the root system of $\mathfrak{g}$
with respect to $\mathfrak{t}$.
Since $\mathfrak{t}$ is $\theta_{1}$-invariant,
$\sigma=-d\theta_{1}|_{\mathfrak{t}}$
gives an involutive linear isometry on $\mathfrak{t}$
satisfying $\sigma(\Delta)=\Delta$.
We set $\Delta_{0}=\{\alpha\in\Delta\mid
\INN{\alpha}{\mathfrak{a}}=\{0\}\}=\{\alpha\in\Delta\mid
\mathrm{pr}(\alpha)=0\}$,
where $\mathrm{pr}:\mathfrak{t}\to\mathfrak{a}$
denotes the orthogonal projection
of $\mathfrak{t}=(\mathfrak{t}\cap\mathfrak{k}_{1})\oplus\mathfrak{a}$.

\begin{lem}[{\cite[p.~89, Lemma 1]{Takeuchi}}]\label{lem:Takeuchi_p.89}
For each $\alpha\in\Delta-\Delta_{0}$,
there exist $S_{\alpha}\in\mathfrak{k}_{1}$
and $T_{\alpha}\in\mathfrak{m}_{1}$ such that
$\{S_{\alpha}\mid \alpha\in\Delta,\mathrm{pr}(\alpha)=\lambda\}$
and $\{T_{\alpha}\mid \alpha\in\Delta,\mathrm{pr}(\alpha)=\lambda\}$
are respectively orthogonal bases of $\mathfrak{k}_{\lambda}$
and $\mathfrak{m}_{\lambda}$ satisfying
the following relations:
\begin{equation}\label{eqn:ST_alpha}
[H,S_{\alpha}]
=\INN{\alpha}{H}T_{\alpha},\quad
[H,T_{\alpha}]
=-\INN{\alpha}{H}S_{\alpha},\quad
H\in\mathfrak{a},
\end{equation}
and $[S_{\alpha},T_{\alpha}]=\lambda$.
\end{lem}

We are ready to prove Proposition \ref{pro:symmiv_theta12sim}.

\begin{proof}[Proof of Proposition \ref{pro:symmiv_theta12sim}]
Clearly,
we have $\tilde{\Sigma}=\bar{\Sigma}$.
By using \eqref{eqn:ST_alpha},
we get
\begin{equation}
e^{\mathrm{ad}(Y)}S_{\alpha}
=\cos\INN{\alpha}{Y}S_{\alpha}+\sin\INN{\alpha}{Y}T_{\alpha},\quad
e^{\mathrm{ad}(Y)}T_{\alpha}
=-\sin\INN{\alpha}{Y}S_{\alpha}+\cos\INN{\alpha}{Y}T_{\alpha}.
\end{equation}
Then $\mathfrak{k}_{2}$ and $\mathfrak{m}_{2}$
have the following decompositions:
\begin{align}
&\mathfrak{k}_{2}=
e^{\mathrm{ad}(Y)}\mathfrak{k}_{1}=\mathfrak{k}_{0}\oplus
\sum_{\lambda\in\bar{\Sigma}^{+};\,\INN{\lambda}{2Y}\in 2\pi\mathbb{Z}}\mathfrak{k}_{\lambda}\oplus
\sum_{\lambda\in\bar{\Sigma}^{+};\,\INN{\lambda}{2Y}\in\pi+2\pi\mathbb{Z}}\mathfrak{m}_{\lambda},\\
&\mathfrak{m}_{2}=
e^{\mathrm{ad}(Y)}\mathfrak{m}_{1}=\mathfrak{a}\oplus
\sum_{\lambda\in\bar{\Sigma}^{+};\,\INN{\lambda}{2Y}\in 2\pi\mathbb{Z}}\mathfrak{m}_{\lambda}\oplus
\sum_{\lambda\in\bar{\Sigma}^{+};\,\INN{\lambda}{2Y}\in\pi+ 2\pi\mathbb{Z}}\mathfrak{k}_{\lambda}.
\end{align}
This yields the following decompositions:
\begin{equation}
\begin{array}{ll}
 \displaystyle{{\mathfrak k}_1\cap {\mathfrak k}_2={\mathfrak k}_0\oplus \sum_{\langle\lambda,2Y\rangle\in 2\pi\mathbb{Z}}{\mathfrak k}_\lambda},& 
 \displaystyle{{\mathfrak m}_1\cap {\mathfrak m}_2=\mathfrak a\oplus \sum_{\langle\lambda,2Y\rangle\in2\pi\mathbb{Z}}{\mathfrak m}_\lambda},
 \\
\displaystyle{{\mathfrak k}_1\cap {\mathfrak m}_2=\sum_{\langle\lambda,2Y\rangle\in\pi+2\pi\mathbb{Z}}{\mathfrak m}_\lambda}, & \displaystyle{{\mathfrak m}_1\cap {\mathfrak k}_2=\sum_{\langle\lambda,2Y\rangle\in\pi+2\pi\mathbb{Z}}{\mathfrak k}_\lambda}.
\end{array}
\end{equation}
From above,
we have the following descriptions of $\Sigma, W$:
\[
\begin{cases}
\Sigma=\{\lambda\in\bar{\Sigma}\mid \langle \lambda ,2Y\rangle \in 2\pi\mathbb{Z}\}=\Sigma_{Y},\\
W=\{\alpha\in\bar{\Sigma}\mid \langle\alpha ,2Y\rangle \in\pi+2\pi\mathbb{Z}\}=\bar{\Sigma}-\Sigma_{Y}.
\end{cases}
\]
Clearly,
the multiplicities $m,n$
are in the statement.
Thus, we have the assertion.
\end{proof}

\begin{rem}
The restricted root system of compact symmetric pairs $(G,\theta_{1})$ are given in \cite[TABLE VI, Chapter X]{Helgason} in the case when $G$ is simple.
From this, we can easily derive the explicit description of the
symmetric triad
$(\tilde{\Sigma},\Sigma,W;m,n)\sim(\bar{\Sigma},\bar{\Sigma},\emptyset;\bar{m},0)$
of type (IV) with multiplicities
corresponding to $(G,\theta_{1},\theta_{2})$ with $\theta_{1}\sim\theta_{2}$.
\end{rem}

We consider analogues
of Proposition
\ref{pro:cstsim_symmsim}
and Corollary \ref{cor:simsim_simplyconnected}.
Namely, we show the following theorem.

\begin{thm}
Let $(G,\theta_{1},\theta_{2})$
and $(G,\theta_{1}',\theta_{2}')$
be two commutative compact symmetric triads
with $\theta_{1}\sim\theta_{2}$
and $\theta_{1}'\sim\theta_{2}'$.
Let $(\tilde{\Sigma},\Sigma,W;m,n)$
and $(\tilde{\Sigma}',\Sigma',W';m',n')$
denote the symmetric triads of type (IV)
with multiplicities corresponding to
$(G,\theta_{1},\theta_{2})$
and $(G,\theta_{1}',\theta_{2}')$, respectively.
Then $(G,\theta_{1},\theta_{2})\sim(G,\theta_{1}',\theta_{2}')$ yields
$(\tilde{\Sigma},\Sigma,W;m,n)\sim(\tilde{\Sigma}',\Sigma',W';m',n')$.
Conversely,
if $(\tilde{\Sigma},\Sigma,W;m,n)\sim(\tilde{\Sigma}',\Sigma',W';m',n')$,
then $(G,\theta_{1},\theta_{2})$ is isomorphic to $(G,\theta_{1}',\theta_{2}')$
with respect to $\sim$ at the Lie algebra level.
Then, in the case when $G$ is simply-connected,
we have $(G,\theta_{1},\theta_{2})\sim(G,\theta_{1}',\theta_{2}')$.
\end{thm}

Here, we remark that
if $\theta_{1}\sim\theta_{2}$,
then $(G,\theta_{1},\theta_{2})\sim (G,\theta_{2},\theta_{1})$ holds.

\begin{proof}
If $(G,\theta_{1},\theta_{2})\sim(G,\theta_{1}',\theta_{2}')$, then $(G,\theta_{1})$
is isomorphic to $(G,\theta_{1}')$
as compact symmetric pairs.
This implies
that the restricted root system
$(\bar{\Sigma}',\bar{m}')$
with multiplicity
for $(G,\theta_{1}')$
is isomorphic to $(\bar{\Sigma},\bar{m})$,
from which
$(\tilde{\Sigma}',\Sigma',W';m',n')$
is isomorphic to $(\tilde{\Sigma},\Sigma,W;m,n)$
with respect to $\sim$ (see, Definition \ref{def:symivequiv}).

Conversely, we suppose that
$(\tilde{\Sigma},\Sigma,W;m,n)\sim(\tilde{\Sigma}',\Sigma',W';m',n')$ holds.
Then $(\bar{\Sigma},\bar{m})$
is isomorphic to
$(\bar{\Sigma}',\bar{m}')$,
from which $(G,\theta_{1})$
is isomorphic to $(G,\theta_{1}')$
at the Lie algebra level.
The following argument is
valid at the Lie algebra level:
\begin{equation}
(G,\theta_{1},\theta_{2})
\sim(G,\theta_{1},\theta_{1})
\equiv(G,\theta_{1}',\theta_{1}')
\sim(G,\theta_{1}',\theta_{2}').
\end{equation}
Thus, we have the assertion.
\end{proof}

\subsection{The classification of commutative compact symmetric triads with respect to $\equiv$}\label{sec:cst_classify_equiv}

In this subsection,
we consider the classification
problem for commutative compact symmetric
triads with respect to $\equiv$.
Let $G$ be a compact connected
simple Lie group.
Let $(G,\theta_{1},\theta_{2})$
and $(G,\theta_{1}',\theta_{2}')$
be two commutative compact symmetric triads
and,
$(\tilde{\Sigma},\Sigma,W;m,n)$
and $(\tilde{\Sigma}',\Sigma',W';m',n')$
denote the symmetric triads
with multiplicities corresponding to
$(G,\theta_{1},\theta_{2})$
and $(G,\theta_{1}',\theta_{2}')$, respectively.
As shown in Lemma
\ref{lem:cstequiv_symmequiv},
$(G,\theta_{1},\theta_{2})\equiv
(G,\theta_{1}',\theta_{2}')$
yields
$(\tilde{\Sigma},\Sigma,W;m,n)\equiv
(\tilde{\Sigma}',\Sigma',W';m',n')$.
Here,
we note that
$(\tilde{\Sigma},\Sigma,W;m,n)$
and $(\tilde{\Sigma}',\Sigma',W';m',n')$
are maybe of type (IV)
since we do not assume that $\theta_{1}\not\sim\theta_{2}$.
Furthermore, 
it can be verified
that an analogue of Lemma
\ref{lem:cstequiv_symmequiv}
holds in the case when $\theta_{1}\sim\theta_{2}$.

Conversely,
we show that
the symmetric triads with multiplicities
determine commutative compact symmetric
triads at the Lie algebra level,
that is, the following theorem holds.

\begin{thm}\label{thm:equivequiv_ok}
Let $(G,\theta_{1},\theta_{2})$
and $(G,\theta_{1}',\theta_{2}')$
be two commutative compact symmetric triads.
Let $(\tilde{\Sigma},\Sigma,W;m,n)$
and $(\tilde{\Sigma}',\Sigma',W';m',n')$
denote
the symmetric triads
with multiplicities corresponding to
$(G,\theta_{1},\theta_{2})$
and $(G,\theta_{1}',\theta_{2}')$, respectively.
Then if $(\tilde{\Sigma},\Sigma,W;m,n)\equiv(\tilde{\Sigma}',\Sigma',W';m',n')$,
then $(G,\theta_{1},\theta_{2})$ is isomorphic to $(G,\theta_{1}',\theta_{2}')$
or $(G,\theta_{2}',\theta_{1}')$
with respect to $\equiv$ at the Lie algebra level.
Then, in the case when $G$ is simply-connected,
we have $(G,\theta_{1},\theta_{2})\equiv(G,\theta_{1}',\theta_{2}')$
or $(G,\theta_{2}',\theta_{1}')$.
\end{thm}

\begin{proof}
From the assumption,
we have $(\tilde{\Sigma},\Sigma,W;m,n)\sim(\tilde{\Sigma}',\Sigma',W';m',n')$.
Then, it follows from Corollary 
\ref{cor:simsim_simplyconnected}
that $(G,\theta_{1},\theta_{2})\sim
(G,\theta_{1}',\theta_{2}')$
or $(G,\theta_{2}',\theta_{1}')$
at the Lie algebra level.
For simplicity,
we may assume that
$(G,\theta_{1},\theta_{2})\sim
(G,\theta_{1}',\theta_{2}')$
at the Lie algebra level.
This implies that
there exists $Y\in\Gamma$
such that $(G,\theta_{1}',\theta_{2}')\equiv
(G,\theta_{1},\tau_{\exp(Y)}\theta_{2}\tau_{\exp(Y)}^{-1})$ at the Lie algebra level.
We denote by
$(\tilde{\Sigma}'',\Sigma'',W'';m'',n'')$
the symmetric triad with multiplicities
corresponding to
$(G,\theta_{1},\tau_{\exp(Y)}\theta_{2}\tau_{\exp(Y)}^{-1})$.
Clearly, we have
$(\tilde{\Sigma}'',\Sigma'',W'';m'',n'')
=(\tilde{\Sigma},\Sigma_{Y},W_{Y};m,n)\equiv
(\tilde{\Sigma},\Sigma,W;m,n)$,
from which we have $\INN{\alpha}{2Y}\in2\pi\mathbb{Z}$ for all $\alpha\in\tilde{\Sigma}$.
This implies that $e^{\mathrm{ad}(Y)}$
gives the identity transformation on $\mathfrak{g}$.
Thus, we have
$(G,\theta_{1},\tau_{\exp(Y)}\theta_{2}\tau_{\exp(Y)}^{-1})= (G,\theta_{1},\theta_{2})$
at the Lie algebra level.
From the above arguments,
we have completed the proof.
\end{proof}

Based on this theorem,
we can classify commutative compact symmetric triads
with respect to $\equiv$ by using
the classifications
for
commutative compact symmetric triads
with respect to $\sim$
and for abstract symmetric triads
with multiplicities with respect to $\equiv$.

\begin{ex}\label{ex:E6Sp4SU6SU2_equiv}
Let us consider the case when $(G,K_{1},K_{2})=(E_{6}, Sp(4), SU(6)\cdot SU(2))$.
As shown in Example \ref{ex:EI-IV},
the corresponding symmetric triad
with multiplicities
is isomorphic to $(\tilde{\Sigma},\Sigma,W;m,n)=(\mbox{I-}F_{4};m,n)$ with respect to $\sim$.
It follows from
Example \ref{ex:if4}
that the isomorphism class
of $(\tilde{\Sigma},\Sigma,W;m,n)$
with respect to $\sim$
consists of two elements
$(\mbox{I-}F_{4};m,n)$
and $(\mbox{I'-}F_{4};m',n')$.
Thus,
by Theorem \ref{thm:equivequiv_ok},
we find that
any commutative compact symmetric
triad $(G,\theta_{1},\theta_{2})$
with $K_{1}=Sp(4), K_{2}=SU(6)\cdot SU(2)$
is isomorphic to one of
commutative compact symmetric triads
whose 
$(\tilde{\Sigma},\Sigma,W;m,n)$ is $(\mbox{I-}F_{4};m,n)$
or $(\mbox{I'-}F_{4};m',n')$ with respect to $\equiv$.
\end{ex}

We can derive the classification of the isomorphism classes
of commutative compact symmetric triads with respect to $\equiv$
by applying a similar argument as in Example \ref{ex:E6Sp4SU6SU2_equiv} to that of 
commutative compact symmetric triads with respect to $\sim$.
There is an enormous number of the isomorphism classes for $\equiv$,
so their list is omitted in this paper.
On the other hand,
we have given the one-to-one correspondence between
the isomorphism classes of commutative compact symmetric triads
with respect to $\equiv$
and those of pseudo-Riemannian symmetric pairs in \cite{BIS}.
By characterizing the isomorphism classes of commutative compact symmetric triads
$(G,\theta_{1},\theta_{2})$
in terms of symmetric triads with multiplicities $(\tilde{\Sigma},\Sigma,W;m,n)$,
we can give a explicit description of this correspondence (see, \cite{BIS2} for the detail).
Furthermore, as discussed in \cite{BIS2},
we can determine the Lie algebras of
$G^{\theta_{1}\theta_{2}}$ and $K_{1}\cap K_{2}$
from $(\tilde{\Sigma},\Sigma,W;m,n)$.

\section{$\sigma$-actions and symmetric triads with multiplicities}\label{sec:sigma_action}

Let $U$ be a compact  connected semisimple Lie group
and $\sigma$ an involution.
The isometric action defined by
\begin{equation}
U\times U\to U;\,
(u,x)\mapsto ux\sigma(u)^{-1}.
\end{equation}
is called
a \textit{$\sigma$-action} on $U$.
It can be verified that
the $\sigma$-action
is isomorphic to the adjoint action of $U$
if and only if $\sigma$ is of inner-type.
The $\sigma$-actions
are realized as Hermann actions
on $U=(U\times U)/\triangle(U)$,
where $\triangle(U)=\{(g,g)\mid g\in U\}$.
We define a compact connected semisimple Lie group
by $G=U\times U$.
Let $\theta_{i}$ ($i=1,2$)
be the involution on $G$ defined by
\begin{equation}\label{eqn:2involution_sigma}
\theta_{1}(g,h)=(h,g),\quad
\theta_{2}(g,h)
=(\sigma^{-1}(h),\sigma(g))=(\sigma(h),\sigma(g)),\quad
g,h\in U.
\end{equation}
The fixed-point subgroup $G^{\theta_{i}}$
has the following expression:
\begin{equation}
K_{1}=\triangle(U),
\quad
K_{2}=\{(g,\sigma(g))\mid g\in U\}.
\end{equation}
Then it is shown that $G/K_{1}=(U\times U)/\triangle(U)$ is isomorphic to $U$
and that the $K_{2}$-action on $G/K_{1}$
is isomorphic to the $\sigma$-action on $U$.
From the view point of the geometry of Hermann actions,
it is essential to study $\sigma$-actions
in the case when $\sigma$ is of outer-type.
Since $\sigma$ is involutive,
$\theta_{1}$ and $\theta_{2}$ commute with each other.
The second author \cite{Ikawa2} constructed symmetric triads
with multiplicities from $(G,\theta_{1},\theta_{2})$,
and studied their properties.

The purpose of this section
is to study the corresponding symmetric triads
with multiplicities and double Satake diagrams
of $\sigma$-actions.
In the latter,
we will derive the relation
between the double Satake diagrams
and the Vogan diagrams for compact symmetric
pairs $(U,\sigma)$.
Based on it,
we will explain
the validity
of the terminologies
in Definitions \ref{dfn:imgroot_cpxroot}
and \ref{dfn:cptroot_noncptroot}.

\subsection{Symmetric triads with multiplicities for $\sigma$-actions}\label{sec:sact_symmtri}

Let $U$ be a compact connected semisimple Lie group
with Lie algebra $\mathfrak{u}$
and $\sigma$ be an involution on it.
The involution $\theta_{i}$ 
defined in \eqref{eqn:2involution_sigma}
induces an involution on $\mathfrak{g}$, which we write the same symbol $\theta_{i}$.
Then we have
\begin{equation}
\theta_{1}(X,Y)=(Y,X),
\quad
\theta_{2}(X,Y)=(\sigma(Y),\sigma(X)),
\quad
X,Y\in\mathfrak{u}.
\end{equation}
Let $\mathfrak{u}=\mathfrak{k}_{\sigma}\oplus\mathfrak{m}_{\sigma}$
be the canonical decomposition of $\mathfrak{u}$ for $\sigma$.
A direct calculation shows
\begin{align}
\mathfrak{k}_{1}\cap\mathfrak{k}_{2}
&=\{(X,X)\mid X\in\mathfrak{k}_{\sigma}\}, &
\mathfrak{m}_{1}\cap\mathfrak{m}_{2}
&=\{(X,-X)\mid X\in\mathfrak{k}_{\sigma}\},\\
\mathfrak{k}_{1}\cap\mathfrak{m}_{2}
&=\{(X,X)\mid X\in\mathfrak{m}_{\sigma}\}, &
\mathfrak{m}_{1}\cap\mathfrak{k}_{2}
&=\{(X,-X)\mid X\in\mathfrak{m}_{\sigma}\}.
\end{align}
By using 
the invariant inner product $\INN{\cdot}{\cdot}$
on $\mathfrak{u}$,
we define an invariant inner product on $\mathfrak{g}=\mathfrak{u}\oplus\mathfrak{u}$,
which we write the same symbol $\INN{\cdot}{\cdot}$,
as follows:
\begin{equation}
\INN{(X_{1},X_{2})}{(Y_{1},Y_{2})}
=\dfrac{1}{2}(\INN{X_{1}}{Y_{1}}+\INN{X_{2}}{Y_{2}}),
\quad X_{i},Y_{i}\in\mathfrak{u}.
\end{equation}
Then,
the linear isomorphism,
$\mathfrak{m}_{2} \to \mathfrak{m}_{\sigma};
\,(X,-X)\mapsto X$, becomes isometric with respect to the above inner products.
For a maximal abelian subalgebra $\bar{\mathfrak{a}}$
of $\mathfrak{k}_{\sigma}$,
\begin{equation}\label{eqn:am1m2_sigma}
\mathfrak{a}=\{(H,-H)\mid H\in\bar{\mathfrak{a}}\}\subset\mathfrak{m}_{1}\cap\mathfrak{m}_{2}
\end{equation}
is a maximal abelian subspace of $\mathfrak{m}_{1}\cap\mathfrak{m}_{2}$.
We can construct $(\tilde{\Sigma},\Sigma,W;m,n)$
of $\mathfrak{a}$ corresponding to $(G,\theta_{1},\theta_{2})=(U\times U,\theta_{1},\theta_{2})$
in a similar way as in Section \ref{sec:ccst_symm}.
Then the second author proved the following theorem.

\begin{thm}[{\cite[Theorem 1.14]{Ikawa2}}]
Assume that $U$ is simple and that $\sigma$ is of outer-type.
Then $(\tilde{\Sigma},\Sigma,W;m,n)$
satisfies the axiom
of symmetric triads with multiplicities stated in Definitions
\ref{dfn:symmtriad} and
\ref{dfn:multi}.
Furthermore,
we have $m(\lambda)=2, n(\alpha)=2$
for $\lambda\in\Sigma, \alpha\in W$.
\end{thm}

We call $(\tilde{\Sigma},\Sigma,W;m,n)$
the \textit{symmetric triad with multiplicities}
of $\mathfrak{a}$ corresponding to $(G,\theta_{1},\theta_{2})$.

Next,
we give
an analogue of Proposition \ref{pro:cstsim_symmsim}
for $\sigma$-actions.
Namely,
we show the following proposition.

\begin{pro}\label{pro:simsim_sigma}
Let $\theta_{2}'$
be an involution on $G$ satisfying $\theta_{1}\theta_{2}'=\theta_{2}'\theta_{1}$.
We denote by $(\tilde{\Sigma}',\Sigma',W')$
the symmetric triad corresponding to $(G,\theta_{1},\theta_{2}')$.
If $(G,\theta_{1},\theta_{2})\sim(G,\theta_{1},\theta_{2}')$,
then $(\tilde{\Sigma},\Sigma,W)\sim(\tilde{\Sigma}',\Sigma',W')$ holds.
\end{pro}

In order to prove this, we need some preparations.
By the definition,
we have
$[\bar{\mathfrak{a}},\mathfrak{k}_{\sigma}]\subset\mathfrak{k}_{\sigma}$
and $[\bar{\mathfrak{a}},\mathfrak{m}_{\sigma}]\subset\mathfrak{m}_{\sigma}$,
from which
$\bar{\mathfrak{a}}$ gives the adjoint representations
on $\mathfrak{k}_{\sigma}$
and $\mathfrak{m}_{\sigma}$.
We denote by $\bar{\Sigma}$
the set of nonzero weight of $\mathfrak{k}_{\sigma}$
with respect to $\bar{\mathfrak{a}}$,
by $\bar{W}$ that of $\mathfrak{m}_{\sigma}$.
In particular,
$\bar{\Sigma}$ satisfies the axiom of root system.
Then $\Sigma$ and $W$ are expressed as follows:
\begin{equation}
\Sigma=\{(\lambda,-\lambda)\mid \lambda\in\bar{\Sigma}\},
\quad
W=\{(\alpha,-\alpha)\mid \alpha\in\bar{W}\}.
\end{equation}
We write the root space decomposition of $\mathfrak{k}_{\sigma}$ with respect to $\bar{\mathfrak{a}}$
as
\begin{equation}
\mathfrak{k}_{\sigma}=\bar{\mathfrak{a}}
\oplus
\sum_{\lambda\in\bar{\Sigma}^{+}}(\mathbb{R}F_{\lambda}\oplus\mathbb{R}G_{\lambda}),
\end{equation}
where $\bar{\Sigma}^{+}$
denotes the set of positive roots of $\bar{\Sigma}$
with respect to some ordering,
and $F_{\lambda},G_{\lambda}\in\mathfrak{k}_{\sigma}$
are given by the following relations:
\begin{equation}\label{eqn:hfhg}
[H,F_{\lambda}]
=\INN{\lambda}{H}G_{\lambda},\quad
[H,G_{\lambda}]=-\INN{\lambda}{H}F_{\lambda},
\quad
H\in\bar{\mathfrak{a}}.
\end{equation}
We set $V(\mathfrak{m}_{\sigma})=\{X\in\mathfrak{m}_{\sigma}\mid [\bar{\mathfrak{a}},X]=\{0\}\}$.
In a similar manner,
we write the weight space decomposition
of $\mathfrak{m}_{\sigma}$ with respect to $\bar{\mathfrak{a}}$ as follows:
\begin{equation}
\mathfrak{m}_{\sigma}
=V(\mathfrak{m}_{\sigma})
\oplus\sum_{\alpha\in\bar{W}^{+}}(\mathbb{R}X_{\alpha}\oplus\mathbb{R}Y_{\alpha}),
\end{equation}
where $\bar{W}^{+}$
denotes the set of positive weights of $\bar{W}$
with respect to some ordering,
and $X_{\alpha},Y_{\alpha}\in\mathfrak{m}_{\sigma}$
are given by the following relations:
\begin{equation}\label{eqn:hxhy}
[H,X_{\alpha}]
=\INN{\alpha}{H}Y_{\alpha},\quad
[H,Y_{\alpha}]=-\INN{\alpha}{H}X_{\alpha},
\quad
H\in\bar{\mathfrak{a}}.
\end{equation}

We are ready to prove Proposition
\ref{pro:simsim_sigma}.

\begin{proof}[Proof of Proposition \ref{pro:simsim_sigma}]
From $(G,\theta_{1},\theta_{2})\sim(G,\theta_{1},\theta_{2}')$,
there exists $g\in G$ satisfying
$(G,\theta_{1},\theta_{2}')\equiv
(G,\theta_{1},\tau_{g}\theta_{2}\tau_{g}^{-1})$.
It follows from Theorem \ref{thm:Hermann}
that there exist $k_{i}\in K_{i}$
and $H\in\mathfrak{a}$
such that $g=k_{1}\exp(H)k_{2}$.
Then we have $(G,\theta_{1},\tau_{g}\theta_{2}\tau_{g}^{-1})
\equiv(G,\theta_{1},\tau_{\exp(H)}\theta_{2}\tau_{\exp(H)}^{-1})$.
The commutativity of $\theta_{1}$
and $\tau_{\exp(H)}\theta_{2}\tau_{\exp(H)}^{-1}$
implies that $H$ is an element in the lattice $\Gamma$
defined in \eqref{eqn:Gamma}.
Hence,
without loss of generalities,
we may assume that 
\begin{equation}\label{eqn:theta2'_sigma}
\theta_{2}'=\tau_{\exp(H)}\theta_{2}\tau_{\exp(H)}^{-1}
\end{equation}
for $H\in\Gamma$.

We write $H=(\bar{H},-\bar{H})$
for $\bar{H}\in\bar{\mathfrak{a}}$.
Since we have
\begin{equation}
\exp(H)
=\exp(\bar{H},-\bar{H})
=(\exp(\bar{H}),\exp(-\bar{H})),
\end{equation}
\eqref{eqn:theta2'_sigma}
is rewritten as follows
\begin{equation}
\theta_{2}'(g,h)
=(\tau_{\exp(\bar{H})}\sigma\tau_{\exp(\bar{H})}^{-1}(g),
\tau_{\exp(\bar{H})}\sigma\tau_{\exp(\bar{H})}^{-1}(h)),
\quad g,h\in U.
\end{equation}
In particular,
$\sigma'=\tau_{\exp(\bar{H})}\sigma\tau_{\exp(\bar{H})}^{-1}$
gives an involution on $U$.
Then we have
$d\sigma'=e^{\mathrm{ad}(\bar{H})}d\sigma e^{-\mathrm{ad}(\bar{H})}$.
In what follows,
we write the differential of $\sigma'$ as the same symbol $\sigma'$.
Let $\mathfrak{u}=\mathfrak{k}_{\sigma'}\oplus\mathfrak{m}_{\sigma'}$
be the canonical decomposition of $\mathfrak{u}$
for $\sigma'$.
Since $\mathfrak{m}_{1}\cap\mathfrak{m}_{2}'
=\{(X,-X)\mid X\in\mathfrak{k}_{\sigma'}\}$ holds,
$\mathfrak{a}$ is also a maximal abelian subspace of 
$\mathfrak{m}_{1}\cap\mathfrak{m}_{2}'$.
We denote by $(\tilde{\Sigma}',\Sigma',W')$
the symmetric triad of $\mathfrak{a}$ corresponding to $(G,\theta_{1},\theta_{2}')$.
Clearly, we get $\tilde{\Sigma}'=\tilde{\Sigma}$.
By using \eqref{eqn:hfhg},
it can be verified that,
\begin{equation}\label{eqn:FG_sigma}
e^{\mathrm{ad}(H)}F_{\lambda}
=\cos\INN{\lambda}{H}F_{\lambda}+\sin\INN{\lambda}{H}G_{\lambda},\quad
e^{\mathrm{ad}(H)}G_{\lambda}
=-\sin\INN{\lambda}{H}F_{\lambda}+\cos\INN{\lambda}{H}G_{\lambda}.
\end{equation}
From \eqref{eqn:hxhy}, we have:
\begin{equation}\label{eqn:XY_sigma}
e^{\mathrm{ad}(H)}X_{\alpha}
=\cos\INN{\alpha}{H}X_{\alpha}+\sin\INN{\alpha}{H}Y_{\alpha},\quad
e^{\mathrm{ad}(H)}Y_{\alpha}
=-\sin\INN{\alpha}{H}X_{\alpha}+\cos\INN{\alpha}{H}Y_{\alpha}.
\end{equation}
Then,
we obtain
\begin{align}
\mathfrak{k}_{\sigma'}
&=\bar{\mathfrak{a}}\oplus
\sum_{\lambda\in\bar{\Sigma}^{+};\,
\INN{\lambda}{2\bar{H}}\in2\pi\mathbb{Z}}(\mathbb{R}F_{\lambda}\oplus\mathbb{R}G_{\lambda})\oplus
\sum_{\alpha\in\bar{W}^{+};\,
\INN{\alpha}{2\bar{H}}\in\pi+2\pi\mathbb{Z}}
(\mathbb{R}X_{\alpha}\oplus\mathbb{R}Y_{\alpha}),\\
\mathfrak{m}_{\sigma'}
&=V(\mathfrak{m}_{\sigma})\oplus
\sum_{\lambda\in\bar{\Sigma}^{+};\,
\INN{\lambda}{2\bar{H}}\in\pi+2\pi\mathbb{Z}}(\mathbb{R}F_{\lambda}\oplus\mathbb{R}G_{\lambda})\oplus
\sum_{\alpha\in\bar{W}^{+};\,
\INN{\alpha}{2\bar{H}}\in2\pi\mathbb{Z}}
(\mathbb{R}X_{\alpha}\oplus\mathbb{R}Y_{\alpha}).
\end{align}
Hence we have:
\begin{align}
\Sigma'&=\{\lambda\in\Sigma\mid \INN{\lambda}{2H}\in2\pi\mathbb{Z}\}\cup\{\alpha\in W\mid
\INN{\alpha}{2H}\in\pi+2\pi\mathbb{Z}\},\\
W'&=\{\lambda\in\Sigma\mid \INN{\lambda}{2H}\in\pi+2\pi\mathbb{Z}\}\cup\{\alpha\in W\mid
\INN{\alpha}{2H}\in2\pi\mathbb{Z}\},
\end{align}
from which
$(\tilde{\Sigma},\Sigma,W)\sim(\tilde{\Sigma}',\Sigma',W')$ holds.
Thus, we have completed the proof.
\end{proof}

The second author (\cite{Ikawa2}) determined
$(\tilde{\Sigma},\Sigma,W)$
corresponding to $(G,\theta_{1},\theta_{2})$
such that $\sigma$ is of outer-type
by means of Vogan diagram for $(U,\sigma)$.
Thus, the converse of Proposition \ref{pro:simsim_sigma}
is true at the Lie algebra level.
Namely, we have the following corollary.

\begin{cor}
Under the same settings as in Proposition \ref{pro:simsim_sigma},
$(G,\theta_{1},\theta_{2})$
is isomorphic to $(G,\theta_{1}',\theta_{2}')$
with respect to $\sim$ at the Lie algebra level
if and only if
$(\tilde{\Sigma},\Sigma,W)\sim
(\tilde{\Sigma}',\Sigma',W')$.
Then, in the case when $G$ is simply-connected,
$(G,\theta_{1},\theta_{2})\sim(G,\theta_{1}',\theta_{2}')$
holds.
\end{cor}

\subsection{Vogan diagrams and double Satake diagrams for $\sigma$-action}\label{sec:validity_imgcpx_cptnoncpt}

Vogan diagram
provides us one of methods
to classify the noncompact semisimple
symmetric Lie algebras.
By using Cartan's duality,
we can classify compact symmetric pairs.
In this subsection,
we first observe the relation
between the Vogan diagram
for a compact symmetric pair $(U,\sigma)$
and the double Satake diagram
for the commutative compact symmetric triad $(G=U\times U,\theta_{1},\theta_{2})$
as in \eqref{eqn:2involution_sigma}.
Then we find the validity
of the terminologies
in Definitions \ref{dfn:imgroot_cpxroot}
and \ref{dfn:cptroot_noncptroot}.
Second, we reconstruct
the Vogan diagram
for $(U,\sigma)$
from 
the double Satake diagram
and the symmetric triad
for $(U\times U,\theta_{1},\theta_{2})$.
Then, we get an alternative method to determine
the symmetric triad corresponding to $(U\times U,\theta_{1},\theta_{2})$
by using the corresponding double Satake diagram.

\subsubsection{Vogan diagram for compact symmetric pair (Review)}\label{sec:vdi_review}

Although the original notion of Vogan
diagrams is defined for noncompact semisimple Lie groups, in order to clarify the relation to double Satake diagrams, our explanation of Vogan diagrams are given by compact symmetric pairs which are obtained from them via Cartan's duality.

Let $U$ be a compact connected semisimple Lie group with Lie algebra $\mathfrak{u}$
and $\sigma$ be an involution on $U$.
We write 
the canonical decomposition
of $\mathfrak{u}$ for $\sigma$
as $\mathfrak{u}=\mathfrak{k}_{\sigma}\oplus\mathfrak{m}_{\sigma}$.
Let $\bar{\mathfrak{a}}$
be a maximal abelian subalgebra of $\mathfrak{k}_{\sigma}$.
We set $V(\mathfrak{m}_{\sigma})=\{X\in\mathfrak{m}_{\sigma}\mid [\bar{\mathfrak{a}},X]=\{0\}\}$.
It is known that
$\bar{\mathfrak{t}}=\bar{\mathfrak{a}}\oplus V(\mathfrak{m}_{\sigma})$
is a maximal abelian subalgebra
of $\mathfrak{u}$
and $\bar{\mathfrak{t}}$ is $\sigma$-invariant (\cite[Proposition 6.60, Chapter VI, p.~386]{Knapp}).
We denote by $\bar{\Delta}$
the root system of $\mathfrak{u}$
with respect to $\bar{\mathfrak{t}}$.

\begin{dfn}[{\cite[p.~390]{Knapp}}]
A root $\alpha$ of $\bar{\Delta}$
is called an \textit{real root}
if $\INN{\alpha}{\bar{\mathfrak{a}}}=\{0\}$ and
an \textit{imaginary root}
if $\INN{\alpha}{V(\mathfrak{m}_{\sigma})}=\{0\}$.
Otherwise, $\alpha$
is called a \textit{complex root}.
\end{dfn}

By the maximality of $\bar{\mathfrak{a}}$ in $\mathfrak{k}_{\sigma}$,
there exist no real roots
(\cite[Proposition 6.70, Chapter VI]{Knapp}).
We denote by $\bar{\Delta}_{\mathrm{im}}$
the set of all imaginary roots
of $\bar{\Delta}$,
and by
$\bar{\Delta}_{\mathrm{cpx}}$
the set of all complex roots
of $\bar{\Delta}$.
Then
$\bar{\Delta}=\bar{\Delta}_{\mathrm{im}}\cup\bar{\Delta}_{\mathrm{cpx}}$
is a disjoint union.
For each $\alpha\in\bar{\Delta}$,
we denote by
$\mathfrak{u}(\bar{\mathfrak{t}},\alpha)$
the root space of $\mathfrak{u}^{\mathbb{C}}$
associated with $\alpha$.

\begin{dfn}[{\cite[p.~390]{Knapp}}]
A root $\alpha$ of $\bar{\Delta}$
is called a \textit{compact root}
if $\mathfrak{u}(\bar{\mathfrak{t}},\alpha)$
is contained in $\mathfrak{k}_{\sigma}^{\mathbb{C}}$;
a \textit{noncompact root}
if $\mathfrak{u}(\bar{\mathfrak{t}},\alpha)$
is contained in $\mathfrak{m}_{\sigma}^{\mathbb{C}}$.
We denote by $\bar{\Delta}_{\mathrm{cpt}}$
the set of all compact root of $\bar{\Delta}$
and by $\bar{\Delta}_{\mathrm{noncpt}}$
the set of all noncompact 
roots of $\bar{\Delta}$.
\end{dfn}

Let $\alpha\in\bar{\Delta}$.
The subspace $\mathfrak{u}(\bar{\mathfrak{t}},\alpha)$
has complex one dimension.
If $\alpha$ is imaginary,
then
$\mathfrak{u}(\bar{\mathfrak{t}},\alpha)$
is $\sigma$-invariant.
Any imaginary root is
either compact or noncompact.

We denote by $>$ the lexicographic ordering on $\bar{\Delta}$
defined by an ordered basis
$\{X_{j}\}\cup\{Y_{k}\}$
of $\bar{\mathfrak{t}}$
such that $\{X_{j}\}$
and $\{Y_{k}\}$
are ordered bases of $\bar{\mathfrak{a}}$ and $V(\mathfrak{m}_{\sigma})$,
respectively.
Let $\bar{\Pi}$ denote its fundamental system of $\bar{\Delta}$
and $\bar{\Delta}^{+}$
denote the set of positive roots
of $\bar{\Delta}$ with respect to $>$.
Since there exist no real roots of $\bar{\Delta}$,
$\sigma(\bar{\Delta}^{+})=\bar{\Delta}^{+}$ holds.
Hence $\sigma$ induces a permutation on $\bar{\Pi}$.
In particular,
the simple roots fixed by $\sigma$
are imaginary roots,
and the two-cyclic simple roots
are complex roots.
The \textit{Vogan diagram} of $(U,\sigma)$
associated with $\bar{\Pi}$
is described as follows:
In the Dynkin diagram
of $\bar{\Pi}$,
two complex roots
$\alpha,\alpha'$
in $\bar{\Pi}$ with $\alpha\neq\alpha'$ are connected by a curved arrow if $\sigma(\alpha)=\alpha'$, and any noncompact root is replaced from a white circle to a black circle.

\subsubsection{The validity
of the terminologies
in Definitions \ref{dfn:imgroot_cpxroot}
and \ref{dfn:cptroot_noncptroot}}

Let $(U,\sigma)$ be a compact symmetric pair
and
$(G=U\times U,\theta_{1},\theta_{2})$
be the commutative compact symmetric triad
as in \eqref{eqn:2involution_sigma}.
We give a maximal abelian subalgebra
of $\mathfrak{g}$ as in Lemma \ref{lem:mas_fundamental}.
Let $\mathfrak{a}$ be the maximal abelian subspace
of $\mathfrak{m}_{1}\cap\mathfrak{m}_{2}$
as in \eqref{eqn:am1m2_sigma}.
Since $V(\mathfrak{m}_{\sigma})=\{X\in\mathfrak{m}_{\sigma}\mid [\bar{\mathfrak{a}},X]=\{0\}\}$ is abelian,
\begin{equation}
\mathfrak{a}_{1}=\mathfrak{a}\oplus
\{(H,-H)\mid H\in V(\mathfrak{m}_{\sigma})\},\quad
\mathfrak{a}_{2}=\mathfrak{a}\oplus
\{(H,H)\mid H\in V(\mathfrak{m}_{\sigma})\}
\end{equation}
give maximal abelian subspaces of $\mathfrak{m}_{1}$
and $\mathfrak{m}_{2}$
containing $\mathfrak{a}$, respectively.
We note that $[\mathfrak{a}_{1},\mathfrak{a}_{2}]=\{0\}$ holds.
Here, we have the following expressions
of $\mathfrak{a}_{i}$ ($i=1,2$):
\begin{equation}
\mathfrak{a}_{1}=\{(H,-H)\mid H\in\bar{\mathfrak{t}}\},
\quad
\mathfrak{a}_{2}=\{(H,-\sigma(H))\mid H\in\bar{\mathfrak{t}}\},
\end{equation}
where we set $\bar{\mathfrak{t}}=\bar{\mathfrak{a}}\oplus V(\mathfrak{m}_{\sigma})$.
If we set $\mathfrak{b}=\{(H,H)\mid H\in\bar{\mathfrak{a}}\}(\subset\mathfrak{k}_{1}\cap\mathfrak{k}_{2})$,
then
\begin{equation}
\mathfrak{t}
=\mathfrak{a}
\oplus\{(H,H)\mid H\in V(\mathfrak{m}_{\sigma})\}
\oplus\{(H,-H)\mid H\in V(\mathfrak{m}_{\sigma})\}
\oplus\mathfrak{b}
=\bar{\mathfrak{t}}\oplus\bar{\mathfrak{t}}
\end{equation}
gives a maximal abelian subalgebra of $\mathfrak{g}$
containing $\mathfrak{a}_{1}$ and $\mathfrak{a}_{2}$.
Under the above settings,
$\mathfrak{t}$ satisfies
the conditions stated in Lemma \ref{lem:mas_fundamental},
so that we get
the root system $\Delta$
of $\mathfrak{g}$ with respect to $\mathfrak{t}$
and involutive automorphism $\sigma_{i}=-d\theta_{i}|_{\mathfrak{t}}$
on $\Delta$.

We give descriptions of $\Delta$ and $\sigma_{i}$
by means of the root system $\bar{\Delta}$ of $\mathfrak{u}$
with respect to $\bar{\mathfrak{t}}$.

\begin{lem}
Under the above settings, we have
\begin{equation}
\Delta=\{(\alpha,0),(0,\alpha)\mid \alpha\in\bar{\Delta}\}.
\end{equation}
Furthermore, for each $\alpha\in\bar{\Delta}$, we get
\begin{equation}
\sigma_{1}(\alpha,0)=(0,-\alpha),~
\sigma_{1}(0,\alpha)=(-\alpha,0),~
\sigma_{2}(\alpha,0)=(0,-\sigma(\alpha)),~
\sigma_{2}(0,\alpha)=(-\sigma(\alpha),0). 
\end{equation}
\end{lem}

Then we have the following descriptions
for $\Delta_{0}$,
$\Delta_{\mathrm{im}}$,
$\Delta_{\mathrm{cpx}}$,
$\Delta_{\mathrm{cpt}}$
and $\Delta_{\mathrm{noncpt}}$.

\begin{lem}\label{lem:0_im_im_cpx_cpx}
Under the above settings,
we have:
\begin{enumerate}[(1)]
\item $\Delta_{0}=\emptyset$.
\item $\Delta_{\mathrm{im}}=\{(\alpha,0),(0,\alpha)\mid
\alpha\in\bar{\Delta}_{\mathrm{im}}\}$.
\item $\Delta_{\mathrm{cpx}}=\{(\alpha,0),(0,\alpha)\mid
\alpha\in\bar{\Delta}_{\mathrm{cpx}}\}$.
\end{enumerate}
\end{lem}

\begin{proof}
(1) This follows straightforward
by the nonexistence of real roots in $\bar{\Delta}$.
(2) Let $\alpha\in\bar{\Delta}$. Since we have
$\sigma_{1}\sigma_{2}(\alpha,0)=(\sigma(\alpha),0)$,
the root $(\alpha,0)$
of $\Delta$ is imaginary
if and only if so is $\alpha$ as an element in $\bar{\Delta}$.
The statement (3) is verified in a similar manner.
\end{proof}

\begin{lem}
We have:
\begin{enumerate}[(1)]
\item $\Delta_{\mathrm{cpt}}=\{(\alpha,0),(0,\alpha)\mid
\alpha\in\bar{\Delta}_{\mathrm{cpt}}\}$.
\item $\Delta_{\mathrm{noncpt}}=\{(\alpha,0),(0,\alpha)\mid
\alpha\in\bar{\Delta}_{\mathrm{noncpt}}\}$.
\end{enumerate}
\end{lem}

\begin{proof}
A direct calculation shows that
\begin{equation}
(\mathfrak{g}^{\theta_{1}\theta_{2}})^{\mathbb{C}}
=\{(X,Y)\mid X,Y\in\mathfrak{k}_{\sigma}^{\mathbb{C}}\},\quad
(\mathfrak{g}^{-\theta_{1}\theta_{2}})^{\mathbb{C}}
=\{(X,Y)\mid X,Y\in\mathfrak{m}_{\sigma}^{\mathbb{C}}\},
\end{equation}
and that, for $\alpha\in\bar{\Delta}$,
\begin{equation}
\mathfrak{g}(\mathfrak{t},(\alpha,0))
=\{(X,0)\mid X\in\mathfrak{u}(\bar{\mathfrak{t}},\alpha)\},\quad
\mathfrak{g}(\mathfrak{t},(0,\alpha))
=\{(0,X)\mid X\in\mathfrak{u}(\bar{\mathfrak{t}},\alpha)\}.
\end{equation}
It is straightforward to verify
the assertion from the above expressions.
\end{proof}

We note that
the validity
of the terminologies
in Definitions \ref{dfn:imgroot_cpxroot}
and \ref{dfn:cptroot_noncptroot}
is given by the above two lemmas.
From the Vogan diagram

\subsubsection{Reconstruction of Vogan diagram for $\sigma$-action}

We observe the relation between
the Vogan diagram for a compact symmetric pair $(U,\sigma)$
and the double Satake diagram for the commutative
compact symmetric triad $(G=U\times U,\theta_{1},\theta_{2})$
as in \eqref{eqn:2involution_sigma}.

Let $\Delta$ be the root system of $\mathfrak{g}$
with respect to $\mathfrak{t}=\bar{\mathfrak{t}}\oplus\bar{\mathfrak{t}}$.
We put $\sigma_{i}=-d\theta_{i}|_{\mathfrak{t}}$ 
($i=1,2$).
We first give a $(\sigma_{1},\sigma_{2})$-fundamental system
of $\Delta$ as follows:
Let $\{X_{j}\}$
and $\{Y_{k}\}$
be bases of $\bar{\mathfrak{a}}$
and $V(\mathfrak{m}_{\sigma})$,
respectively.
We denote by $>$ the lexicographic ordering on $\Delta$
defined by an ordered basis
$\mathcal{X}\cup
\mathcal{Y}\cup
\mathcal{Z}\cup
\mathcal{W}$
defined as follows:
\begin{equation}
\mathcal{X}
=\{(X_{j},-X_{j})\},\quad
\mathcal{Y}
=\{(Y_{k},Y_{k})\},\quad
\mathcal{Z}
=\{(Y_{k},-Y_{k})\},\quad
\mathcal{W}
=\{(X_{j},X_{j})\}.
\end{equation}
Then, it is shown that
the fundamental system
$\Pi$ of $\Delta$
for $>$
is a $(\sigma_{1},\sigma_{2})$-fundamental system.
We denote by $\Delta^{+}$
the corresponding set of positive roots of $\Delta$.

\begin{lem}
Let $\bar{\Pi},\bar{\Delta}^{+}$
as in Section \ref{sec:vdi_review}.
Then we have:
\begin{enumerate}[(1)]
\item $\Delta^{+}=
\{(\alpha,0),(0,-\alpha)\mid \alpha\in\bar{\Delta}^{+}\}$.
\item $\Pi=
\{(\alpha,0),(0,-\alpha)\mid \alpha\in\bar{\Pi}\}$.
\end{enumerate}
\end{lem}

We write the double Satake diagram
corresponding $(G,\theta_{1},\theta_{2})$
associated with $\Pi$
as
\begin{equation}
(S_{1},S_{2})
=(S(\Pi,\Pi_{1,0},p_{1}),S(\Pi,\Pi_{2,0},p_{2})).
\end{equation}
Then we have
$\INN{(\alpha,0)}{\mathfrak{a}_{i}}\neq\{0\}$
and $\INN{(0,-\alpha)}{\mathfrak{a}_{i}}\neq\{0\}$
for all $\alpha\in\bar{\Pi}$,
from which
$\Pi_{i,0}=\emptyset$ holds,
namely,
there exist no black circles
in the Satake diagram $S_{i}$.
Then,
the Satake involution $p_{i}$
gives a permutation on $\Pi$,
which is expressed as follows:
\begin{equation}\label{eqn:sigma_arrows}
\begin{cases}
p_{1}(\alpha,0)=(0,-\alpha),&
p_{1}(0,-\alpha,0)=(\alpha,0),\\
p_{2}(\alpha,0)=(0,-\sigma(\alpha)),&
p_{2}(0,-\alpha,0)=(\sigma(\alpha),0).
\end{cases}
\end{equation}
It is known that $\sigma$ gives a permutation of $\bar{\Pi}$ (cf.~\cite[p.~397]{Knapp}).
This permutation can be read off from the Vogan diagram of $(U,\sigma)$ associated with $\bar{\Pi}$,
from which the Satake involutions $p_{1}, p_{2}$ are determined.
Hence, we obtain 
the double Satake diagram of $(G,\theta_{1},\theta_{2})$ associated with $\Pi$.

We can determine
whether $\sigma$
is of inner-type or not
by means of the double Satake diagram
of $(G,\theta_{1},\theta_{2})$.

\begin{lem}
$\sigma$
is of inner-type
if and only if $(S_{1},S_{2})\sim
(S_{1},S_{1})$.
\end{lem}

\begin{proof}
It is known that
$\sigma$ is of inner-type
if and only if
the Vogan diagram
corresponding to $(U,\sigma)$
has no arrows,
that is, $\bar{\Pi}\subset\bar{\Delta}_{\mathrm{im}}$.
Then we have $\sigma(\alpha)=\alpha$
for all $\alpha\in\bar{\Pi}$.
Hence $p_{1}=p_{2}$ holds,
from which $(S_{1},S_{2})
\sim (S_{1},S_{1})$.
Conversely,
we suppose that
$(S_{1},S_{2})\sim(S_{1},S_{1})$ holds.
Then there exists an automorphism $\psi$
on $\Pi$ satisfying
$\psi\circ p_{2}\circ \psi^{-1}=p_{1}$
and $\psi\circ p_{1}\circ \psi^{-1}=p_{1}$.
Then we have $p_{1}=p_{2}$,
so that we have $\Pi\subset\Delta_{\mathrm{im}}$.
Hence we have $\bar{\Pi}\subset\bar{\Delta}_{\mathrm{im}}$, that is,
$\sigma$ is of inner-type.
\end{proof}

Now, we explain
our method to reconstruct
the Vogan diagram
for $\bar{\Pi}$
from
$(S_{1},S_{2})$
and $(\tilde{\Sigma},\Sigma,W)$.
From \eqref{eqn:sigma_arrows},
the permutation
of $\sigma$ on $\bar{\Pi}$
is reconstructed
from the difference between
$p_{1}$ and $p_{2}$.
Then we have $\bar{\Delta}_{\mathrm{im}}$
and $\bar{\Delta}_{\mathrm{cpx}}$.
In order to determine
$\bar{\Delta}_{\mathrm{cpx}}$
and $\bar{\Delta}_{\mathrm{noncpt}}$,
we need to characterize $(G,\theta_{1},\theta_{2})$
up to $\equiv$,
which is given by
the double Satake diagram
$(S_{1},S_{2})$
and the corresponding symmetric triad
$(\tilde{\Sigma},\Sigma,W)$ as explained in the
following example.

\begin{ex}
Let us consider the case when $U=E_{6}$.
We determine
the Vogan diagram of $(E_{6},\sigma)$ from
$(S_{1},S_{2})$
and $(\tilde{\Sigma},\Sigma,W)$
corresponding to the commutative compact symmetric triad
$(E_{6}\times E_{6},\theta_{1},\theta_{2})$
as in \eqref{eqn:2involution_sigma}.
We first consider the case when
$(S_{1},S_{2}) \not\sim (S_{1},S_{1})$.
Table \ref{table:E6_outer}
shows the corresponding double
Satake diagram.
Here, we write $(\alpha_{i},0)$
as $\alpha_{i}'$
and $(0,-\alpha_{i})$
as $\alpha_{i}''$
for $i=1,\dotsc,6$ in this table.
Indeed, the difference between $p_{1}$
and $p_{2}$ is found in $\{\alpha_{1},\alpha_{6}\}$ and $\{\alpha_{3},\alpha_{5}\}$.
Hence the permutation on $\bar{\Pi}=\{\alpha_{1},\dotsc,\alpha_{6}\}$
induced from $\sigma$
is reconstructed as follows:
\[
\sigma:\bar{\Pi}\to\bar{\Pi};\,
\begin{cases}
\alpha_{1}\mapsto \alpha_{6},\quad
\alpha_{2}\mapsto \alpha_{2},\quad
\alpha_{3}\mapsto \alpha_{5},\\
\alpha_{4}\mapsto \alpha_{4},\quad
\alpha_{5}\mapsto \alpha_{3},\quad
\alpha_{6}\mapsto \alpha_{1}.
\end{cases}
\]
In particular, we find $\mathrm{rank}(K_{\sigma})=4$.
By using Propositions \ref{pro:cpt+cpt_cpt_etc}
and \ref{pro:delta_inductive2},
we find all elements of $\Delta_{\mathrm{cpt}}$
and $\Delta_{\mathrm{noncpt}}$
whether $\alpha_{2}$ and $\alpha_{4}$
are compact or noncompact.
By the first half of Proposition \ref{pro:symm_determ_base},
we have
$(\tilde{\Sigma},\Sigma,W)=(\mbox{I-}F_{4})=(F_{4},F_{4},W)$
in the case when $\alpha_{2}, \alpha_{4}$ are compact;
and
$(\tilde{\Sigma},\Sigma,W)=(\mbox{I'-}F_{4})=(F_{4},C_{4},W)$
otherwise.
Then,
the corresponding Vogan diagram
of $(U,\sigma)$
is 
that for $(E_{6}, F_{4})$
in the case when $(\tilde{\Sigma},\Sigma,W)=(\mbox{I-}F_{4})$;
that for $(E_{6},Sp(4))$ in the case when $(\tilde{\Sigma},\Sigma,W)=(\mbox{I'-}F_{4})$.

Second, we consider the case when $(S_{1},S_{2})\sim (S_{1},S_{1})$.
In particular, $\sigma$ acts trivially on $\bar{\Pi}$.
Then we get $\Delta=\Delta_{\mathrm{im}}$.
We also have $\mathrm{rank}(K_{\sigma})=6$.
A similar argument shows that
the corresponding Vogan diagram
is that for $(E_{6},SU(6)\cdot SU(2))$
and for $(E_{6},SO(10)\cdot U(1))$
in the case when $(\tilde{\Sigma},\Sigma,W)$
is $(\mbox{IV-}E_{6})=(E_{6},A_{1}\cup A_{5},W)$
and $(\mbox{IV'-}E_{6})=(E_{6},D_{5},W)$,
respectively.
Here, we except for the case when
$(\tilde{\Sigma},\Sigma,W)=(E_{6},E_{6},\emptyset)$,
since the corresponding involution $\sigma$
becomes the identity transformation on $E_{6}$.
\end{ex}

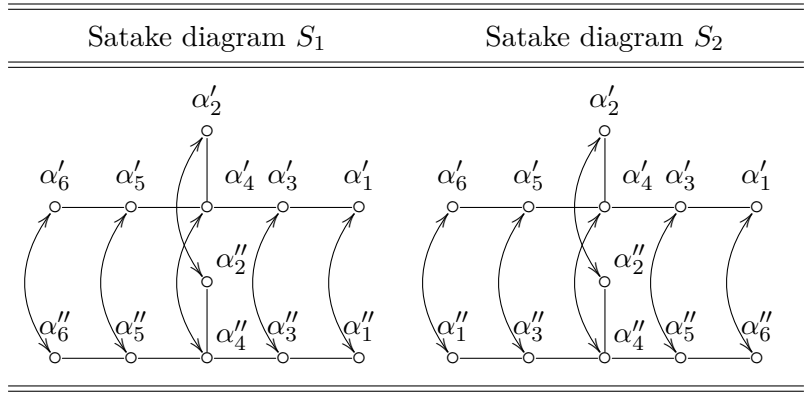
\begin{table}[!!ht]
\renewcommand{\arraystretch}{1.5}
\centering
\caption{Double Satake diagram
$(S_{1},S_{2})$ of $(E_{6}\times E_{6},\theta_{1},\theta_{2})$
with $(S_{1},S_{2})\not\sim(S_{1},S_{1})$
}\label{table:E6_outer}
\begin{tabular}{cc}
\hline
\hline
Satake diagram $S_{1}$ & Satake diagram $S_{2}$ \\
\hline
\hline
\begin{xy}
\ar@{-}(-20,0)*++!D{\alpha_{6}'}*{\circ}="a6";(-10,0)*++!D{\alpha_{5}'}*{\circ}="a5"
\ar@{-}"a5";(0,0)*++!DL{\alpha_{4}'}*{\circ}="a4"
\ar@{-}"a4";(10,0)*++!D{\alpha_{3}'}*{\circ}="a3"
\ar@{-}"a3";(20,0)*++!D{\alpha_{1}'}*{\circ}="a1"
\ar@{-}"a4";(0,10)*++!D{\alpha_{2}'}*{\circ}="a2"
%%%
\ar@{-}(-20,-20)*++!D{\alpha_{6}''}*{\circ}="a6'";(-10,-20)*++!D{\alpha_{5}''}*{\circ}="a5'"
\ar@{-}"a5'";(0,-20)*+!DL{\alpha_{4}''}*{\circ}="a4'"
\ar@{-}"a4'";(10,-20)*++!D{\alpha_{3}''}*{\circ}="a3'"
\ar@{-}"a3'";(20,-20)*++!D{\alpha_{1}''}*{\circ}="a1'"
\ar@{-}"a4'";(0,-10)*+!DL{\alpha_{2}''}*{\circ}="a2'"
%%% arrows
\ar@/_4mm/@{<->} "a1";"a1'"
\ar@/_4mm/@{<->} "a2";"a2'"
\ar@/_4mm/@{<->} "a3";"a3'"
\ar@/_4mm/@{<->} "a4";"a4'"
\ar@/_4mm/@{<->} "a5";"a5'"
\ar@/_4mm/@{<->} "a6";"a6'"
\end{xy}
&
\begin{xy}
\ar@{-}(-20,0)*++!D{\alpha_{6}'}*{\circ}="a6";(-10,0)*++!D{\alpha_{5}'}*{\circ}="a5"
\ar@{-}"a5";(0,0)*++!DL{\alpha_{4}'}*{\circ}="a4"
\ar@{-}"a4";(10,0)*++!D{\alpha_{3}'}*{\circ}="a3"
\ar@{-}"a3";(20,0)*++!D{\alpha_{1}'}*{\circ}="a1"
\ar@{-}"a4";(0,10)*++!D{\alpha_{2}'}*{\circ}="a2"
%%%
\ar@{-}(-20,-20)*++!D{\alpha_{1}''}*{\circ}="a6'";(-10,-20)*++!D{\alpha_{3}''}*{\circ}="a5'"
\ar@{-}"a5'";(0,-20)*+!DL{\alpha_{4}''}*{\circ}="a4'"
\ar@{-}"a4'";(10,-20)*++!D{\alpha_{5}''}*{\circ}="a3'"
\ar@{-}"a3'";(20,-20)*++!D{\alpha_{6}''}*{\circ}="a1'"
\ar@{-}"a4'";(0,-10)*+!DL{\alpha_{2}''}*{\circ}="a2'"
%%% arrows
\ar@/_4mm/@{<->} "a1";"a1'"
\ar@/_4mm/@{<->} "a2";"a2'"
\ar@/_4mm/@{<->} "a3";"a3'"
\ar@/_4mm/@{<->} "a4";"a4'"
\ar@/_4mm/@{<->} "a5";"a5'"
\ar@/_4mm/@{<->} "a6";"a6'"
\end{xy}
\\[13ex]
\hline
\hline
\end{tabular}
\end{table}

%%%%%%%%%%%%%%%%%%%%%%%%%%%%%%%%%%%%%%%%%%%%%%%%

\subsubsection*{Future directions}
The motivation of symmetric triads with multiplicities
comes from the study of Hermann actions.
In the present paper, we have developed
the theory of symmetric triads with multiplicities
corresponding to compact symmetric triads
$(G,\theta_{1},\theta_{2})$ with $\theta_{1}\theta_{2}=\theta_{2}\theta_{1}$.
Then the following questions arises naturally:
(1) Is there a similar theory for abstract symmetric triads with multiplicities
corresponding to compact symmetric triads $(G,\theta_{1},\theta_{2})$
with $\theta_{1}\theta_{2}\neq\theta_{2}\theta_{1}$?
(2) As its applications, 
study the geometry of Hermann actions corresponding to
noncommutative compact symmetric triads.
Recently, we find
the study of
the weak reflectivity
in the sense of \cite{IST}
for orbits of Hermann actions
corresponding to noncommutative compact
symmetric triads due to Ohno \cite{Ohno23}.
(3) Classify
noncommutative compact symmetric triads
with respect to $\equiv$.

\medskip

\subsubsection*{Errata}
\textit{A note on symmetric triad and Hermann action, by O.~Ikawa,
Proceedings of the workshop on differential geometry of submanifolds
and its related topics, Saga, August 4--6 (2012),
220--229.
The following list should be added to the table on p.~228:}
\begin{center}
\begin{tabular}{c|c}
(II-$BC_{r}$) & $(SO(4r+2),U(2r+1),S(O(2r+1)\times O(2r+1)))$ \\
(III-$BC_{r}$) & $(SU(2(2r+1)),S(U(2r+1)\times U(2r+1)),Sp(2r+1))$
\end{tabular}
\end{center}

\end{document}